\documentclass{amsart}%
\usepackage{amssymb}
\usepackage{amsmath}
\usepackage{amsfonts}
\usepackage[all,cmtip]{xy}
\usepackage[margin=1.5in]{geometry}
\usepackage{hyperref}
\usepackage{graphicx}%
\newtheorem{theorem}{Theorem}
\newtheorem{theoremannounce}{Theorem}
\numberwithin{theorem}{section}
\theoremstyle{plain}
\newtheorem*{acknowledgement}{Acknowledgement}

\newtheorem{corollary}[theorem]{Corollary}

\newtheorem{definition}[theorem]{Definition}

\newtheorem{lemma}[theorem]{Lemma}

\newtheorem{convention}{Convention}

\newtheorem{proposition}[theorem]{Proposition}
\theoremstyle{remark}
\newtheorem{conjecture}{Conjecture}
\newtheorem{example}[theorem]{Example}
\newtheorem{elaboration}{Elaboration}

\newtheorem{remark}[theorem]{Remark}

\numberwithin{equation}{section}

\begin{document}
\title[Equivariant Tamagawa numbers]{An alternative construction of equivariant\linebreak Tamagawa numbers}
\author{Oliver Braunling}
\address{Mathematical Institute, University of Freiburg, Ernst-Zermelo-Strasse 1, 79104
Freiburg im Breisgau, Germany}
\thanks{The author was supported by DFG GK1821 \textquotedblleft Cohomological Methods
in Geometry\textquotedblright\ and a Junior Fellowship at the Freiburg
Institute for Advanced Studies (FRIAS)}

\begin{abstract}
We propose a new formulation of the equivariant Tamagawa number conjecture
(ETNC) for non-commutative coefficients. We remove Picard groupoids,
determinant functors, virtual objects and relative $K$-groups. Our Tamagawa
numbers lie in an id\`{e}le group instead of any kind of $K$-group. Our
formulation is proven equivalent to the one of Burns--Flach.

\end{abstract}
\maketitle

In this paper we give a new approach to constructing the equivariant Tamagawa
numbers of Burns--Flach \cite{MR1884523}. We do \textit{not} wish to claim
that this method is in any way better or worse than their original one. We
merely have a different kind of perspective, focussing on ad\`{e}les and local
compactness, and we just want to set up Tamagawa numbers in the way which
appears most natural from this slightly different angle.\medskip

As in our previous paper \cite{etnclca} we restrict to regular orders
$\mathfrak{A}\subset A$ in a semisimple algebra $A$. Given the current state
of our foundations, this presently cannot be avoided. However, we will remove
this assumption in a future paper and our formulation will remain intact
almost verbatim.\medskip

Let us explain our approach: Suppose $A$ is a finite-dimensional semisimple
$\mathbb{Q}$-algebra and $\mathfrak{A}\subset A$ a regular order, e.g. a
hereditary or maximal one. Usually (following \cite{MR1884523}) the
equivariant Tamagawa number is an element $T\Omega$ in $K_{0}(\mathfrak{A}%
,\mathbb{R})$, a relative $K$-group, whose elements have the shape%
\[
\lbrack P,\varphi,Q]
\]
in the so-called Swan presentation. Already here, we shall proceed a little
differently. The relative $K$-group sits in an exact sequence%
\begin{equation}
\cdots\longrightarrow K_{1}(\mathfrak{A})\longrightarrow K_{1}(A_{\mathbb{R}%
})\longrightarrow K_{0}(\mathfrak{A},\mathbb{R})\overset{\operatorname*{cl}%
}{\longrightarrow}\operatorname*{Cl}(\mathfrak{A})\longrightarrow0\text{,}
\label{l_Rado_1}%
\end{equation}
where $\operatorname*{Cl}(\mathfrak{A})$ is the locally free class group. In
\cite{MR0376619}\ Fr\"{o}hlich has proven a formula for the latter, namely%
\[
\operatorname*{Cl}(\mathfrak{A})\cong\frac{J(A)}{J^{1}(A)\cdot A^{\times}\cdot
U\mathfrak{A}\cdot A_{\mathbb{R}}^{\times}}\text{,}%
\]
where $J(A)$ denotes the non-commutative id\`{e}les of $A$, $J^{1}(A)$ the
reduced norm one id\`{e}les, and $U\mathfrak{A}$ the unit finite id\`{e}les of
the order. Our first idea is to extend this formula of Fr\"{o}hlich to
$K_{0}(\mathfrak{A},\mathbb{R})$:

\begin{theoremannounce}
There is a canonical isomorphism%
\begin{equation}
K_{0}(\mathfrak{A},\mathbb{R})\cong\frac{J(A)}{J^{1}(A)\cdot A^{\times}\cdot
U\mathfrak{A}}\text{,} \label{l_Rado_3}%
\end{equation}
and under this identification the map `$\operatorname*{cl}$' in Equation
\ref{l_Rado_1} amounts to quotienting out $A_{\mathbb{R}}^{\times}$ in the
infinite places of $J(A)$.\footnote{This is not really a new result. Agboola
and Burns have given a Hom-description formulation of a much more general
result in \cite{MR2192383}.}
\end{theoremannounce}

So, in our picture, we will most naturally view an equivariant Tamagawa number
as an element of the id\`{e}le-style group on the right. Next, from a previous
article, we already know that $K_{0}(\mathfrak{A},\mathbb{R})\cong
K_{1}(\mathsf{LCA}_{\mathfrak{A}})$, where $\mathsf{LCA}_{\mathfrak{A}}$
denotes the category of locally compact topological $\mathfrak{A}$-modules. We
refer to \cite{etnclca} for both the philosophy why this should hold (keyword:
equivariant Haar measures), as well as an actual proof.

Now suppose we are given a pure motive $M$ with an action by the semisimple
algebra $A$. As in Burns--Flach \cite{MR1884523} some further assumptions and
data, like lattices $T_{v}$, are needed\footnote{a projective $\mathfrak{A}%
$-structure plus the Coherence Hypothesis, \cite[\S 3.3]{MR1884523}}. It would
unreasonably inflate this introduction to carefully go through the rather
involved setup, so we shall assume that the reader is familiar with
\cite[\S 2-3]{MR1884523}. We use exactly the same notation. As in
Burns--Flach, we begin with the objects%
\[
R\Gamma_{c}(\mathcal{O}_{F,S_{p}},T_{p})\qquad\text{and}\qquad\Xi(M)
\]
of a certain nature, defined over $\mathfrak{A}_{p}$ and $A$. In
\cite{MR1884523} the former object is regarded as a bounded complex and the
latter, the \textit{fundamental line}, as an object in a certain Picard groupoid.

We drop this and shall look at them differently: We work in a special model of
algebraic $K$-theory due to \cite{MR1167575}. It provides us with a space such
that (a)\ bounded complexes define points in it, (b) generalized lines like
$\Xi(M)$ also define points in it, and (c), quasi-isomorphisms between bounded
complexes define paths between points.

Our Tamagawa numbers are now defined as follows: We prove that there is a
fibration of pointed spaces%
\begin{equation}
K(\widehat{\mathfrak{A}})\times K(A)\longrightarrow K(\widehat{A})\times
K(A_{\mathbb{R}})\longrightarrow K(\mathsf{LCA}_{\mathfrak{A}})\text{.}
\label{l_Rado_2}%
\end{equation}
The reader should really think of this in the sense of topology. A base space
on the right, the total space in the middle, and on the left the fiber over
the base point. Because we use the aforementioned versatile model of
$K$-theory, $R\Gamma_{c}(\mathcal{O}_{F,S_{p}},T_{p})$ and $\Xi(M)$ simply
define points in the fiber, i.e. the leftmost space in Equation \ref{l_Rado_2}%
.
\begin{equation}%
{\includegraphics[
height=1.2142in,
width=4.7818in
]%
{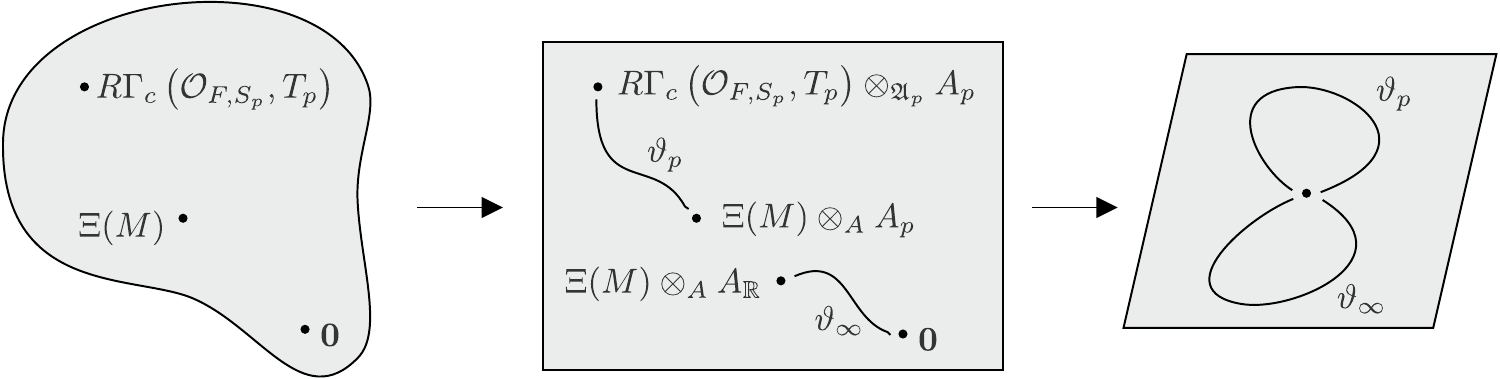}%
}
\label{l_Rado_4}%
\end{equation}
After base change from $\mathfrak{A}_{p}$ to $A_{p}$, and from $A$ to
$A_{\mathbb{R}}$, there exist comparison isomorphisms $\vartheta_{p}$ resp.
$\vartheta_{\infty}$ (this is exactly as in \cite[\S 3.4]{MR1884523}).
However, the underlying quasi-isomorphisms then just define paths between the
points in the middle term, i.e. the total space $K(\widehat{A})\times
K(A_{\mathbb{R}})$. But wait: The start point and end point objects of these
paths all came from the fiber, i.e. once we go all to the right to the base
space of the fibration, all these points collapse to the base point. However,
this means that the paths all get mapped to closed loops. So they define an
element in the fundamental group $\pi_{1}K(\mathsf{LCA}_{\mathfrak{A}}%
)=:K_{1}(\mathsf{LCA}_{\mathfrak{A}})$.

This is it. As in Burns--Flach, we call this element $R\Omega$, and after
adding the term coming from the $L$-function, $L\Omega$, \textbf{we have now
constructed our Tamagawa number}. Returning to Equation \ref{l_Rado_3}, we
may, if we want, get rid of $K$-groups and regard this as an element of%
\begin{equation}
\frac{J(A)}{J^{1}(A)\cdot A^{\times}\cdot U\mathfrak{A}}\text{,} \label{lsol1}%
\end{equation}
leading to a formulation in which not a single $K$-group is present anymore.
Or, at least not literally.\medskip

In the end, we have just obtained the same concept as Burns--Flach. Using
$K_{0}(\mathfrak{A},\mathbb{R})\cong K_{1}(\mathsf{LCA}_{\mathfrak{A}})$ we
prove the following.

\begin{theoremannounce}
\label{thm2_Intro}Our construction of the equivariant Tamagawa number
$T\Omega$ is equivalent to the one of Burns--Flach in \cite{MR1884523}.
\end{theoremannounce}

See \S \ref{sect_CompatProofSection}. The reader might worry that the
identification of our closed loops with id\`{e}les in Equation \ref{lsol1} is
something elusive. Not at all. The non-commutative id\`{e}les $J(A)$ act as
automorphisms on the non-commutative ad\`{e}les%
\[
A_{\mathbb{A}}:=\left\{  \left.  (x_{\mathfrak{p}})_{\mathfrak{p}}\in
\prod_{\mathfrak{p}}A_{\mathfrak{p}}\right\vert a_{\mathfrak{p}}%
\in\mathfrak{A}_{\mathfrak{p}}\text{ for all but finitely many places
}\mathfrak{p}\right\}  \text{,}%
\]
where $\mathfrak{p}$ runs through the finite and infinite places. As we had
already explained above, in our model of $K$-theory any isomorphism determines
a path. Hence, any $\alpha\in J(A)$ defines a loop%
\[%
{\includegraphics[
height=0.4987in,
width=0.894in
]%
{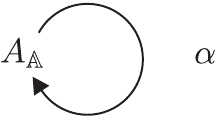}%
}
\]
and thus we get a map $\varphi:J(A)\rightarrow K_{1}(\mathsf{LCA}%
_{\mathfrak{A}})$. When phrasing Equation \ref{l_Rado_3} in terms of
$K_{1}(\mathsf{LCA}_{\mathfrak{A}})$ instead of $K_{0}(\mathfrak{A}%
,\mathbb{R})$, this map $\varphi$ is the one inducing the isomorphism. This
will be Theorem \ref{thm_GlobalLocal_Nenashev}.\medskip

In Figure \ref{l_Rado_4} we had equipped $R\Gamma_{c}(\mathcal{O}_{F,S_{p}%
},T_{p})\otimes A_{p}$ and $\Xi(M)\otimes A_{\mathbb{R}}$ with their natural
topologies, giving objects which carry $p$-adic and real topologies, just like
the ad\`{e}les $A_{\mathbb{A}}$. Our theorem says that the closed loops made
from $\vartheta_{p}$, $\vartheta_{\infty}$ on the right in Figure
\ref{l_Rado_4} are \textquotedblleft equivalent\textquotedblright\footnote{in
a complicated sense: firstly derived, i.e. up to quasi-isomorphisms, and
secondly $K$-theoretically, i.e. after transforming exact sequences into
alternating sums.} to a closed loop coming from an automorphism of the
non-commutative ad\`{e}les, and in fact one coming from plainly multiplying
with an id\`{e}le. This id\`{e}le\ (class) \textit{is} the Tamagawa
number.\medskip

\begin{center}
$-$ \textit{Some further results\medskip\ }$-\medskip$
\end{center}

We prove Equation \ref{l_Rado_3} in terms of the Swan presentation of
$K_{0}(\mathfrak{A},\mathbb{R})$, but we can also spell out an explicit map
from the id\`{e}le quotient to the Nenashev presentation of the $K$-group
$K_{1}(\mathsf{LCA}_{\mathfrak{A}})$. Running those isomorphisms back to back,
we obtain a new isomorphism%
\[
K_{0}(\mathfrak{A},\mathbb{R})\overset{\sim}{\longrightarrow}K_{1}%
(\mathsf{LCA}_{\mathfrak{A}})\text{.}%
\]
Hence, we now have \textit{three} such isomorphisms: the inexplicit one of the
first paper \cite{etnclca}, a rather enigmatic one with the special property
to be `universal in Swan generators' in \cite{etnclca2}, as well as the one of
this paper. \textit{We do not know} whether any two of them agree. Given how
reluctant the Nenashev presentation is towards Swan generators, we wish to
propose the following analogy:%
\begin{align*}
\text{ideal class group}  &  \leftrightarrow\text{id\`{e}le class group}\\
\text{Swan presentation}  &  \leftrightarrow\text{Nenashev presentation.}%
\end{align*}
Why the ideal class group allusion makes sense is surely clear from the map
`$\operatorname*{cl}$' in Equation \ref{l_Rado_1}: It sends $[P,\varphi,Q]$ to
$[P]-[Q]$. To justify the right side, consider the explicit formulas for the
maps in Theorem \ref{thm_GlobalLocal_Swan} and Theorem
\ref{thm_GlobalLocal_Nenashev}. Both essentially rely on Fr\"{o}hlich's
id\`{e}le classification of projective $\mathfrak{A}$-modules,
\cite{MR0376619}.

\begin{acknowledgement}
We heartily thank B. Chow, B. Drew, A. Huber, M. Wendt for discussions and
help. The delicate role of signs in the proof of Theorem
\ref{thm_PrincipalIdeleFibration} only became clear after some very valuable
remarks by Brad Drew. Most of this project was carried out at FRIAS and I
heartily thank them for providing perfect working conditions. I still cannot
imagine a better place for inspiration and creativity.
\end{acknowledgement}

\section{\label{sect_Overview}Detailed Overview}

\label{marker_OverviewTNCBegin}We first recall some basics of $K$-theory, but
rather differently from the material of many surveys. We allow ourselves some
imprecisions, for pedagogical reasons, and provide rigorous details only later
in \S \ref{sect_GettingPrecise}.\medskip

Suppose $\mathsf{C}$ is an exact category, e.g., finitely generated projective
modules over a ring, for which we write $\operatorname*{PMod}(R)$, or any
abelian category. In many situations people are only interested in the
$K$-groups $K_{i}(\mathsf{C})$ themselves. However, being all honest,
$K$-theory is a pointed space $K(\mathsf{C})$ and then the $K$-groups arise as
its homotopy groups $K_{i}(\mathsf{C}):=\pi_{i}K(\mathsf{C})$. The space
$K(\mathsf{C})$ is practically never the kind of space one could draw on a
sheet of paper. And really, what kind of space it is, depends on our concrete
approach to $K$-theory, e.g. Quillen's $Q$-construction, Waldhausen's
$S$-construction, etc. These spaces are all distinct, but have the same
homotopy type. What we describe in this section is the Gillet--Grayson model
in a version due to \cite{MR1167575}. A textbook explanation of the basic
method can be found in Weibel's book \cite[Chapter IV, \S 9]{MR3076731}. This
space has a number of properties making it a lot more convenient than other
spaces giving $K$-theory.

Before giving a precise description, let us just summarize the most important
principles:\newline$\left.  \qquad\text{\textbf{(a)}}\right.  $ Every object
in $\mathsf{C}$ determines a point in the space $K(\mathsf{C})$.\newline%
$\left.  \qquad\text{\textbf{(b)}}\right.  $ Every isomorphism $X\overset
{\sim}{\rightarrow}Y$ determines a path from the point of $X$ to $Y$ in
$K(\mathsf{C})$.\newline Let us quickly connect this with $K$-groups: The
zero-th $K$-group $K_{0}(\mathsf{C})=\pi_{0}K(\mathsf{C})$ corresponds to the
connected components of the space. Usually one writes $[X]$ for the $K_{0}%
$-class determined by an object $X\in\mathsf{C}$. And indeed, if two objects
$X,Y$ are isomorphic, say $\varphi:X\overset{\sim}{\rightarrow}Y$, then there
is a path between them by principle (b), so they lie in the same connected
component and correspondingly $[X]=[Y]$ in $K_{0}(\mathsf{C})$.%
\[%
{\includegraphics[
height=0.8294in,
width=3.1548in
]%
{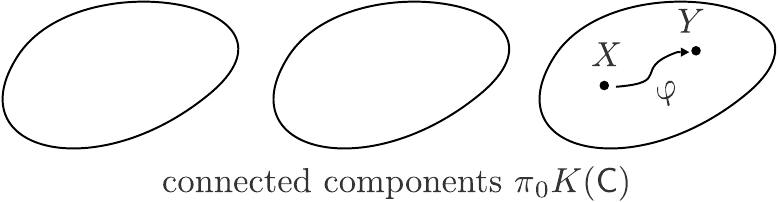}%
}
\]
Further, if $\mathsf{C}=\operatorname*{PMod}(R)$, then it is well-known that%
\begin{equation}
K_{1}(R)=\operatorname*{GL}(R)/[\operatorname*{GL}(R),\operatorname*{GL}%
(R)]\text{,} \label{lww1}%
\end{equation}
the abelianization of $\operatorname*{GL}(R)$. And indeed, given any element
$\alpha\in\operatorname*{GL}(R)$, it determines an automorphism $R^{n}%
\overset{\sim}{\rightarrow}R^{n}$ for some sufficiently large $n$, and by
principle (b) this determines a path from the point of $R^{n}$ to itself, i.e.
a closed loop, and thus an element in $\pi_{1}$. This leads us back to the
fact that $K_{1}$, beyond the description in Equation \ref{lww1}, is also the
fundamental group $\pi_{1}K(\mathsf{C})$.

At this point, the reader might wish for a more precise formulation of (a) and
(b) and a concrete rigorous justification of these principles. We will do
this, but only later, see \S \ref{sect_GettingPrecise}.

There is also an addition operation on $K(\mathsf{C})$ and a negation map:%
\begin{equation}
\left.  +\right.  :K(\mathsf{C})\times K(\mathsf{C})\longrightarrow
K(\mathsf{C})\qquad\text{and}\qquad\left.  -\right.  :K(\mathsf{C}%
)\longrightarrow K(\mathsf{C})\text{.} \label{lww_AddSubtract}%
\end{equation}
If $X,Y$ are objects, the point of the direct sum $X\oplus Y$ is the sum of
the points of $X$ and $Y$ under this map \textquotedblleft$+$%
\textquotedblright. We write $-X$ for the negation of the point of $X$. While
this exists as a point in $K(\mathsf{C})$, there is usually no object
producing this point under principle (a).

Care is needed here: The above maps do \textit{not} give the space
$K(\mathsf{C})$ a group structure. The problem is that while for example
$X\oplus Y$ and $Y\oplus X$ are canonically isomorphic, according to principle
(b) this canonical isomorphism merely defines a path from the point of
$X\oplus Y$ to the point of $Y\oplus X$, but they will usually be different
points. Similarly, the associativity isomorphism $(X\oplus Y)\oplus
Z\overset{\sim}{\rightarrow}X\oplus(Y\oplus Z)$ only yields a path between the
corresponding points. Once going to the homotopy groups $\pi_{i}K(\mathsf{C})$
these problems all disappear and the above two maps induce honest group
structures. For example, $\pi_{0}K(\mathsf{C})$ only sees connected
components, and since the above remarks mean that paths exist between these
points, they lie in the same component and thus $\pi_{0}K(\mathsf{C})$ gets an
honest group structure.

In general the weaker type of structure given by the maps in Equation
\ref{lww_AddSubtract} is sometimes called a `homotopy commutative and homotopy
associative $H$-space'\footnote{The book \cite{MR0270372} might be a bit out
of date, but it carefully develops and explores such structures along a lot of
examples.}. All these phenomena are well-understood and a big and active field
of investigation. However, for our purposes, \textit{they do not matter} and
we will be fine not digging deeper into this. We merely wanted to point out
that this is an issue where some caution is appropriate.

A second elaboration: For $K$-theory it does not matter whether we work with
the genuine category $\mathsf{C}$ or with bounded complexes in $\mathsf{C}$.
Thus, simultaneously to the above principles (a) and (b), the following two
are also true:\newline$\left.  \qquad\text{\textbf{(a')}}\right.  $ Every
bounded complex $X_{\bullet}$ of objects in $\mathsf{C}$ determines a point
$X_{\bullet}$ in the space $K(\mathsf{C})$, and also determines a point
$-X_{\bullet}$ using negation.\newline$\left.  \qquad\text{\textbf{(b')}%
}\right.  $ Every quasi-isomorphism $X_{\bullet}\overset{\sim}{\rightarrow
}Y_{\bullet}$ determines a path from the point of $X_{\bullet}$ to
$Y_{\bullet}$ in $K(\mathsf{C})$.\newline Finally, if $F:\mathsf{C}%
\longrightarrow\mathsf{C}^{\prime}$ is an exact functor of exact categories,
then there is an induced map of spaces $K(\mathsf{C})\rightarrow
K(\mathsf{C}^{\prime})$.

In \S \ref{sect_GettingPrecise} we will give a fully rigorous and precise
justification for principles (a) and (b), as well as (a') and (b').

\begin{remark}
One may think about the principles (a), (b), (a'), (b') as follows, in
analogy: Both a single point as well as the real line $\mathbb{R}$ have the
same homotopy type. However, $\mathbb{R}$ has a lot more points and a lot more
paths. Similarly, when we choose whether we use Quillen's plus construction,
the $Q$-construction, Waldhausen's $S$-construction, etc. to model the
$K$-theory space, we always get the same homotopy type having the same
homotopy groups and therefore same $K$-groups. However, the wealth of points
or paths in these spaces varies a lot. Thus, our choice to use the
Gillet--Grayson type model from \cite{MR1167575} is special. One may think of
it as a sufficiently \textquotedblleft fattened up\textquotedblright%
\ incarnation of the homotopy type of $K$-theory such that all our principles
(a), (b), (a'), (b') hold.\footnote{and to be fully honest we do not just need
that the model has the correct homotopy type; it also needs to have the
correct infinite loop space structure, which is true for all of the well-known
models of $K$-theory (except for the plus construction approach).}
\end{remark}

We are ready to state our construction of the Tamagawa number. Let $F$ be a
number field, $S_{\infty}$ its set of infinite places, and fix a separable
closure $F^{\operatorname*{sep}}$ (i.e. an algebraic closure). Let
$M\in\operatorname*{CHM}(F,\mathbb{Q})$ be a Chow motive over $F$ in the
category of Chow motives with $\mathbb{Q}$-coefficients, \cite{MR2115000},
\cite{MR1265529}. If the reader does not like motives, we may take $M$ to be a
smooth proper $F$-variety, or even something as concrete as an elliptic curve.
Let $A$ be a finite-dimensional semisimple $\mathbb{Q}$-algebra and suppose
$M$ carries a right action by $A$ as a Chow motive. This just means that we
provide a $\mathbb{Q}$-algebra homomorphism%
\begin{equation}
A\longrightarrow\operatorname*{End}\nolimits_{\operatorname*{CHM}%
(F,\mathbb{Q})}(X)\text{.}\label{lww_Action}%
\end{equation}

If $\mathfrak{A}\subset A$ is an order in the algebra, we use the standard
notation%
\[
\widehat{\mathfrak{A}}:=\mathfrak{A}\otimes_{\mathbb{Z}}\widehat{\mathbb{Z}%
}\text{,}\qquad\widehat{A}:=A\otimes_{\mathbb{Q}}\mathbb{A}_{fin}%
\text{,}\qquad A_{\mathbb{R}}:=A\otimes_{\mathbb{Q}}\mathbb{R}\text{,}%
\]%
\[
\mathfrak{A}_{p}:=\mathfrak{A}\otimes_{\mathbb{Z}}\mathbb{Z}_{p}%
\qquad\text{and}\qquad A_{p}:=A\otimes_{\mathbb{Q}}\mathbb{Q}_{p}%
\text{,}\qquad\qquad\text{(for any prime number }p\text{)}%
\]
where $\mathbb{A}_{fin}$ denotes the ring of finite ad\`{e}les of the
rationals (Example: $\widehat{F}$ are the finite ad\`{e}les of $F$). We follow
Burns and\ Flach and shall use the same notation, so in particular

\begin{enumerate}
\item for any infinite place $v$ we write $H_{v}(M):=H^{\ast}(M_{v}%
(\mathbb{C}),(2\pi i)^{\ast}\mathbb{Q})$ for the Betti realization $M_{v}$
along $v:F\hookrightarrow\mathbb{C}$, and for $(\ast,\ast)$ picked appropriately,

\item and for each prime $p$, we write $H_{p}(M):=H_{\acute{e}t}^{\ast
}(M\times_{F}F^{\operatorname*{sep}},\mathbb{Q}_{p}(\ast))$ for the \'{e}tale
realization, for $(\ast,\ast)$ picked appropriately,

\item $H_{dR}$ for the de Rham realization with suitably shifted Hodge filtration.
\end{enumerate}

We will not go into details. This setting is roughly the same in all papers
about the ETNC or its historical ancestors.

\begin{elaboration}
\label{elab_Motives1}How to attach realizations to the motive is explained in
many places. Omitting a lot of details, the story is as follows: If $X$ is a
smooth proper $F$-variety, a typical choice of $M$ would be%
\[
M:=h^{i}(X)(r)\text{.}%
\]
In general, the splitting of $X$ into a direct sum of cohomology pieces
$h^{i}$ in the category of Chow motives is conjectural. But let us assume this
direct summand exists. Then%
\[
M=(X,q,r)\text{,}%
\]
where $X$ is the variety as before, $q$ an idempotent self-correspondence
which has the property to cut out the direct summand $h^{i}$, $r$ a formal
Tate twist parameter.\ One may write $M=q_{\ast}(X,\Gamma_{\operatorname*{id}%
:X\rightarrow X},0)(r)$, where $\Gamma_{\operatorname*{id}:X\rightarrow X}$ is
the graph of the identity map. Then, just as $M$ is cut out from an idempotent
inside the motive of all of $X$, for an infinite place $v$ we would define the
Betti realization as the corresponding image of the idempotent in the Betti
cohomology of $X$,%
\[
H_{v}(M)=q_{\ast}H^{\ast}(X_{v}(\mathbb{C}),(2\pi i)^{r}\mathbb{Q})\text{.}%
\]
This story then is analogous for the other realizations. The realizations all
come with extra structures (e.g. a pure $\mathbb{Q}$-Hodge structure on the
Betti cohomology groups), which we also tacitly keep as data, and which need
to be shifted according to $r$.
\end{elaboration}

The key point is that correspondences act on all Weil cohomology theories and
thus one can define these cohomology groups for all Chow motives,
\cite{MR2115000}. The same is possible for mixed motives, albeit technically
much harder \cite{MR1775312, MR2008720}.

\begin{definition}
[{\cite[\S 3.3, Definition 1]{MR1884523}}]For every infinite place $v\in
S_{\infty}$ pick a choice of $T_{v}$ of a projective $\mathfrak{A}$-lattice in
the Betti realization $H_{v}(M)$ (which is a right $A$-module by Equation
\ref{lww_Action}).

\begin{enumerate}
\item Let $p$ be any prime. Suppose for all $v\in S_{\infty}$ the
Betti-to-\'{e}tale comparison isomorphism%
\[
H_{v}(M)\otimes_{\mathbb{Q}}\mathbb{Q}_{p}\overset{\sim}{\longrightarrow}%
H_{p}(M)
\]
sends the $\mathfrak{A}_{p}$-submodule $T_{v}\otimes_{\mathbb{Z}}%
\mathbb{Z}_{p}$ to the same image $T_{p}$ on the right-hand side.

\item Suppose further that the $T_{p}$ of (1) is stable under the Galois
action $G_{F}$ on the right side.
\end{enumerate}

Any such choice $(T_{v})_{v\in S_{\infty}}$ is called a\emph{ projective}
$\mathfrak{A}$\emph{-structure} on the motive $M$.
\end{definition}

As explained in loc. cit., a projective $\mathfrak{A}$-structure need not
exist in general. We write $\mathsf{LCA}_{\mathfrak{A}}$ for the exact
category of locally compact topological right $\mathfrak{A}$-modules, as
introduced in \cite{etnclca}.

In this paper, we shall establish that the following is a fibration:

\begin{theoremannounce}
Suppose $\mathfrak{A}\subset A$ is a regular order. Then there is a canonical
fibration of pointed spaces%
\begin{equation}
K(\widehat{\mathfrak{A}})\times K(A)\longrightarrow K(\widehat{A})\times
K(A_{\mathbb{R}})\longrightarrow K(\mathsf{LCA}_{\mathfrak{A}})\text{,}
\label{lcixi1}%
\end{equation}
which we call the \emph{principal id\`{e}le fibration}.
\end{theoremannounce}

This will be Theorem \ref{thm_PrincipalIdeleFibration}. We will shortly see
how the spaces in this sequence relate to the $p$-adic, de Rham and Betti
realization. The word \textquotedblleft fibration\textquotedblright\ is meant
in the sense of topology (but we shall not need to know anything technical
about it right now): The zero object $0$ of any of the involved categories
defines, by principle (a), a point in the $K$-theory spaces. We use this
canonical point as the base point for each of the involved spaces. The
statement means that the space $K(\widehat{A})\times K(A_{\mathbb{R}})$ is
fibered over the base space $K(\mathsf{LCA}_{\mathfrak{A}})$, and each fiber
looks like (that is: has the homotopy type of) $K(\widehat{\mathfrak{A}%
})\times K(A)$. For example, the M\"{o}bius band is fibered over the circle,%
\begin{equation}%
{\includegraphics[
height=0.8441in,
width=2.1958in
]%
{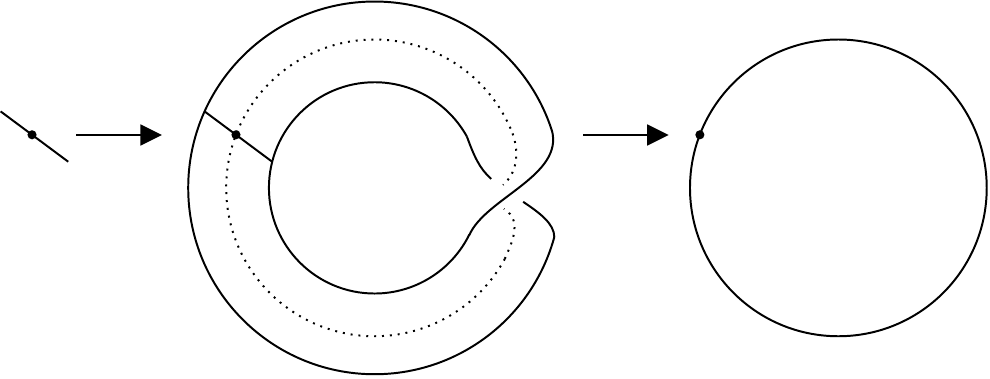}%
}
\label{l_Mobius}%
\end{equation}
and the above three constituents, fiber (left), total space (middle), and base
space (right) correspond to the three terms in Equation \ref{lcixi1}. Of
course, the spaces in Equation \ref{lcixi1} are a lot more complicated and it
would be impossible to draw a picture.

Actually, our construction of the Tamagawa number does not even need the full
strength of the above theorem. We shall only use that the composition of both
maps is zero. More geometrically: Once the fiber is mapped all to the right in
Equation \ref{lcixi1}, there exists a homotopy contracting this image to the
base point.

Using exactly the same notation as in Burns--Flach \cite{MR1386106},
\cite{MR1884523} for the individual groups, we define%

\begin{align}
\Xi(M)  &  :=H_{f}^{0}(F,M)-H_{f}^{1}(F,M)+H_{f}^{1}(F,M^{\ast}(1))^{\ast
}-H_{f}^{0}(F,M^{\ast}(1))^{\ast}\label{ltiops1}\\
&  -\sum_{v\in S_{\infty}}H_{v}(M)^{G_{v}}+\sum_{v\in S_{\infty}}\left(
H_{dR}(M)/F^{0}\right) \nonumber
\end{align}
as a point in $K(A)$. To clarify: (a) The meaning of \textquotedblleft%
$+$\textquotedblright\ and \textquotedblleft$-$\textquotedblright\ is as in
Equation \ref{lww_AddSubtract} and unravelled from left to right.

Exactly as in Burns--Flach, we write $H^{i}(F,M)$ for (what is usually
shortened to be called) the motivic cohomology of $M$ with $\mathbb{Q}%
$-coefficients. See \cite[\S 3.1]{MR1884523}. This notation is chosen to be
suggestive for%
\begin{equation}
H^{i}(F,M)=\operatorname*{Ext}\nolimits_{\mathcal{MM}_{F}}^{i}(\mathbb{Q}%
,M)\text{,}\label{l_lemo2}%
\end{equation}
which is how these motivic cohomology groups can be defined in terms of the
category of mixed motives over $F$ (e.g. using Voevodsky's $DM$) and using the
natural contravariant functor sending a Chow motive into mixed motives.
Analogously, we write $H_{f}^{i}(F,M)$ for what is called the `finite part'
(and also known as the `unramified part' sometimes); $H_{v}(M)$ denotes Betti
cohomology under the complex realization of the base change $M\times
_{F}\mathbb{C}$ along $\sigma:F\hookrightarrow\mathbb{C}$, $G_{v}$ the
decomposition group, so that $H_{v}(M)^{G_{v}}$ is the piece fixed under
complex conjugation for real places, $H_{dR}(M)/F^{0}$ is de Rham cohomology
(over $F$) modulo $F^{0}$, where $F^{\bullet}$ is the standard decreasing filtration.

\begin{elaboration}
Let us continue Elaboration \ref{elab_Motives1}. The picture is as follows: In
a lot of literature by the \textquotedblleft motivic cohomology of a motive
$M$\textquotedblright\ one would mean $H^{i}(M,\mathbb{Q}%
(j))=\operatorname*{Ext}\nolimits_{\mathcal{MM}_{F}}^{i}(M,\mathbb{Q}(j))$,
where $\mathbb{Q}(j)$ are the motivic coefficient sheaves as for example
introduced in the book \cite[Lecture 3]{MR2242284}. As we restrict to rational
coefficients, the motivic cohomology groups $\operatorname*{Ext}%
\nolimits_{\mathcal{MM}_{F}}^{i}(X,\mathbb{Q}(r))$ can also be expressed as
eigenspaces of the Adams operations on $K$-theory, which is historicially the
pioneering approach to define them at all. This is also called
\textquotedblleft\textit{absolute} motivic cohomology\textquotedblright\ and
is simply called motivic cohomology in \cite{MR2242284}. The connection
between this usage of the term and the one here is as follows: if we again
consider a motive of the particular form%
\[
M:=h^{i}(X)(r)
\]
as in Elaboration \ref{elab_Motives1}, then in a 6-functor formalism of mixed
motivic sheaves\footnote{as exists by work of Ayoub} there is a (Leray-type)
spectral sequence%
\[
\operatorname*{Ext}\nolimits_{\mathcal{MM}_{F}}^{j}(\mathbb{Q},h^{i}%
(X)(r))\Longrightarrow\operatorname*{Ext}\nolimits_{\mathcal{MM}_{F}}%
^{j+i}(X,\mathbb{Q}(r))\text{,}%
\]
where one has $\operatorname*{Ext}\nolimits_{\mathcal{MM}_{F}}^{j}%
(\mathbb{Q},-)=0$ for $j\neq0,1$ thanks to $F$ being a number field
(\cite[Corollary 1.1.13]{MR1775312}). Having only two possibly non-zero
columns, one gets a supply of short exact sequences%
\begin{equation}
0\rightarrow\operatorname*{Ext}\nolimits_{\mathcal{MM}_{F}}^{1}(\mathbb{Q}%
,h^{i}(X)(r))\rightarrow\operatorname*{Ext}\nolimits_{\mathcal{MM}_{F}}%
^{i+1}(X,\mathbb{Q}(r))\rightarrow\operatorname*{Ext}\nolimits_{\mathcal{MM}%
_{F}}^{0}(\mathbb{Q},h^{i+1}(X)(r))\rightarrow0\text{.}\label{l_lemo3}%
\end{equation}
Instead of following this picture coming from the Beilinson conjectures, a lot
of literature (like \cite{MR1386106}, \cite{MR1884523}, \cite{MR2882695}%
,\ldots) takes this as an implicit axiom, and presents the $A$-modules
$H^{i}(F,M)$ of Equation \ref{l_lemo2} as being plainly defined as the output
of what the sequence in Equation \ref{l_lemo3} would give. The $H^{i}(F,M)$
would be called \textquotedblleft\textit{geometric} motivic
cohomology\textquotedblright; see for example the introduction of
\cite{MR1439046} or \cite{MR1265544}.
\end{elaboration}

All these objects%
\[
H^{i}(F,M)\qquad H_{f}^{i}(F,M)\qquad H_{v}(M)\qquad H_{dR}(M)
\]
as well as their counterparts for the Tate twist $M(1)$, carry a canonical
right $A$-module structure coming from the right action of $A$ on the motive,
Equation \ref{lww_Action}. Hence, by principle (a) they each define a point in
$K(A)$ and via the operations of Equation \ref{lww_AddSubtract} we can form
$\Xi(M)$. Yes, it is true, this object \textit{depends} on how we bracket it
to evaluate the sums and negations, but we just once and for all choose to
unravel it from left to right. In fact, a posteriori the choice turns out not
to matter, so this aspect is not very important. Any choice is good enough.

Now we have a point $\Xi(M)$ in the space $K(A)$. Next, for any prime number
$p$ there is a comparison quasi-isomorphism leading via principle (b') to a
path%
\begin{equation}
\vartheta_{p}:\Xi(M)\otimes_{A}A_{p}\overset{\sim}{\longrightarrow}R\Gamma
_{c}\left(  \mathcal{O}_{F,S_{p}},T_{p}\right)  \otimes_{\mathfrak{A}_{p}%
}A_{p} \label{lww_Z3}%
\end{equation}
i.e. a path from $\Xi(M)\otimes_{A}A_{p}$ to $R\Gamma_{c}\left(
\mathcal{O}_{F,S_{p}},T_{p}\right)  \otimes_{\mathfrak{A}_{p}}A_{p}$.

\begin{elaboration}
\label{elab_ThetaPConstructedAsInBF}We explain where this comes from: The
bounded complex $R\Gamma_{c}\left(  \mathcal{O}_{F,S_{p}},T_{p}\right)
\otimes_{\mathfrak{A}_{p}}A_{p}$ defines a point in $K(\widehat{A})$ by
principle (a'). Next, $\Xi(M)$ is a point in $K(A)$ by construction, and under
the exact functor $(-)\mapsto(-)\otimes_{A}A_{p}$ inducing a map
$K(A)\rightarrow K(\widehat{A})$ it gets sent to a point which we may
reasonably call $\Xi(M)\otimes_{A}A_{p}$ (it can be spelled out explicitly as
the result of tensoring each summand in Equation \ref{ltiops1} with $A_{p}$).
Now use the quasi-isomorphisms of (and here we quote Burns--Flach
\cite[\S 3.4, middle of page 526]{MR1884523} directly) \textquotedblleft(27),
(28), (23), the isomorphisms (24), (19) or the triangle (22) for all $v\in
S_{p,f}$ and finally [to] the triangle (26)$_{\text{vert}}$\textquotedblright%
\ (exactly the input used in the framework of virtual objects loc. cit.) as
input for principle (b') to turn a quasi-isomorphism into a path. By directly
quoting this from Burns--Flach, we do not only save ourselves from repeating
the setup of \cite[\S 3.2]{MR1884523}, it will also help us in
\S \ref{sect_CompatProofSection} to prove the comparison to the Burns--Flach
approach, because our path $\vartheta_{p}$ is literally made from the same
quasi-isomorphisms as the analogous isomorphism of virtual objects in their
paper (and also called $\vartheta_{p}$ there).
\end{elaboration}

Thus, we learn: Once we send $\Xi(M)$ along the first arrow to $K(\widehat
{A})$, then in this space we have a canonical path from $\Xi(M)\otimes
_{A}A_{p}$ to $R\Gamma_{c}\left(  \mathcal{O}_{F,S_{p}},T_{p}\right)
\otimes_{\mathfrak{A}_{p}}A_{p}$. On the other hand, it is conjectured that:

\begin{conjecture}
[{e.g., \cite{MR1206069}, \cite[Conjecture 1]{MR1884523}}]There is the basic
exact sequence%
\begin{align*}
&  0\longrightarrow H^{0}(F,M)_{\mathbb{R}}\overset{\epsilon}{\longrightarrow
}\ker\left(  \alpha_{M}\right)  \overset{r_{B}^{\ast}}{\longrightarrow}\left(
H_{f}^{1}(F,M^{\ast}(1))_{\mathbb{R}}\right)  ^{\ast}\overset{\delta
}{\longrightarrow}\\
&  \qquad\qquad\qquad H_{f}^{1}(F,M)_{\mathbb{R}}\overset{r_{B}}%
{\longrightarrow}\operatorname*{coker}\left(  \alpha_{M}\right)
\overset{\epsilon^{\ast}}{\longrightarrow}\left(  H^{0}(F,M^{\ast
}(1))_{\mathbb{R}}\right)  ^{\ast}\longrightarrow0\text{.}%
\end{align*}
in the category $\operatorname*{PMod}(A_{\mathbb{R}})$.
\end{conjecture}

However, a sequence being exact is of course the same as saying that it is
quasi-isomorphic to the zero complex. The map to a zero object is canonical,
so we obtain a canonical quasi-isomorphism%
\begin{equation}
\vartheta_{\infty}:\Xi(M)\otimes_{A}A_{\mathbb{R}}\overset{\sim}%
{\longrightarrow}\mathbf{0}\text{.} \label{lww_x5}%
\end{equation}
Hence, by principle (b') we learn:\ Once we send $\Xi(M)$ along the first
arrow in Equation \ref{lcixi1} to $K(A_{\mathbb{R}})$, then in this space we
have a canonical path from $\Xi(M)\otimes_{A}A_{\mathbb{R}}$ to zero.

Let us summarize this: Under the first arrow in the principal id\`{e}le
fibration, the point $\Xi(M)$ gets send to a point in $K(\widehat{A})\times
K(A_{\mathbb{R}})$ and here for all primes and at infinity we get a canonical
collection of paths%
\[
\vartheta_{(-)}:\Xi(M)\otimes_{A}(-)\qquad\longrightarrow\qquad\left(  \left.
\text{zero/some object from }K(\mathfrak{A}_{p})\right.  \right)  \text{.}%
\]
Now, let us use the second arrow in Equation \ref{lcixi1}. Since the
composition of both arrows in a fibration is zero (or more precisely: can be
contracted to the constant zero map), we obtain the following: The object
$\Xi(M)$ goes to the zero base point in $K(\mathsf{LCA}_{\mathfrak{A}})$ since
it comes all from the left, namely $K(A)$. Moreover, all the objects
$R\Gamma_{c}\left(  \mathcal{O}_{F,S_{p}},T_{p}\right)  \otimes_{\mathfrak{A}%
_{p}}A_{p}$ on the right in Equation \ref{lww_Z3} go to zero in
$K(\mathsf{LCA}_{\mathfrak{A}})$ because they come all from the left, namely
$K(\widehat{\mathfrak{A}})$ (the $\mathfrak{A}_{p}$-modules are $\widehat
{\mathfrak{A}}$-modules as well). But this just means that the endpoints of
the paths $\vartheta_{p}$ resp. $\vartheta_{\infty}$ all go to zero in
$K(\mathsf{LCA}_{\mathfrak{A}})$. That is: \textit{they all define closed
loops} around the zero object. Hence, we get elements%
\[
\vartheta_{p},\vartheta_{\infty}\in\pi_{1}K(\mathsf{LCA}_{\mathfrak{A}%
})\text{,}%
\]
and this in turn is just the $K$-group $K_{1}(\mathsf{LCA}_{\mathfrak{A}})$.

\begin{example}
We illustrate this construction by drawing the principal id\`{e}le fibration
in a similar style as our example of the M\"{o}bius band in Equation
\ref{l_Mobius}. For simplicity, we only include a single finite prime $p$ and
$\infty$.%
\[%
{\includegraphics[
height=1.2635in,
width=4.9744in
]%
{fibration_2a-eps-converted-to.pdf}%
}
\]
On the left, we have the fiber. Besides the zero object $\mathbf{0}$, we have
$\Xi(M)$ as a point in $K(A)$ and $R\Gamma_{c}\left(  \mathcal{O}_{F,S_{p}%
},T_{p}\right)  $ as a point in $K(\widehat{\mathfrak{A}})$, by principle (a)
and its variations. These are the various points depicted on the left. Once we
map them to the total space, depicted in the middle, we can construct paths
between these points. These come from principle (b) and its variations:
Quasi-isomorphisms like $\vartheta_{p}$ and $\vartheta_{\infty}$ give rise to
paths. These paths do not need to exist in the fiber on the left, because in
general they are not $\widehat{\mathfrak{A}}$- or $A$-module isomorphisms. On
the right, we have the base space. The three points in the fiber now all get
mapped to the base point. Thus, the paths we have drawn in the total space now
become closed loops. Thus, they define an element in the fundamental group
$\pi_{1}K(\mathsf{LCA}_{\mathfrak{A}})$.
\end{example}

We also construct a class $L\Omega(M,\mathfrak{A}):=\hat{\delta}%
_{\mathfrak{A},\mathbb{R}}^{1}(L^{\ast}(\left.  M_{A}\right.  ,s))$ attached
to the equivariant special $L$-value. To this end, we construct an extended
boundary map $\hat{\delta}_{\mathfrak{A},\mathbb{R}}^{1}$ as in
\cite{MR1884523}. The construction of $L\Omega$ is essentially the same as
Burns and Flach give, so nothing new happens here. It turns out that all but
finitely many of the loops $\vartheta_{(-)}$ are trivial, so:

\begin{definition}
We define%
\begin{equation}
R\Omega(M,\mathfrak{A}):=\prod_{v}\vartheta_{v}\qquad\in\qquad K_{1}%
(\mathsf{LCA}_{\mathfrak{A}})\text{,} \label{lww_Z4}%
\end{equation}
where $v$ runs through all places of $\mathbb{Q}$. We call%
\[
T\Omega(M,\mathfrak{A}):=L\Omega(M,\mathfrak{A})+R\Omega(M,\mathfrak{A})
\]
the\emph{ equivariant Tamagawa number} of the motive $M$ with respect to the
order $\mathfrak{A}$.
\end{definition}

\begin{elaboration}
We have swept the dependency on $S$ and the projective $\mathfrak{A}%
$-structure under the rug. However, an argument analogous to \cite[Lemma
5]{MR1884523} removes this apparent dependency in a way completely analogous
to how Burns--Flach establish this.
\end{elaboration}

\section{\label{sect_GettingPrecise}Getting precise}

In the previous section we have explained the construction of our equivariant
Tamagawa number $T\Omega$ along what we have called principles (a) and (b). We
had focussed on explaining the geometry of our construction, but had neglected
justifying these principles rigorously. We do this in this section.

\subsection{Spaces\label{subsect_Spaces}}

In \S \ref{sect_Overview} we were talking about $K$-theory as a \textit{space}%
.\ What do we mean? Basically, there are two fundamentally equivalent ways to
do homotopy theory. Close to intuition is the following one: By
\textquotedblleft space\textquotedblright\ we refer to a topological space
$X$. A point $x$ is really an element $x\in X$ of this space, a path is a
continuous map $p:[0,1]\rightarrow X$, $p(0)$ is the starting point and $p(1)$
the endpoint, and so on.

The reader has surely seen this. The category $\mathsf{Top}_{\bullet}$ has all
topological spaces as objects along with a chosen point, called the base
point. Morphisms are the continuous maps preserving the pointing. The term
\textquotedblleft fibration\textquotedblright\ is defined as a Serre
fibration. Homotopy groups are defined as the based homotopy classes of
pointed maps%
\[
S^{n}\longrightarrow X\text{,}%
\]
where $S^{n}$ denotes the $n$-sphere, pointed at $(1,0,\ldots,0)$.

This setup is very intuitive since it connects well with how we usually do
geometry. However, one can also do homotopy theory entirely combinatorially
without ever touching a topology: Then, by \textquotedblleft
space\textquotedblright\ we refer to a simplicial set $X_{\bullet}$. A point
$x$ is a $0$-simplex, i.e. an element $x\in X_{0}$. An elementary path is a
$1$-simplex, i.e. an element $p\in X_{1}$, $\partial_{1}p\in X_{0}$ is the
starting point and $\partial_{0}p\in X_{0}$ the endpoint, paths are finite
concatenations of elementary paths\footnote{Some people prefer only working
with elementary paths as the notion of \textquotedblleft
path\textquotedblright. This is also reasonable (and simpler), but then one
only gets a well-behaved concept of paths for fibrant simplicial sets.}, and
so on.

The category $\mathsf{sSet}_{\bullet}$ has simplicial sets as objects along
with a chosen point, called the base point. Morphisms are maps of simplicial
sets preserving the pointing. The term \textquotedblleft
fibration\textquotedblright\ is defined as a Kan fibration. Homotopy groups
are defined as the simplicial homotopy classes of pointed maps%
\begin{equation}
\Delta(n)\longrightarrow\operatorname*{Ex}\nolimits^{\infty}X\text{,}
\label{l_Kan}%
\end{equation}
where $\Delta(n)$ denotes the standard simplicial $n$-simplex and
$\operatorname*{Ex}\nolimits^{\infty}$ is Kan's functorial fibrant replacement
functor (a technical device which is of no importance to what we do in this
paper). General references for simplicial homotopy theory are \cite{MR1206474}%
, \cite{MR1650134}, or \cite{MR2840650}.

Both approaches are very parallel. And indeed Quillen proved both
$\mathsf{Top}_{\bullet}$ and $\mathsf{sSet}_{\bullet}$ are so-called `model
categories', which one can think of as saying that they both possess all the
structure to do homotopy theory. A few references: $\mathsf{Top}_{\bullet}$
and its model category structure is very carefully set up and discussed in
\cite[\S 2.4]{MR1650134}; $\mathsf{sSet}_{\bullet}$ and its model category
structure is set up in \cite[\S 3.2]{MR1650134}.

Indeed, there is an adjunction%
\begin{equation}
\mathsf{sSet}_{\bullet}\rightleftarrows\mathsf{Top}_{\bullet}\text{,}
\label{lwaa1}%
\end{equation}
the left adjoint sending a simplicial set $X_{\bullet}$ to its geometric
realization $\left\vert X_{\bullet}\right\vert $ (which basically glues
topological $i$-cells according to the glueing rules prescribed by the
simplicial set structure) and reversely the right adjoint sending a space to
its simplicial set of maps $\operatorname*{Sing}_{\bullet}(X):=\{f:S^{n}%
\rightarrow X$, $f$ continuous$\}$, \cite[p. 77]{MR1650134}. Sweeping some
technicalities under the rug, this adjunction can be promoted to a so-called
Quillen equivalence, which roughly speaking means that the concepts of
fibration, homotopy groups, etc. of both model categories are
compatible\footnote{Really $\mathsf{Top}_{\bullet}$ should first be replaced
by $k$-spaces $\mathsf{K}_{\bullet}$. Doing this, the said adjunction with the
same functors gives a Quillen equivalence, \cite[Theorem 3.6.7]{MR1650134}.
Then, there is a further Quillen equivalence between $\mathsf{K}_{\bullet}$
and $\mathsf{Top}_{\bullet}$, \cite[Corollary 2.4.24]{MR1650134}. It follows
that both approaches are connected by a zig-zag of Quillen equivalences.}. We
do not need to understand any of that for this paper, only the following
consequence: There is no difference between whether we do homotopy theory in
$\mathsf{sSet}_{\bullet}$ or $\mathsf{Top}_{\bullet}$.\medskip

As a convention: From now on, we work in the setting of $\mathsf{sSet}%
_{\bullet}$, i.e. the word `space' means a simplicial set. Keeping the
equivalence of $\mathsf{sSet}_{\bullet}$ and $\mathsf{Top}_{\bullet}$ in mind,
we may however always use $\mathsf{Top}_{\bullet}$ whenever we feel in need to
get some geometric intuition.

\subsection{Algebraic $K$-theory\label{sect_AlgKThy}}

\subsubsection{Definition as a space}

As a motivation, recall the definition of $K_{0}$ (we ask the reader for
forgiveness if this appears too elementary, but there is a good reason to go
through this): If $R$ is a ring, let $Pr(R)$ denote the set of isomorphism
classes of finitely generated projective right $R$-modules. This is an abelian
monoid under the direct sum $[X]+[Y]:=[X\oplus Y]$. However, there is no
reason why additive inverses, like some \textquotedblleft$-[X]$%
\textquotedblright\ would have to exist. Then define%
\[
K_{0}(R):=GC(Pr(R))\text{,}%
\]
where $GC(-)$ denotes the group completion: This is a general operation
turning abelian monoids into abelian groups. It can be defined as follows: If
$M$ is an abelian monoid, consider the quotient set%
\[
GC(M):=\left\{  \frac{\text{pairs }(P,Q)\in M\times M}{(P,Q)\sim(P\underset
{M}{+}S,Q\underset{M}{+}S)\text{ for all }S\in M}\right\}  \text{.}%
\]
It is easy to show that defining%
\[
(P,Q)+(P^{\prime},Q^{\prime}):=(P\underset{M}{+}P^{\prime},Q\underset{M}%
{+}Q^{\prime})\qquad\text{and}\qquad-(P,Q):=(Q,P)
\]
renders $GC(M)$ into an abelian group. Define%
\[
M\longrightarrow GC(M)\qquad\text{by}\qquad P\mapsto(0,P)\text{.}%
\]
One can show that any monoid morphism from $M$ to an abelian group $A$ factors
uniquely over $GC(M)$, so $GC(-)$ is the universal construction transforming
abelian monoids into abelian groups.\footnote{more precise: it is the left
adjoint of the forgetful functor from abelian groups to abelian monoids.} This
construction of $K_{0}$ extends to split exact categories $\mathsf{C}$, define
$K_{0}(\mathsf{C}):=GC(Iso(\mathsf{C}))$, where $Iso(\mathsf{C})$ is the set
of isomorphism classes of objects, turned into a monoid using the direct sum.

\begin{remark}
\label{rmk_MinusOneMapGC}Every element $P\in M$ maps to $(0,P)$ in $GC(M)$.
Correspondingly, we observe $-P=(P,0)$, and in particular the automorphism of
$GC(M)$ exchanging $(P,Q)$ with $(Q,P)$ corresponds to multiplication by $-1$.
\end{remark}

There is a way to define also all higher $K$-groups and in particular the
$K$-theory space $K(\mathsf{C})$ in a rather similar way: If $\mathsf{C}$ is a
category, we write $s\mathsf{C}$ for the category of simplicial objects in
$\mathsf{C}$. A \emph{Waldhausen category} is a pointed\footnote{a pointed
category is a category along with a choice of a fixed zero object} category
with a choice of cofibrations and a choice of weak equivalences, satisfying
the usual axioms. A detailed definition is given in Weibel \cite[Ch. II,
Definition 9.1.1]{MR3076731}. We write $X^{\prime}\hookrightarrow X$ to denote
cofibrations. In a category with cofibrations every cofibration admits a
cokernel (as follows from the axioms), and any sequence%
\[
X^{\prime}\hookrightarrow X\longrightarrow C\text{,}%
\]
where $C$ is a cokernel object, is called a \emph{cofibration sequence}. We
will write any cofibration sequence as%
\[
X^{\prime}\hookrightarrow X\twoheadrightarrow C\text{.}%
\]

\begin{remark}
[Coproducts exist]\label{rmk_Coproducts}Every Waldhausen category has finite
coproducts. By the axioms for any object $X\in\mathsf{C}$ there is a canonical
cofibration $0\hookrightarrow C$ and taking the pushout of this arrow along a
second copy of itself, which exists by the axioms, we get an object which we
denote by $C\vee C$ (following the notation of \cite{MR1167575}). This object
is unique up to unique isomorphism.
\end{remark}

As is customary, we usually write $w\mathsf{C}$ for the category whose objects
are the same as those of the Waldhausen category $\mathsf{C}$, but we only
keep the weak equivalences as morphisms. Similarly, we write $i\mathsf{C}$ if
we use the same objects, but only keep the isomorphisms. Let $\mathsf{Wald}$
denote the category\footnote{treat this as a $2$-category if you prefer, but
it is not really necessary for our purposes} whose objects are Waldhausen
categories and morphisms are exact functors. In \cite{MR1167575} this category
is called \textquotedblleft$w\mathcal{C}of$\textquotedblright. Waldhausen's
$S$-construction is a functor%
\[
S_{\bullet}:\mathsf{Wald}\longrightarrow s\mathsf{Wald}%
\]
from Waldhausen categories to simplicial Waldhausen categories. Note that the
latter is a simplicial object in categories (and not, as one could think, a
category enriched in simplicial sets. Unfortunately, both are frequently
called a `simplicial category'). We write $PX_{\bullet}$ for the simplicial
path space (where customarily, we use the right path space. There is also a
left path space, see \cite[\S 2.2.3]{MR3782417} for a comparison. In the end,
this choice does not matter). On simplices, $PX_{n}:=X_{n+1}$, in both cases.

Following \cite{MR1167575}, define a functor%
\[
G_{\bullet}:\mathsf{Wald}\longrightarrow s\mathsf{Wald}%
\]
by forming the Cartesian square%
\begin{equation}%
\xymatrix{
G_{\bullet}\mathsf{C} \ar[r] \ar[d] & PS_{\bullet}\mathsf{C} \ar[d] \\
PS_{\bullet}\mathsf{C} \ar[r] & S_{\bullet}\mathsf{C}.
}
\label{lkiops1}%
\end{equation}
While this a priori only defines an object $G\in s\mathsf{Cat}$, define the
cofibrations (resp. weak equivalences) to be the Cartesian product, too. This
means that a morphism in $G_{n}$ is a cofibration (resp. weak equivalence) if
it is given by a pair $(f_{1},f_{2})$ of cofibrations (resp. weak
equivalences) of $PS_{\bullet}\mathsf{C}\times PS_{\bullet}\mathsf{C}$.

\begin{definition}
[\cite{MR909784}, \cite{MR1167575}]If $\mathsf{C}$ is a Waldhausen category,
the construction $G_{\bullet}\mathsf{C}$ is called the\ \emph{Gillet--Grayson
model} of $\mathsf{C}$.
\end{definition}

\begin{remark}
[Exact categories as Waldhausen categories]\label{rmk_ExactToWald}Gillet and
Grayson \cite{MR909784, MR2007234} originally introduced $G_{\bullet
}\mathsf{C}$, but they only considered it for exact categories. This amounts
to taking a pointed exact category, taking its admissible monics as the class
of cofibrations and its isomorphisms as the class of weak equivalences.
\end{remark}

Only in \cite{MR1167575} the construction was extended to (fairly) arbitrary
Waldhausen categories. We will have crucial need for this broader variant, so
\cite{MR1167575} will be our principal foundation for the Gillet--Grayson model.

\begin{example}
\label{example_GG}We can unravel the definition of $G_{\bullet}(\mathsf{C})$
in concrete terms. Its $q$-simplices $G_{q}(\mathsf{C})$ is given by the
following Waldhausen category. Its objects are pairs of diagrams
\begin{equation}%
\xymatrix@!=0.157in{
&                         &                         &                        & X_{n/(n-1)}
\\
&                         &                         & \cdots\ar@{^{(}%
.>}[r] & \vdots\ar@{.>>}[u] \\
&                         & X_{2/1} \ar@{^{(}.>}[r] & \cdots\ar@{^{(}%
.>}[r] & X_{n/1} \ar@{.>>}[u] \\
& X_{1/0} \ar@{^{(}.>}[r] & X_{2/0} \ar@{^{(}.>}[r] \ar@{.>>}[u] & \cdots
\ar@{^{(}.>}[r] & X_{n/0} \ar@{.>>}[u] \\
X_0 \ar@{^{(}.>}[r] & X_1 \ar@{^{(}.>}[r] \ar@{.>>}[u] & X_2 \ar@{^{(}%
.>}[r] \ar@{.>>}[u] & \cdots\ar@{^{(}.>}[r] & X_n \ar@{.>>}[u]
}%
\qquad%
\xymatrix@!=0.157in{
&                         &                         &                        & X_{n/(n-1)}
\\
&                         &                         & \cdots\ar@{^{(}%
->}[r] & \vdots\ar@{->>}[u] \\
&                         & X_{2/1} \ar@{^{(}->}[r] & \cdots\ar@{^{(}%
->}[r] & X_{n/1} \ar@{->>}[u] \\
& X_{1/0} \ar@{^{(}->}[r] & X_{2/0} \ar@{^{(}->}[r] \ar@{->>}[u] & \cdots
\ar@{^{(}->}[r] & X_{n/0} \ar@{->>}[u] \\
X^{\prime}_0 \ar@{^{(}->}[r] & X^{\prime}_1 \ar@{^{(}->}[r] \ar@
{->>}[u] & X^{\prime}_2 \ar@{^{(}->}[r] \ar@{->>}[u] & \cdots\ar@{^{(}%
->}[r] & X^{\prime}_n \ar@{->>}[u]
}%
\text{,} \label{lkiops4}%
\end{equation}
such that (1) the diagrams commute and agree above the bottom row, (2) every
sequence $X_{i}\hookrightarrow X_{j}\twoheadrightarrow X_{j/i}$ is a
cofibration sequence, (2') every sequence $X_{i}^{\prime}\hookrightarrow
X_{j}^{\prime}\twoheadrightarrow X_{j/i}^{\prime}$ is a cofibration sequence,
(3) every sequence $X_{i/j}\hookrightarrow X_{m/j}\twoheadrightarrow X_{m/i}$
is a cofibration sequence. The face and degeneracy maps amount to duplicating
the $i$-th row and column or deleting them.\medskip\newline We only use the
solid vs. dotted arrows to distinguish the two pieces of the otherwise fully
symmetric pairs. We call the side with solid arrows the \emph{Yin side}, and
the one with dotted arrows the \emph{Yang side}. The morphisms in
$G_{q}(\mathsf{C})$ are all morphisms between such diagrams (i.e. such that
all arrows commute). The cofibrations (resp. weak equivalences) in
$G_{q}(\mathsf{C})$ are those morphisms such that for all objects in the
Diagram \ref{lkiops4} entry-wise it is a cofibration (resp. weak equivalence)
in $\mathsf{C}$.
\end{example}

The main theorem of Gillet and Grayson is that the simplicial set $G_{\bullet
}(\mathsf{C})$ has the same homotopy type as the $K$-theory space
$K(\mathsf{C})$ as defined by Quillen. In particular, we may simply use
the\ Gillet--Grayson model as \textit{the} definition of the $K$-theory space
in this paper.

Working in terms of simplicial homotopy theory, for us, the term `geometric
realization $\left\vert X_{\bullet}\right\vert $' of a simplicial set
$X_{\bullet}$ denotes a functorial fibrant replacement functor. To fix
matters, let us use Kan's $\operatorname*{Ex}^{\infty}$-functor (as we had
already done in Equation \ref{l_Kan}), although the precise nature of the
fibrant replacement will be fully irrelevant for what is to come.

We may now define the $K$-theory space of a (pseudo-additive\footnote{in the
sense of \cite[Definition 2.3]{MR1167575}}) Waldhausen category $\mathsf{C}$
with weak equivalences $w$ by%
\begin{equation}
K(\mathsf{C})=\left\vert wG_{\bullet}\mathsf{C}\right\vert \text{.}
\label{latix5}%
\end{equation}
While writing this in this way is standard accepted practice, this notation
sweeps a bunch of things under the rug, so let us instead give a precise definition:

\begin{definition}
If $\mathsf{C}$ denotes any category, we write $N_{\bullet}\mathsf{C}$ to
denote the nerve of the category.
\end{definition}

If we write a bisimplicial set $X_{\bullet,\bullet}\in ss\mathsf{Set}$ as a
functor%
\[
X_{\bullet,\bullet}:\triangle^{op}\times\triangle^{op}\longrightarrow
\mathsf{Set}\text{,}%
\]
where $\triangle$ is the ordinal number category, then the diagonal simplicial
set $(\operatorname*{diag}X)_{\bullet}$ is defined as the composite functor%
\[
\triangle^{op}\overset{d}{\longrightarrow}\triangle^{op}\times\triangle
^{op}\longrightarrow\mathsf{Set}\text{,}%
\]
where $d$ is the diagonal functor $q\mapsto(q,q)$ for ordinal numbers $q$. We
can make this more concrete: For the simplices we have%
\[
(\operatorname*{diag}X_{\bullet,\bullet})_{q}:=X_{q,q}%
\]
and if we write $\partial_{\ast}^{h}$, $\partial_{\ast}^{v}$ to denote the
horizontal (resp. vertical) face maps of $X_{\bullet,\bullet}$, then%
\begin{equation}
\partial_{i}^{\operatorname*{diag}X}:=\partial_{i}^{h}\circ\partial_{i}%
^{v}\text{.} \label{lkiops3}%
\end{equation}

\begin{definition}
\label{def_KThyViaWaldGG}For every pseudo-additive Waldhausen category
$\mathsf{C}$ with weak equivalences $w$, we call%
\[
K(\mathsf{C}):=\left\vert \operatorname*{diag}N_{\bullet}wG_{\bullet
}(\mathsf{C})\right\vert
\]
the $K$\emph{-theory space} of $\mathsf{C}$. Take $(0,0)$ as the base point.
Equation \ref{latix5} is just a shorthand for the same thing.
\end{definition}

This includes the case where $\mathsf{C}$ is an exact category by using the
Waldhausen category structure of Remark \ref{rmk_ExactToWald}. The above is
the simplicial geometric realization of the diagonal of a bisimplicial
set\ (one simplicial direction comes from the Gillet--Grayson construction,
the other from taking the nerve of the categories $wG_{q}\mathsf{C}$ for any
$q$).

\begin{remark}
If in \S \ref{subsect_Spaces} you prefer doing homotopy theory in
$\mathsf{Top}_{\bullet}$, you need to take the \textquotedblleft
true\textquotedblright\ geometric realization (i.e. in the original meaning of
this term) of this simplicial set to obtain an object in $\mathsf{Top}%
_{\bullet}$ as in Equation \ref{lwaa1}. On the other hand, if you prefer
simplicial sets, Definition \ref{def_KThyViaWaldGG} \textit{is} the space on
the nose.
\end{remark}

\subsubsection{Explicit structure for $K_{0}$\label{subsect_ExplicitK0}}

We extract from Example \ref{example_GG} that a $0$-simplex in $G_{\bullet
}(\mathsf{C})$ corresponds to a pair of objects $(P,Q)$ with $P,Q\in
\mathsf{C}$. Indeed, the concrete isomorphism%
\begin{equation}
\pi_{0}\left\vert G_{\bullet}(\mathsf{C})\right\vert \longrightarrow
K_{0}(\mathsf{C}) \label{lmexi1}%
\end{equation}
is given as follows: Given a connected component on the left, let $(P,Q)$ be
any point in this component and then send it to $[Q]-[P]$ in $K_{0}%
(\mathsf{C})$, i.e. the difference of the isomorphism classes of these
objects. See \cite[Ch. IV, Lemma 9.2]{MR3076731} for a proof.

\subsubsection{Justification of principles (a) and
(b)\label{subsect_JusitfyPrinciplesAB}}

In \S \ref{sect_Overview} our construction of the Tamagawa number rested on
the following basic principles: Suppose $\mathsf{C}$ is an exact
category.\newline$\left.  \qquad\text{\textbf{(a)}}\right.  $ Every object in
$\mathsf{C}$ determines a point in the space $K(\mathsf{C})$.\newline$\left.
\qquad\text{\textbf{(b)}}\right.  $ Every isomorphism $X\overset{\sim
}{\rightarrow}Y$ determines a path from the point of $X$ to $Y$ in
$K(\mathsf{C})$.\newline

We can fully justify them now: For (a) if $P\in\mathsf{C}$ is an object,
simply take the $0$-simplex $(0,P)$ as the point in $K(\mathsf{C})$. Using the
map in Equation \ref{lmexi1} we see that it lies in the connected component of
$[P]\in K_{0}(\mathsf{C})$, so this is in line with our descriptions given in
\S \ref{sect_Overview}. For (b), let us look at the Gillet--Grayson model
again. Unravelling the definition of $G_{1}(\mathsf{C})$ explicitly, the
$1$-simplices turn out to be given by pairs of exact sequences%
\begin{equation}%
\xymatrix{
P_0 \ar@{^{(}.>}[r] & P_1 \ar@{.>>}[r] & P_{1/0} & \qquad& P^{\prime}%
_0 \ar@{^{(}->}[r] & P^{\prime}_1 \ar@{->>}[r] & P_{1/0}
}
\label{lmexid1}%
\end{equation}
having the same cokernel. Using the description of paths in simplicial sets,
such a $1$-simplex is a path from $(P_{0},P_{0}^{\prime})$ to $(P_{1}%
,P_{1}^{\prime})$.

Hence, if $\varphi:X\rightarrow Y$ is an isomorphism in $\mathsf{C}$, attach
the $1$-simplex of the pair of exact sequences
\begin{equation}%
\xymatrix{
0 \ar@{^{(}.>}[r] & 0 \ar@{.>>}[r] & 0 & \qquad& X \ar@{^{(}->}[r]^{\varphi}
& Y \ar@{->>}[r] & 0
}
\label{lmexi4}%
\end{equation}
with matching cokernel zero to it. As discussed above, this is a path from
$(0,X)$ to $(0,Y)$, i.e. a path between the points associated to the objects
$X$ and $Y$ by principle (a).

Thus, principles (a) and (b) are set up rigorously now.

\subsubsection{Justification of principles (a') and
(b')\label{subsect_JusitfyPrinciplesABprime}}

We had also claimed that the same principles hold on the derived level,
essentially. Let $\mathsf{C}$ be any exact category. Let $\mathbf{Ch}%
^{b}(\mathsf{C})$ be the exact category of bounded chain complexes in
$\mathsf{C}$, \cite[\S 10]{MR2606234}. We write $q\mathbf{Ch}^{b}(\mathsf{C})$
to denote the subcategory where we only keep quasi-isomorphisms as morphisms,
\cite[\S 10.3]{MR2606234}. Now make $\mathbf{Ch}^{b}(\mathsf{C})$ a Waldhausen
category as in Remark \ref{rmk_ExactToWald}, but use the class of morphisms
$q$ (i.e. the quasi-isomorphisms) as weak equivalences instead. We note that
this is a pseudo-additive Waldhausen category in the sense of \cite{MR1167575}.

Write $K(\mathsf{C},w)$ if we wish to stress that we use the class of weak
equivalences $w$. By the\ Gillet--Waldhausen theorem (\cite[Chapter V, Theorem
2.2]{MR3076731}), we have the equivalence%
\begin{equation}
K(\mathsf{C},i)\overset{\sim}{\longrightarrow}K(\mathbf{Ch}^{b}(\mathsf{C}%
),q)\text{,} \label{lseip1}%
\end{equation}
which is usually proven in terms of the Waldhausen $S$-construction, but since
the\ Gillet--Grayson model of \cite{MR1167575} can also handle Waldhausen
categories with non-trivial weak equivalences, we obtain%
\[
K(\mathsf{C},i)\overset{\sim}{\longrightarrow}K(\mathbf{Ch}^{b}(\mathsf{C}%
),q)=\left\vert \operatorname*{diag}N_{\bullet}qG_{\bullet}(\mathbf{Ch}%
^{b}\mathsf{C})\right\vert \text{,}%
\]
using Definition \ref{def_KThyViaWaldGG} for the right-hand side of Equation
\ref{lseip1}. Now repeat the constructions of
\S \ref{subsect_JusitfyPrinciplesAB}. We obtain principle (a') since
$0$-simplices in $G_{\bullet}(\mathbf{Ch}^{b}\mathsf{C})$ are pairs $(-,-)$ of
bounded complexes in $\mathsf{C}$ now, and principle (b') since we can now
plug in quasi-isomorphisms for $\varphi$ in Equation \ref{lmexi4}.

\subsubsection{Justification of sum and
negation\label{subsect_JustifySumAndNegation}}

Next, let us set up the maps%
\begin{equation}
\left.  +\right.  :K(\mathsf{C})\times K(\mathsf{C})\longrightarrow
K(\mathsf{C})\qquad\text{and}\qquad\left.  -\right.  :K(\mathsf{C}%
)\longrightarrow K(\mathsf{C}) \label{lwimo1}%
\end{equation}
of\ Equation \ref{lww_AddSubtract}. For the addition in Equation \ref{lwimo1}
let%
\begin{equation}
\left.  \oplus\right.  :\mathsf{C}\times\mathsf{C}\longrightarrow\mathsf{C}
\label{lexit1}%
\end{equation}
be a symmetric monoidal structure giving the coproduct \textquotedblleft$\vee
$\textquotedblright\ (see Remark \ref{rmk_Coproducts}). Then define%
\begin{equation}
(P,Q)+(P^{\prime},Q^{\prime}):=(P\oplus P^{\prime},Q\oplus Q^{\prime})\text{.}
\label{l_maxiu_1}%
\end{equation}
Since $\left.  \oplus\right.  $ is a functor, one can naturally extend this to
a map%
\[
G_{\bullet}\mathsf{C}\times G_{\bullet}\mathsf{C}\longrightarrow G_{\bullet
}\mathsf{C}\text{.}%
\]
As we had pointed out, this map is neither associative nor commutative, and in
fact depends on choosing a concrete bifunctor as in Equation \ref{lexit1}
(since in general coproducts are only well-defined up to unique isomorphism).
However, the above definition is good enough for the moment. We shall later
set up a homotopy correct addition, see Definition \ref{def_SegalNerve}, but
the above operation is one possible representative. In particular, we defer
justifying that this map induces addition to Corollary \ref{cor_SegalNerveSum}
much later. On $K_{0}$ it is easy to check directly, of course.

We define the negation%
\[
\left.  -\right.  :K(\mathsf{C})\longrightarrow K(\mathsf{C})
\]
by simply swapping the Yin and Yang side, i.e. on $0$-simplices this is
$(P,Q)\mapsto(Q,P)$. These maps have all the properties we had discussed in
\S \ref{sect_Overview}, and more concretely:

\begin{proposition}
[\cite{MR909784}]\label{prop_NegationOnGGModel}With these definitions,
$K(\mathsf{C})$ is an $H$-space,

\begin{enumerate}
\item on all homotopy groups $\pi_{i}K(\mathsf{C})$ this addition map induces
the genuine addition of the homotopy group,

\item on all homotopy groups $\pi_{i}K(\mathsf{C})$ this negation map induces
multiplication with $-1$,
\end{enumerate}

and in particular on the level of homotopy groups both operations define an
abelian group structure.
\end{proposition}

This is proven by Gillet and\ Grayson in \cite[Theorem 3.1]{MR909784}. Note
that negation is given by swapping $(P,Q)\mapsto(Q,P)$, fully analogous to
what happens for the group completion $GC(-)$, see Remark
\ref{rmk_MinusOneMapGC}.

\begin{convention}
[Rigorous interpretation of \S \ref{sect_Overview}]%
\label{convention_KThySpace}Use the space $K(\mathbf{Ch}^{b}(\mathsf{C}),q)$
of \S \ref{subsect_JusitfyPrinciplesABprime} as the meaning of the $K$-theory
space for the constructions in \S \ref{sect_Overview}. As we have just
explained, principles (a') and (b') are available, and so are (a) and (b) by
viewing objects as complexes concentrated in degree zero. Along with the sum
and negation, now all the operations employed in \S \ref{sect_Overview} have a
rigorous foundation.
\end{convention}

\begin{remark}
We have just set up $K$-theory as a space and $H$-space here. There is also a
canonical infinite loop space structure. We will not discuss this yet because
it only becomes relevant later, but the reader may jump ahead to Lemma
\ref{lemma_InfLoopSpaceStructures} to see how this structure arises.
\end{remark}

\subsubsection{Explicit structure of $K_{1}$}

The explicit description of $K_{0}$ in \S \ref{subsect_ExplicitK0} can be
complemented by a description of $K_{1}$. Suppose $\mathsf{C}$ is a pointed
exact category. We write $0$ for the designated zero object. A \emph{double
(short) exact sequence} (in the sense of Nenashev) consists of two short exact
sequences%
\[
\mathrm{Yin}:A\overset{p}{\hookrightarrow}B\overset{r}{\twoheadrightarrow
}C\qquad\quad\text{and}\quad\qquad\mathrm{Yang}:A\overset{q}{\hookrightarrow
}B\overset{s}{\twoheadrightarrow}C\text{,}%
\]
whose three objects are the same for Yin and Yang, but the morphisms $p,r$
resp. $q,s$ need not agree. We denote this datum in the format%
\[
l=\left[
\xymatrix{
A \ar@<1ex>@{^{(}->}[r]^{p} \ar@<-1ex>@{^{(}.>}[r]_{q} & B \ar@<1ex>@{->>}%
[r]^{r} \ar@<-1ex>@{.>>}[r]_{s} & C,
}%
\right]
\]
as a shorthand. Given any such $l$, it describes a closed loop around the base
point $(0,0)$ in the Gillet--Grayson model $G_{\bullet}(\mathsf{C})$ which is
made up from the concatenation of three elementary paths (i.e. three
$1$-simplices), namely%
\begin{equation}%
{\includegraphics[
height=1.2246in,
width=1.6873in
]%
{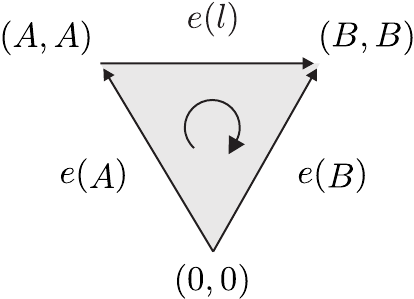}%
}
\label{lmixi3b}%
\end{equation}
where $e(l)$ denotes the $1$-simplex from $(A,A)$ to $(B,B)$ which comes from
interpreting $l$ as a pair of exact sequences with the same cokernel (see
Equation \ref{lmexid1}, where we had already talked about the $1$-simplices).
Moreover, for any object $A\in\mathsf{C}$, $e(A)$ denotes the $1$-simplex from
the pair $0\hookrightarrow A\overset{1}{\twoheadrightarrow}A$, taken both for
Yin and\ Yang, which also defines a pair of exact sequences with the same
cokernel. See also \cite[Definition 3.1]{etnclca2} or \cite{MR1409623,
MR1621690, MR1637539}.

\begin{theorem}
[Nenashev]Suppose $\mathsf{C}$ is an arbitrary exact category. Then the
abelian group $K_{1}(\mathsf{C})$ has the following explicit presentation:

\begin{enumerate}
\item Attach an abstract generator to each double exact sequence%
\[%
\xymatrix{
A \ar@<1ex>@{^{(}->}[r]^{p} \ar@<-1ex>@{^{(}.>}[r]_{q} & B \ar@<1ex>@{->>}%
[r]^{r} \ar@<-1ex>@{.>>}[r]_{s} & C.
}%
\]

\item Whenever the Yin and Yang sides happen to agree, i.e.,%
\[%
\xymatrix{
A \ar@<1ex>@{^{(}->}[r]^{p}_{=} \ar@<-1ex>@{^{(}.>}[r]_{p} & B \ar@
<1ex>@{->>}[r]^{r}_{=} \ar@<-1ex>@{.>>}[r]_{r} & C\text{,}
}%
\]
declare the class of this generator to vanish.

\item Suppose there is a (not necessarily commutative) $(3\times3)$-diagram%
\[%
\xymatrix@W=0.3in@H=0.3in{
A \ar@<1ex>@{^{(}->}[r] \ar@<1ex>@{^{(}.>}[d] \ar@<-1ex>@{^{(}->}%
[d] \ar@<-1ex>@{^{(}.>}[r] & B \ar@<1ex>@{->>}[r] \ar@<-1ex>@{.>>}%
[r] \ar@<1ex>@{^{(}.>}[d] \ar@<-1ex>@{^{(}->}[d] & C \ar@<1ex>@{^{(}%
.>}[d] \ar@<-1ex>@{^{(}->}[d] \\
D \ar@<1ex>@{^{(}->}[r] \ar@<-1ex>@{^{(}.>}[r] \ar@<1ex>@{.>>}[d] \ar@
<-1ex>@{->>}[d] & E \ar@<1ex>@{->>}[r] \ar@<-1ex>@{.>>}[r] \ar@<1ex>@{.>>}%
[d] \ar@<-1ex>@{->>}[d] & F \ar@<1ex>@{.>>}[d] \ar@<-1ex>@{->>}[d] \\
G \ar@<1ex>@{^{(}->}[r] \ar@<-1ex>@{^{(}.>}[r] & H \ar@<1ex>@{->>}%
[r] \ar@<-1ex>@{.>>}[r] & I, \\
}%
\]
whose rows $Row_{i}$ and columns $Col_{j}$ are each a double exact sequence.
Suppose after removing all Yin (resp. all Yang) exact sequences, the remaining
diagram commutes. Whenever this holds, impose the relation%
\begin{equation}
Row_{1}-Row_{2}+Row_{3}=Col_{1}-Col_{2}+Col_{3}\text{.} \label{l_C_Nenashev}%
\end{equation}

\end{enumerate}
\end{theorem}

\begin{example}
\label{example_MakeAutToNenashevRepresentative}We use the same notation as in
the paper \cite{etnclca2}, which is a little different from the one in
Nenashev's papers. If $\varphi:X\rightarrow X$ is an automorphism of an object
in $\mathsf{C}$, the canonical map $\operatorname*{Aut}(X)\rightarrow
K_{1}(\mathsf{C})$ sends it to the Nenashev representative%
\begin{equation}%
\xymatrix{
0 \ar@<1ex>@{^{(}->}[r] \ar@<-1ex>@{^{(}.>}[r] & X \ar@<1ex>@{->>}%
[r]^{\varphi} \ar@<-1ex>@{.>>}[r]_{1} & X
}%
\text{.} \label{lmixi3a}%
\end{equation}

\end{example}

\subsection{\label{sect_Setup}Equivariance setup}

Let $F$ be a number field and $\mathcal{O}_{F}$ its ring of integers. Let $A$
be a finite-dimensional semisimple $F$-algebra and $\mathfrak{A}$ an
$\mathcal{O}_{F}$-order. This assumption implies that $\mathcal{O}_{F}$ lies
in the center of $\mathfrak{A}$.

These assumptions in place, $A$ is also a finite-dimensional semisimple
$\mathbb{Q}$-algebra and $\mathfrak{A}$ a $\mathbb{Z}$-order, i.e. an order in
the usual sense.

For any place $\mathfrak{p}$ of $\mathcal{O}_{F}$, we write $A_{\mathfrak{p}%
}:=A\otimes_{F}F_{\mathfrak{p}}$, where $F_{\mathfrak{p}}$ is the local field
at $\mathfrak{p}$. If $\mathfrak{p}$ is a finite place, define $\mathfrak{A}%
_{\mathfrak{p}}:=\mathfrak{A}\otimes_{\mathcal{O}_{F}}\mathcal{O}%
_{F_{\mathfrak{p}}}$, where $\mathcal{O}_{F_{\mathfrak{p}}}$ is the ring of
integers of $F_{\mathfrak{p}}$. For any finitely generated projective right
$\mathfrak{A}$-module $\mathfrak{X}$, we introduce the shorthands%
\[
\mathfrak{X}_{\mathfrak{p}}:=\mathfrak{X}\otimes_{\mathfrak{A}}\mathfrak{A}%
_{\mathfrak{p}}\text{,}\qquad\qquad X_{\mathfrak{p}}:=\mathfrak{X}%
\otimes_{\mathfrak{A}}A_{\mathfrak{p}}\text{,}\qquad\qquad X:=\mathfrak{X}%
\otimes_{\mathfrak{A}}A\text{.}%
\]
Here we tacitly equip $X$ with the discrete topology, i.e. its natural
topology as a finite-dimensional $\mathbb{Q}$-vector space. We equip
$X_{\mathfrak{p}}$ with its natural topology as a finite-dimensional
$F_{\mathfrak{p}}$-vector space. We also write%
\[
X_{\mathbb{R}}:=X\otimes_{A}A_{\mathbb{R}}=X\otimes_{\mathbb{Q}}%
\mathbb{R}\text{,}%
\]
equipped with the real vector space topology.

\begin{remark}
In the special case $F=\mathbb{Q}$ this is compatible with the notation used
by Burns and Flach, see \cite[\S 2.7]{MR1884523}. In particular, it is
compatible with the notation in \cite{etnclca} and \cite{etnclca2}.
\end{remark}

We also use the notation%
\[
\widehat{A}:=A\otimes_{\mathbb{Z}}\widehat{\mathbb{Z}}=\mathfrak{A}%
\otimes_{\mathbb{Z}}\mathbb{A}_{fin}\text{.}%
\]
Here $\mathbb{A}_{fin}$ denotes the \textit{finite part }of the ad\`{e}les of
the rational number field $\mathbb{Q}$, i.e. the restricted product $\left.
\prod\nolimits_{p}^{\prime}\right.  (\mathbb{Q}_{p},\mathbb{Z}_{p})$, where
$p$ only runs over the primes. We equip $\widehat{A}$ with the locally compact
topology coming from the ad\`{e}les. Note that%
\[
\widehat{F}=F\otimes_{\mathbb{Z}}\widehat{\mathbb{Z}}%
\]
agrees with the finite part of the ad\`{e}les of the number field $F$.

Next, we discuss id\`{e}les in the non-commutative setting, following
Fr\"{o}hlich \cite[\S 2]{MR0376619}. Define the \emph{id\`{e}le group} by%
\begin{equation}
J(A):=\left\{  (a_{\mathfrak{p}})_{\mathfrak{p}}\in\left.  \prod
_{\mathfrak{p}}A_{\mathfrak{p}}^{\times}\right\vert a_{\mathfrak{p}}%
\in\mathfrak{A}_{\mathfrak{p}}^{\times}\text{ for all but finitely many places
}\mathfrak{p}\right\}  \text{.} \label{lmixi1}%
\end{equation}
This group is independent of the choice of the order $\mathfrak{A}$ (because
if $\mathfrak{A}^{\prime}$ is a further $\mathcal{O}_{F}$-order in $A$, we
have $\mathfrak{A}_{\mathfrak{p}}=(\mathfrak{A}^{\prime})_{\mathfrak{p}}$ for
all but finitely many places $\mathfrak{p}$). If $\mathfrak{A}\subset A$ is a
fixed order, we also get the group of \emph{unit finite id\`{e}les
}$U(\mathfrak{A})\subset J(A)$ defined by%
\begin{equation}
U(\mathfrak{A}):=\left\{  (a_{\mathfrak{p}})_{\mathfrak{p}}\in\prod
_{\mathfrak{p}\text{ finite}}\mathfrak{A}_{\mathfrak{p}}^{\times}\right\}
\qquad\text{(also denoted }U^{\operatorname*{fin}}(\mathfrak{A})\text{)}
\label{lcixiDex1}%
\end{equation}
This group depends crucially on the choice of the order $\mathfrak{A}$. We
view this as a subgroup of $J(A)$ by letting $a_{\mathfrak{p}}=1$ for all
infinite places $\mathfrak{p}$.

\begin{remark}
\label{rmk_FroehlichU_DifferentDef}This differs from \cite{MR0376619}, because
Fr\"{o}hlich's definition of $U(\mathfrak{A})$ includes the infinite places,
so what he calls $U$ is $U(\mathfrak{A})\cdot A_{\mathbb{R}}^{\times}$ in our notation.
\end{remark}

The \emph{ideal class group} of $\mathfrak{A}$ can be defined as%
\begin{equation}
\operatorname*{Cl}(\mathfrak{A}):=\ker\left(  \operatorname*{rk}%
:K_{0}(\mathfrak{A})\longrightarrow\mathbb{Z}\right)  \text{,} \label{lmixi7}%
\end{equation}
where \textquotedblleft$\operatorname*{rk}$\textquotedblright\ denotes the
rank map. If $A$ is a number field and $\mathfrak{A}$ any order,
$\operatorname*{Cl}(\mathfrak{A})$ agrees with the usual ideal class group. To
see this, use \cite[Ch. II, Corollary 2.6.3]{MR3076731}.

There is a reduced norm%
\[
\operatorname*{nr}:A^{\times}\longrightarrow\zeta(A)^{\times}\text{,}%
\]
where $\zeta(A)$ denotes the center of $A$, see \cite[\S 2]{MR0376619} (if $A$
is not simple, define it by taking the direct sum of the reduced norms of each
simple summand). We define the \emph{reduced norm one subgroup}%
\[
J^{1}(A):=\ker\left(  \operatorname*{nr}:J(A)\longrightarrow\zeta(A)^{\times
}\right)  \text{.}%
\]
As $\zeta(A)^{\times}$ is abelian, it follows that $[J(A),J(A)]\subseteq
J^{1}(A)$, i.e. the commutator subgroup is contained in $J^{1}(A)$. Next, we
recall the classification of finitely generated projective right
$\mathfrak{A}$-modules. Given any $a=(a_{\mathfrak{p}})_{\mathfrak{p}}\in
J(A)$, there exists a unique $\mathcal{O}_{F}$-lattice, denoted
\textquotedblleft$a\mathfrak{A}$\textquotedblright, inside $A$ such that%
\begin{equation}
(a\mathfrak{A})_{\mathfrak{p}}=a_{\mathfrak{p}}\mathfrak{A}_{\mathfrak{p}}
\label{lciops1}%
\end{equation}
for all finite places $\mathfrak{p}$ of $F$. By a result of Fr\"{o}hlich
\cite[Theorem 1]{MR0376619}, every finitely generated projective right
$\mathfrak{A}$-module $\mathfrak{X}$ of rank $m\geq1$ is isomorphic to%
\begin{equation}
\mathfrak{X}\cong a_{1}\mathfrak{A}\oplus\cdots\oplus a_{m}\mathfrak{A}
\label{lmixi4}%
\end{equation}
for some $a_{1},\ldots,a_{m}\in J(A)$, and further any two such are isomorphic
if and only if they have (a) the same rank and, (b) moreover $a_{1}\cdots
a_{m}\equiv a_{1}^{\prime}\cdots a_{m}^{\prime}$ holds in the double coset set
$A^{\times}\backslash J(A)/(U(\mathfrak{A})\cdot A_{\mathbb{R}}^{\times})$ for
$m=1$, resp. in the quotient group%
\begin{equation}
\frac{J(A)}{J^{1}(A)\cdot A^{\times}\cdot U(\mathfrak{A})\cdot A_{\mathbb{R}%
}^{\times}} \label{lmixiDex1}%
\end{equation}
in the case of $m\geq2$. (Keep in mind Remark
\ref{rmk_FroehlichU_DifferentDef} when comparing this with \cite{MR0376619}.)

\begin{example}
\label{example_CancellationRankTwo}For id\`{e}les $a=(a_{\mathfrak{p}%
})_{\mathfrak{p}}$ and $b=(b_{\mathfrak{p}})_{\mathfrak{p}}$, there exists an
isomorphism $a\mathfrak{A}\oplus b\mathfrak{A}\overset{\sim}{\longrightarrow
}ab\mathfrak{A}\oplus\mathfrak{A}$.
\end{example}

\begin{example}
\label{example_IdeleNontrivOnlyAtInfinity}If an id\`{e}le $a=(a_{\mathfrak{p}%
})_{\mathfrak{p}}$ satisfies $a_{\mathfrak{p}}=1$ for all finite places, then
$a\mathfrak{A}=\mathfrak{A}$.
\end{example}

Since $J^{1}(A)$ contains the commutator subgroup, Equation \ref{lmixiDex1}
describes an abelian group. The classification result generalizes Steinitz's
classification of vector bundles over affine Dedekind schemes. Based on this
result, Fr\"{o}hlich obtains a second characterization of the ideal class
group of Equation \ref{lmixi7}.

\begin{theorem}
[Fr\"{o}hlich]\label{thm_FrohlichTheory}There exists an isomorphism%
\begin{align}
\operatorname*{Cl}(\mathfrak{A})  &  \longrightarrow\frac{J(A)}{J^{1}(A)\cdot
A^{\times}\cdot U(\mathfrak{A})\cdot A_{\mathbb{R}}^{\times}}\label{lmixi5}\\
\lbrack\mathfrak{X}]-[\mathfrak{A}^{n}]  &  \longmapsto\lbrack a_{1}\cdots
a_{m}]\text{,}\nonumber
\end{align}
where we take any presentation of the module $\mathfrak{X}$ as in Equation
\ref{lmixi4} for any $m\geq2$ and $n:=\operatorname*{rk}(\mathfrak{A})$. Since
all classes on the left-hand side are represented by rank one modules, the map
is uniquely determined by declaring $[a\mathfrak{A}]-[\mathfrak{A}%
]\mapsto\lbrack a]$, using the notation $a\mathfrak{A}$ of Equation
\ref{lciops1}.
\end{theorem}

See \cite[Consequence \textquotedblleft II\textquotedblright\ of Theorem
1]{MR0376619} (and again keep in mind Remark \ref{rmk_FroehlichU_DifferentDef}).

\section{Additivity and its consequences}

Let us recall the Additivity Theorem of algebraic $K$-theory. Suppose
$\mathsf{C}$, $\mathsf{D}$ are exact categories and $f_{i}:\mathsf{C}%
\rightarrow\mathsf{D}$ for $i=1,2,3$ are exact functors such that for each
object $C\in\mathsf{C}$ we get a short exact sequence $f_{1}(C)\hookrightarrow
f_{2}(C)\twoheadrightarrow f_{3}(C)$, and this short exact sequence is
functorial in $C$. A more elegant and precise way to set this up is to write
$\mathcal{E}\mathsf{D}$ for the exact category of exact sequences in
$\mathsf{D}$, \cite[Exercise 3.9]{MR2606234} and consider $(f_{i})_{i=1,2,3}$
a single exact functor $\mathsf{C}\rightarrow\mathcal{E}\mathsf{D}$.%
\begin{equation}%
\xymatrix{
& C \ar[dl]_{f_1} \ar[d]^{f_2} \ar[dr]^{f_3} \\
f_1(C) \ar@{^{(}->}[r] & f_2(C) \ar@{->>}[r] & f_3(C)
}
\label{lmisu6}%
\end{equation}

\begin{theoremannounce}
[Additivity]For every exact functor $(f_{i})_{i=1,2,3}:\mathsf{C}%
\rightarrow\mathcal{E}\mathsf{D}$ we have $f_{2\ast}=f_{1\ast}+f_{3\ast}$,
where $f_{i\ast}:K(\mathsf{C})\rightarrow K(\mathsf{D})$ denotes the map
induced from the exact functor $f_{i}$.
\end{theoremannounce}

See \cite[Ch. V, Theorem 1.2]{MR3076731}. First of all, we deduce the
following standard vanishing theorem.

\begin{lemma}
[Eilenberg swindle]\label{lemma_EilenbergSwindle}If an exact category
$\mathsf{C}$ is closed under countable products (or under countable
coproducts), then $K(\mathsf{C})=0$.
\end{lemma}

\begin{proof}
For example given in \cite[Lemma 4.2]{obloc}, but since the proof illustrates
how to use Additivity in a powerful way, we repeat the full argument here:
Suppose $\mathsf{C}$ is closed under coproducts. Use the exact functor
$\mathsf{C}\rightarrow\mathcal{E}\mathsf{C}$ sending any object $X$ to the
exact sequence $X\hookrightarrow\bigoplus_{i\in\mathbb{N}}X\overset
{s}{\twoheadrightarrow}\bigoplus_{i\in\mathbb{N}}X$, where the map $s$ sends
the $i$-th factor to the $(i-1)$-th for $i\geq1$. Naming these functors
$f_{1},f_{2},f_{3}$ as in Diagram \ref{lmisu6}, we obtain $\operatorname*{id}%
_{\mathsf{C}\ast}=f_{2\ast}-f_{3\ast}$, but $f_{2}=f_{3}$, showing that the
identity map agrees with the zero map, forcing our claim to hold. If
$\mathsf{C}$ is closed under products instead, use the same sequence, but with
products instead.
\end{proof}

We can now prove several fundamental theorems solely on the basis of
Additivity and topological considerations in the category $\mathsf{LCA}%
_{\mathfrak{A}}$. In particular, at this point we will do a few things which
hinge mostly on topology, and far less on the underlying algebraic right
$\mathfrak{A}$-module structure of objects.

\begin{theorem}
[Local Triviality]\label{thm_LocalTriviality}Let $F$ be a number field and
$\mathfrak{A}$ an order in a finite-dimensional semisimple $F$-algebra $A$.
Suppose $\mathfrak{p}$ is a finite place of $F$. Let $\mathfrak{O}$ be any
order in $A_{\mathfrak{p}}$ (for example $\mathfrak{A}_{p}$ or the maximal
order). Then the composition%
\[
K(\mathfrak{O})\longrightarrow K(A_{\mathfrak{p}})\longrightarrow
K(\mathsf{LCA}_{\mathfrak{A}})
\]
is zero. Here the first arrow is induced from the ring inclusion
$\mathfrak{O}\subset A_{\mathfrak{p}}$, while the latter sends
$A_{\mathfrak{p}}$ to itself, but equipped with the locally compact topology.
\end{theorem}

\begin{proof}
We define an exact functor $p:\operatorname*{PMod}(\mathfrak{O})\rightarrow
\mathcal{E}\mathsf{LCA}_{\mathfrak{A}}$. We define it on the projective
generator $\mathfrak{O}$ of the category $\operatorname*{PMod}(\mathfrak{O})$
by sending it to%
\[
\mathfrak{O}\hookrightarrow A_{\mathfrak{p}}\twoheadrightarrow A_{\mathfrak{p}%
}/\mathfrak{O}%
\]
in $\mathsf{LCA}_{\mathfrak{A}}$. Let $p$ be the residual characteristic of
the local field $Z_{\mathfrak{p}}$. Then $\mathfrak{O}$ carries the topology
of a finite rank free $\mathcal{O}_{\mathfrak{p}}$-module, where
$\mathcal{O}_{\mathfrak{p}}$ is the ring of integers of $Z_{\mathfrak{p}}$,
$A_{\mathfrak{p}}$ carries the topology of a finite-dimensional
$Z_{\mathfrak{p}}$-vector space and the quotient $A_{\mathfrak{p}%
}/\mathfrak{O}$ is seen to necessarily carry the discrete topology. Note that
$\mathfrak{O}$ is compact. The Additivity Theorem implies that $p_{2\ast
}=p_{1\ast}+p_{3\ast}$, where $p_{1},p_{2},p_{3}$ denote the exact functors to
the left (resp. middle, resp. right) entry of the short exact sequence. Since
$p_{1\ast}$ and $p_{3\ast}$ factor over $\mathsf{LCA}_{\mathfrak{A},D}$ resp.
$\mathsf{LCA}_{\mathfrak{A},C}$, both of which have zero $K$-theory (Lemma
\ref{lemma_EilenbergSwindle}, use that arbitrary direct sums of discrete
groups are discrete, and arbitrary products of compact groups are compact by
Tychonoff's Theorem), it follows that $p_{2\ast}=0+0$.
\end{proof}

\begin{remark}
Of course it would have been sufficient to prove this with $\mathfrak{O}$ the
unique maximal order and use that every order is contained in it. However, the
way we present the proof above it is particularly clear that all which is
really used is the compactness of $\mathfrak{O}$ and the discreteness of the
respective cokernel, so the above proof is in a way simpler since it does not
even use the algebraic theory of orders in semisimple algebras.
\end{remark}

In degree one this has the following important consequence.

\begin{corollary}
\label{cor_LocalTrivInDegreeOne}Let $F$ be a number field and $\mathfrak{A}$
an order in a finite-dimensional semisimple $F$-algebra $A$. Suppose
$\mathfrak{p}$ is a finite place of $F$. Then the composition%
\[
\mathfrak{A}_{\mathfrak{p}}^{\times}\longrightarrow A_{\mathfrak{p}}^{\times
}\longrightarrow K_{1}(\mathsf{LCA}_{\mathfrak{A}})
\]
is zero.
\end{corollary}

We recall the reciprocity law, \cite[Theorem 13.1]{etnclca}:

\begin{theorem}
[Reciprocity Law]\label{thm_reciprocity_law}Let $F$ be a number field and
$\mathfrak{A}$ an order in a finite-dimensional semisimple $F$-algebra $A$.
Then the composition%
\begin{equation}
K(A)\longrightarrow K(\widehat{A})\oplus K(A_{\mathbb{R}})\overset
{\operatorname*{sum}}{\longrightarrow}K(\mathsf{LCA}_{\mathfrak{A}})
\label{lmits1}%
\end{equation}
is zero.
\end{theorem}

Here the first map stems from the exact functor $X\mapsto(X\otimes_{A}%
\widehat{A}\,,\,X\otimes_{\mathbb{Z}}\mathbb{R)}$. The second map sends an
$\widehat{A}$-module to itself, but equipped with the adelic topology, and
maps a free right $A_{\mathbb{R}}$-module to itself, equipped with the real
vector space topology.

\begin{remark}
[Signs]\label{rem_SignInReciprocityLaw}In Theorem \ref{thm_reciprocity_law} we
really mean the sum map on the right, and not the difference.
\end{remark}

\begin{corollary}
\label{cor_RecipLawInDegreeOne}Let $F$ be a number field and $\mathfrak{A}$ an
order in a finite-dimensional semisimple $F$-algebra $A$. Then the composition%
\[
A^{\times}\longrightarrow J(A)\longrightarrow K_{1}(\mathsf{LCA}%
_{\mathfrak{A}})
\]
is zero.
\end{corollary}

\section{Noncommutative id\`{e}les I}

In this section we shall establish an id\`{e}le presentation of the group
$K_{1}(\mathsf{LCA}_{\mathfrak{A}})$. We first prove an analogous result using
$K_{1}$-id\`{e}les and then use reduced norms to translate this into the claim
which we want to prove. This is analogous to the proof of Fr\"{o}hlich's
id\`{e}le presentation, Equation \ref{lmixi5} in Curtis--Reiner \cite[\S 49A]%
{MR892316}. It originates from ideas of Wilson \cite{MR0447211}.

For auxiliary purposes, we define the $K_{1}$-analogue of the id\`{e}le group,%
\begin{equation}
JK_{1}(A):=\left\{  (a_{\mathfrak{p}})_{\mathfrak{p}}\in\left.  \prod
_{\mathfrak{p}}K_{1}(A_{\mathfrak{p}})\right\vert a_{\mathfrak{p}}%
\in\operatorname*{im}K_{1}(\mathfrak{A}_{\mathfrak{p}})\text{ for all but
finitely many places }\mathfrak{p}\right\}  \text{,} \label{lzal1}%
\end{equation}
where $\mathfrak{p}$ runs over all places, finite and infinite, and the
condition on $a_{\mathfrak{p}}$ is considered satisfied if $\mathfrak{p}$ is
an infinite place. The image in \textquotedblleft$\operatorname*{im}%
K_{1}(\mathfrak{A}_{\mathfrak{p}})$\textquotedblright\ refers to the natural
map $K_{1}(\mathfrak{A}_{p})\longrightarrow K_{1}(A_{\mathfrak{p}})$, which in
general need not be injective. The definition of $JK_{1}(A)$ does not depend
on the choice of the order $\mathfrak{A}$, for the same reason as in the
definition of $J(A)$. If $\mathfrak{A}^{\prime}$ is a further order, we have
$\mathfrak{A}_{\mathfrak{p}}=(\mathfrak{A}^{\prime})_{\mathfrak{p}}$ for all
but finitely many places. Next, we define%
\begin{equation}
UK_{1}^{\operatorname{fin}}(\mathfrak{A}):=\left\{  (a_{\mathfrak{p}%
})_{\mathfrak{p}}\in\prod_{\mathfrak{p}\text{ finite}}K_{1}(\mathfrak{A}%
_{\mathfrak{p}})\right\}  \text{.} \label{lmixi5a}%
\end{equation}
These definitions roughly match the ones in Curtis--Reiner
\cite[(49.16)\ Proposition]{MR892316} and Wilson \cite{MR0447211}, except that
we also include the infinite places.

\begin{definition}
Suppose $\mathfrak{p}$ is any place of $F$. Define%
\begin{equation}
\tilde{\xi}_{\mathfrak{p}}:K_{1}(A_{\mathfrak{p}})\longrightarrow
K_{1}(\mathsf{LCA}_{\mathfrak{A}})\text{,} \label{lmixi6}%
\end{equation}
based on the exact functor sending a finitely generated projective right
$A_{\mathfrak{p}}$-module to itself, equipped with its natural locally compact
topology (i.e. the $\mathbb{Q}_{p}$-vector space topology if $\mathfrak{p}$ is
a finite place over the prime $p$, or the $\mathbb{R}$-vector space topology
if $\mathfrak{p}$ is an infinite place).
\end{definition}

We can make this map explicit in the Nenashev presentation: By Example
\ref{example_MakeAutToNenashevRepresentative} the natural morphism below on
the left%
\begin{equation}
A_{\mathfrak{p}}^{\times}\rightarrow K_{1}(A_{\mathfrak{p}})\text{,}\qquad
a\mapsto\left[
\xymatrix{
0 \ar@<1ex>@{^{(}->}[r]^-{0} \ar@<-1ex>@{^{(}.>}[r]_-{0} & {A_{\mathfrak{p}}}
\ar@<1ex>@{->>}[r]^{\cdot a} \ar@<-1ex>@{.>>}[r]_{1} & {{A_{\mathfrak{p}}}}
}%
\right]  \label{lmixi10_1}%
\end{equation}
is given in terms of the Nenashev presentation by the double exact sequence
above on the right, and moreover this map is an isomorphism since
$A_{\mathfrak{p}}$ is semisimple. Use the same Nenashev representative for its
image in $K_{1}(\mathsf{LCA}_{\mathfrak{A}})$, just additionally equipped with
the natural topology.

\begin{proposition}
\label{Prop_IdentifyK1LCA_Use_K1Ideles}Let $\mathfrak{A}$ be a regular order
in a finite-dimensional semisimple $\mathbb{Q}$-algebra $A$. The map%
\[
\tilde{\xi}:\frac{JK_{1}(A)}{\operatorname*{im}K_{1}(A)+\operatorname*{im}%
UK_{1}^{\operatorname{fin}}(\mathfrak{A})}\overset{\sim}{\longrightarrow}%
K_{1}(\mathsf{LCA}_{\mathfrak{A}})
\]
given by $\tilde{\xi}_{\mathfrak{p}}$ on all factors in the restricted product
in Equation \ref{lzal1}, induces an isomorphism.
\end{proposition}

We shall see how the possibility to quotient out the image of $K_{1}(A)$ comes
precisely from the Reciprocity Law, Theorem \ref{thm_reciprocity_law}, while
quotienting out the image of $UK_{1}^{\operatorname{fin}}(\mathfrak{A})$ stems
from Local Triviality, Theorem \ref{thm_LocalTriviality}. The key point in the
proof of the proposition is to show that these account for the entire kernel
of $\tilde{\xi}$.\medskip

We split the proof into a series of individual verifications.

\begin{lemma}
\label{mz1}The map $\tilde{\xi}$ is well-defined.
\end{lemma}

\begin{proof}
Observe that%
\[
JK_{1}(A)=\underset{S}{\underrightarrow{\operatorname*{colim}}}\left(
\operatorname*{im}UK_{1}^{\operatorname{fin}}(\mathfrak{A})+\bigoplus
_{\mathfrak{p}\in S}K_{1}(A_{\mathfrak{p}})\right)  \text{,}%
\]
where $S$ runs over all finite subsets of places of $Z$, partially ordered by
inclusion, and we understand the sum in the big round brackets as the subgroup
generated inside $JK_{1}(A)$ by these subgroups. Hence, in order to define
$\tilde{\xi}$ it suffices to define it on each $\operatorname*{im}%
UK_{1}^{\operatorname{fin}}(\mathfrak{A})+\bigoplus_{\mathfrak{p}\in S}%
K_{1}(A_{\mathfrak{p}})$ in a way compatible with replacing $S$ by a bigger
finite set. We define%
\[
\tilde{\xi}_{S}:u\cdot\prod_{\mathfrak{p}\in S}a_{\mathfrak{p}}\mapsto
\prod_{\mathfrak{p}\in S}\tilde{\xi}_{\mathfrak{p}}(a_{\mathfrak{p}}%
)\qquad\text{for}\qquad u\in\operatorname*{im}UK_{1}^{\operatorname{fin}%
}(\mathfrak{A})\text{, }a_{\mathfrak{p}}\in K_{1}(A_{\mathfrak{p}})\text{.}%
\]
We claim that $\tilde{\xi}_{S}$ is well-defined: We only need to show that the
intersection%
\[
\operatorname*{im}UK_{1}^{\operatorname{fin}}(\mathfrak{A})\cap\left(
\bigoplus_{\mathfrak{p}\in S}K_{1}(A_{\mathfrak{p}})\right)  =\bigoplus
_{\mathfrak{p}\in S}\operatorname*{im}K_{1}(\mathfrak{A}_{\mathfrak{p}})
\]
gets sent to zero. However, this follows from Local Triviality, Theorem
\ref{thm_LocalTriviality}. Thus, $\tilde{\xi}:=$ $\operatorname*{colim}%
_{S}\tilde{\xi}_{S}$, verifying that we get a well-defined map%
\begin{equation}
\frac{JK_{1}(A)}{\operatorname*{im}UK_{1}^{\operatorname{fin}}(\mathfrak{A}%
)}\longrightarrow K_{1}(\mathsf{LCA}_{\mathfrak{A}})\text{.} \label{lmixi2}%
\end{equation}
Given any $a\in K_{1}(A)$, we obtain that $\tilde{\xi}(a)=\tilde{\xi}_{S}(a)$
for $S$ sufficiently big. Hence, by the fundamental Reciprocity Law, Theorem
\ref{thm_reciprocity_law}, we have $\xi(a)=0$. Thus, the morphism set up in
Equation \ref{lmixi2} descends to the quotient modulo $\operatorname*{im}%
K_{1}(A)+\operatorname*{im}UK_{1}^{\operatorname{fin}}(\mathfrak{A})$.
\end{proof}

Next, we set up a commutative diagram%
\begin{equation}%
\xymatrix{
& K_2(\mathsf{LCA}_{\mathfrak{A} }) \ar[d] \\
K_1(\mathfrak{A}) \ar[d]_{\gamma} \ar[r]^{1} \ar@{}[dr]|{X}
& K_1(\mathfrak{A}) \ar[d] \\
K_1(A) \oplus K_1(A_{\mathbb{R}}) \ar[r]^{\operatorname{pr}_2 } \ar@
{->}[d]_{\alpha} \ar@{}[dr]|{Y} & K_1(A_{\mathbb{R}}) \ar[d] \\
\frac{JK_{1}(A)}{\operatorname*{im}UK^{\operatorname{fin}}_{1}(\mathfrak{A})}
\ar[r]^{\tilde{\xi} }
\ar@{->>}[d]_{j} \ar@{}[dr]|{Z} & K_1(\mathsf{LCA}_{\mathfrak{A} }) \ar@
{->>}[d] \\
\frac{JK_{1}(A)}{\operatorname*{im}K_1(A) + \operatorname*{im}%
UK^{\operatorname{fin}}_{1}(\mathfrak{A}) + \operatorname*{im}K_1(A_{\mathbb
{R}})} \ar[r]_-{w}
& \operatorname{Cl}(\mathfrak{A}),
}
\label{lfigA1}%
\end{equation}
also in several steps. The morphism $\alpha$ is the difference of the natural
maps (we elaborate on the precise definition in the proof),
$\operatorname*{pr}_{2}$ denotes the projection to the second summand. The
bottom horizontal map $w$ amounts to Fr\"{o}hlich's id\`{e}le description of
the class group, Equation \ref{lmixi5}.

\begin{lemma}
\label{mz2}The bottom horizontal map $w$ is an isomorphism.
\end{lemma}

\begin{proof}
This is \cite[(49.16) Proposition]{MR892316}. Loc. cit. Curtis and Reiner
define $JK_{1}(A)$ without the infinite places. However, since we additionally
quotient out by the infinite place contribution $\operatorname*{im}%
K_{1}(A_{\mathbb{R}})$, this difference gets remedied.
\end{proof}

\begin{lemma}
\label{mz3}The columns in Figure \ref{lfigA1} are exact.
\end{lemma}

\begin{proof}
(Step 1) In the right column, we use the long exact sequence of \cite[Theorem
11.3]{etnclca}. This is the only input of the proof which uses the assumption
that $\mathfrak{A}$ is a regular order. This sequence terminates in%
\begin{equation}
\cdots\longrightarrow K_{1}(\mathsf{LCA}_{\mathfrak{A}})\overset
{c}{\longrightarrow}K_{0}(\mathfrak{A})\overset{a}{\longrightarrow}%
K_{0}(A_{\mathbb{R}})\longrightarrow K_{0}(\mathsf{LCA}_{\mathfrak{A}%
})\longrightarrow0\text{.} \label{lmixi8}%
\end{equation}
We know that $K_{0}(A_{\mathbb{R}})\cong\mathbb{Z}^{n}$ for some $n$ since
$A_{\mathbb{R}}$ is semisimple (and concretely $n$ is the number of factors in
the Artin--Wedderburn decomposition), and $\operatorname*{Cl}(\mathfrak{A}%
)\hookrightarrow K_{0}(\mathfrak{A})$\ is the kernel of the rank map. It
follows that $\ker(a)$, in the notation introduced in Equation \ref{lmixi8},
contains at most the subgroup $\operatorname*{Cl}(\mathfrak{A})$. On the other
hand, by the Jordan--Zassenhaus theorem \cite[(26.4) Theorem]{MR1972204} the
class group $\operatorname*{Cl}(\mathfrak{A})$ is finite, but $K_{0}%
(A_{\mathbb{R}})\cong\mathbb{Z}^{n}$ is torsion-free, so $\operatorname*{Cl}%
(\mathfrak{A})$ is contained in the kernel. We deduce $\ker
(a)=\operatorname*{Cl}(\mathfrak{A})$, and by the exactness of Equation
\ref{lmixi8}, we have $\operatorname*{im}(c)=\operatorname*{Cl}(\mathfrak{A}%
)$. This yields the truncation of the exact sequence which we use as the right
column.\newline(Step 2) The left column is set up as follows: The map $\gamma$
is just the sum of the natural maps coming from the ring homomorphisms
$\mathfrak{A}\rightarrow A$ and $\mathfrak{A}\rightarrow A_{\mathbb{R}}$.
Analogously, $\alpha$ is the difference of the identity map $K_{1}%
(A_{\mathbb{R}})\rightarrow K_{1}(A_{\mathbb{R}})$ for the infinite places,
minus the diagonal map%
\[
K_{1}(A)\longrightarrow JK_{1}(A)\text{,}\qquad\qquad a\longmapsto
(a,a,\ldots)
\]
(involving all places of $F$, even the infinite ones). The composition is
zero: Given any $\alpha\in K_{1}(\mathfrak{A})$, in the factors of $JK_{1}(A)$
corresponding to infinite places we subtract the same values, so it is zero at
the infinite places. At each finite place $\mathfrak{p}$, we have the
factorization $\mathfrak{A}\longrightarrow\mathfrak{A}_{\mathfrak{p}%
}\longrightarrow A_{p}$, showing that the image of this contribution comes
from $UK_{1}^{\operatorname{fin}}(\mathfrak{A})$.\newline(Step 3) It is also
exact at this position. Assume we are given $(x,y)\in K_{1}(A)\oplus
K_{1}(A_{\mathbb{R}})$ such that $\alpha(x,y)=0$. Firstly, this means that for
every finite place $\mathfrak{p}$ the image of $x$ under $K_{1}(A)\rightarrow
K_{1}(A_{\mathfrak{p}})$ lies in the image $\operatorname*{im}K_{1}%
(\mathfrak{A}_{\mathfrak{p}})$. Collecting this data for all finite places, we
find $x^{\prime}\in K_{1}(\widehat{\mathfrak{A}})$ such that $(x,x^{\prime})$
maps to zero in the Wall exact sequence%
\[
K_{1}(\mathfrak{A})\overset{\operatorname*{diag}}{\longrightarrow}%
\underset{(x,x^{\prime})}{K_{1}(A)\oplus K_{1}(\widehat{\mathfrak{A}}%
)}\overset{\operatorname*{diff}}{\longrightarrow}K_{1}(\widehat{A}%
)\longrightarrow K_{0}(\mathfrak{A})\longrightarrow\cdots\text{,}%
\]
see \cite[(42.19)\ Theorem]{MR892316}. By the exactness of this sequence, we
learn that $x=x^{\prime}\in K_{1}(\mathfrak{A})$. For the infinite places,
$\alpha(x,y)=0$ now just means that $y\in K_{1}(A_{\mathbb{R}})$ also agrees
with the image of $x$ under the map $\mathfrak{A}\rightarrow A_{\mathbb{R}}$.
However, this means that $(x,y)$ is diagonal coming from $K_{1}(\mathfrak{A}%
)$, settling exactness at this point of the column.\newline(Step 4) The
composition $j\circ\alpha$ is visibly zero, we just quotient out exactly the
image of this map; and for the same reason we have exactness here. Finally,
being a quotient, the last arrow is clearly surjective.
\end{proof}

\begin{lemma}
The square $X$ in Figure \ref{lfigA1} commutes.
\end{lemma}

\begin{proof}
Obvious.
\end{proof}

\begin{lemma}
\label{sw_SqY}The square $Y$ in Figure \ref{lfigA1} commutes.
\end{lemma}

\begin{proof}
(Step 1) Suppose $x\in K_{1}(A)$. Then $\alpha$ sends it to the diagonal
element $(-x,-x,\ldots)$. By the Reciprocity Law, Theorem
\ref{thm_reciprocity_law}, this gets mapped to zero in $K_{1}(\mathsf{LCA}%
_{\mathfrak{A}})$. Correspondingly, the projection on the second factor,
$\operatorname*{pr}_{2}$, also sends it to zero. (Step 2) Suppose $y\in
K_{1}(A_{\mathbb{R}})$. Then $\alpha$ sends it to the $K_{1}$-id\`{e}le
$(z_{\mathfrak{p}})_{\mathfrak{p}}$ with $z_{\mathfrak{p}}=y$ (under the
natural map) for $\mathfrak{p}$ an infinite place, and $z_{\mathfrak{p}}=1$
for a finite place. The map $K_{1}(A_{\mathbb{R}})\longrightarrow
K_{1}(\mathsf{LCA}_{\mathfrak{A}})$ is induced from sending a right
$A_{\mathbb{R}}$-module to itself, equipped with the real topology, see
\cite[Theorem 11.2]{etnclca}. However, this is the same as what $\tilde{\xi
}_{\mathfrak{p}}$ (of Equation \ref{lmixi6}) does in the case of
$\mathfrak{p}$ infinite.
\end{proof}

It remains to check that the square $Z$ commutes. This is a little more
difficult than the previous steps in the proof. Let us first recall a few
useful facts about the structure of division algebras over the $p$-adic numbers.

\begin{elaboration}
\label{rmk_LocalStructureAtFinitePlace}Let $F$ be a number field and suppose
$\mathfrak{p}$ is a finite place. Suppose $D$ is a division algebra over $F$
(this means: its center comes with a given inclusion $F\hookrightarrow
\zeta(D)$) and whose center is a finite field extension of $F_{\mathfrak{p}}$.
Let%
\[
v:F_{\mathfrak{p}}^{\times}\longrightarrow\mathbb{R}%
\]
be the normalized $\mathfrak{p}$-adic valuation, i.e. its image (usually
called the `value group') is $\alpha\mathbb{Z}\subset\mathbb{R}$ for some
$\alpha\in\mathbb{R}$. Then there is a unique extension $\tilde{v}:D^{\times
}\rightarrow\mathbb{R}$ to a discrete valuation on the division algebra. It is
still a discrete valuation with value group $\frac{\alpha}{e}\mathbb{Z}%
\subset\mathbb{R}$ for some integer $e\geq1$. Define%
\[
\Delta:=\{x\in D\mid\tilde{v}(x)\geq0\}\text{.}%
\]
Then $\Delta$ is an $\mathcal{O}_{\mathfrak{p}}$-order in $D$ and more
generally (a) it is the unique maximal $\mathcal{O}_{F}$-order inside
$D_{\mathfrak{p}}$, (b) it can alternatively be characterized as the integral
closure of $\mathcal{O}_{\mathfrak{p}}$ inside $D$. Upon normalization to have
integer values, the valuation $\tilde{v}$ gives rise to an exact sequence%
\begin{equation}
\Delta^{\times}\hookrightarrow D^{\times}\twoheadrightarrow\mathbb{Z}\text{.}
\label{lmixi11}%
\end{equation}
A uniformizer $\pi$ is any element $\pi\in D^{\times}$ which gets mapped to
$+1$ in this sequence, as in the commutative case. These results are found as
an overview in \cite[\S 1.4]{MR1278263}, or with complete proofs in \cite[Ch.
3, \S 12]{MR1972204}.
\end{elaboration}

\begin{lemma}
The square $Z$ in Figure \ref{lfigA1} commutes.
\end{lemma}

\begin{proof}
(Step 1) It suffices to check this for an arbitrary place $\mathfrak{p}$ and
an arbitrary $a\in K_{1}(A_{\mathfrak{p}})$, because if the claim is settled
in all these cases, it follows from all maps being group homomorphisms. So,
let us assume $\mathfrak{p}$ is chosen and fixed.\newline(Step 2) Let us get
the case of $\mathfrak{p}$ an infinite place out of the way. In this case, we
need to check it for an $a$ coming from a summand of $K_{1}(A_{\mathbb{R}})$.
However, any such $a$ lies in the image of $\alpha$ and hence by the
commutativity of the square $Y$ (Lemma \ref{sw_SqY}), any such element goes to
zero in the bottom row of Figure \ref{lfigA1}. In particular, the square $Z$
commutes for this input.\newline(Step 3)\ It remains to deal with
$\mathfrak{p}$ a finite place. By the Artin--Wedderburn Theorem, we can split
$A_{\mathfrak{p}}$ into a direct sum of matrix algebras $M_{n}(D)$ over
division algebras over $F_{\mathfrak{p}}$. By Morita invariance of $K$-theory,
we have the equivalence of $K$-theory spaces%
\[
K(D)\overset{\sim}{\longrightarrow}K(M_{n}(D))\qquad\text{under}\qquad
D\hookrightarrow M_{n}(D)
\]
as a top left $(1\times1)$-minor in an $(n\times n)$-matrix, so it suffices to
check it for arbitrary $a\in K_{1}(D)$, where $D$ is a division algebra over
$F_{\mathfrak{p}}$. The natural map $A_{\mathfrak{p}}^{\times}\rightarrow
K_{1}(A_{\mathfrak{p}})$ is surjective, see Equation \ref{lmixi10_1}. Hence,
we can start with an arbitrary $a\in A_{\mathfrak{p}}^{\times}$. Now, we use a
little bit of structure theory: Since $D$ is a division algebra over $F$ and
$\mathfrak{p}$ a finite place, and we are in the setting which we had recalled
in Elaboration \ref{rmk_LocalStructureAtFinitePlace}. Let us use the notation
loc. cit. Then by Equation \ref{lmixi11} we may write $a=u\pi^{n}$ for
$u\in\Delta^{\times}$ a unit of the maximal order, $\pi$ a uniformizer and
$n\in\mathbb{Z}$. Since our maps are group homomorphisms, it suffices to check
commutativity in the two cases%
\begin{equation}
\text{(a) }a:=u\qquad\qquad\text{and}\qquad\qquad\text{(b) }a:=\pi
\label{lmixi15}%
\end{equation}
separately. Doing this, the first steps of the computation agree in the cases
(a) and (b), so henceforth assume we are in one (but any) of these
cases.\newline(Step 4) We begin by considering first the map $\tilde{\xi}$,
followed by the right downward arrow to $\operatorname*{Cl}(\mathfrak{A})$.
From here onward, the proof uses the same strategy as \cite[Lemma
3.5]{etnclca2}. We recall from loc. cit. that the right downward sequence
comes from the long exact sequence%
\[
\cdots\longrightarrow K_{1}(A_{\mathbb{R}})\longrightarrow K_{1}%
(\mathsf{LCA}_{\mathfrak{A}})\overset{\partial}{\longrightarrow}%
K_{0}(\mathsf{Mod}_{\mathfrak{A},fg})\longrightarrow K_{0}(A_{\mathbb{R}%
})\longrightarrow\cdots
\]
and more specifically the right downward arrow in Figure \ref{lfigA1}
corresponds to $\partial$ in the above sequence. Thus, in order to compute
$\partial$, we need to go through the construction of this long exact
sequence. As was explained loc. cit., this differential agrees with%
\begin{equation}
\partial=\partial^{\ast}\circ\Phi^{-1}\circ q\text{,} \label{lmixi12}%
\end{equation}
where these maps come from the diagram%
\begin{equation}%
\xymatrix{
\cdots\ar[r] & {{\pi}_1}K({\mathsf{Mod}_{\mathfrak{A}}}) \ar[r] \ar
[d] & {{\pi}_1}K({{\mathsf{Mod}_{\mathfrak{A}}}}/{{\mathsf{Mod}_{{\mathfrak
{A}},fg}}}) \ar@{=}[d]^{\Phi} \ar[r]^-{{\partial}^{\ast}} & {{\pi}%
_0}K({\mathsf{Mod}_{{\mathfrak{A}},fg}}) \ar[d] \\
{{\pi}_1}K(\mathsf{LCA}_{\mathfrak{A},cg}) \ar[r] & {{\pi}_1}K(\mathsf
{LCA}_{\mathfrak{A}}) \ar[r]_-{q} & {{\pi}_1}K({\mathsf{LCA}_{\mathfrak{A}}%
}/{\mathsf{LCA}_{\mathfrak{A},cg}}) \ar[r] & {{\pi}_0}K(\mathsf{LCA}%
_{\mathfrak{A},cg}),
}
\label{lmixi14}%
\end{equation}
where in turn $\partial^{\ast}$ arises as the boundary map of the long exact
sequence of homotopy groups coming from the localization sequence%
\[
\mathsf{Mod}_{\mathfrak{A},fg}\longrightarrow\mathsf{Mod}_{\mathfrak{A}%
}\longrightarrow\mathsf{Mod}_{\mathfrak{A}}/\mathsf{Mod}_{\mathfrak{A}%
,fg}\text{.}%
\]
We refer to the proof of \cite[Lemma 3.5]{etnclca2} for a little more
background how these facts are proven. Following Equation \ref{lmixi10_1}, the
$K_{1}$-class we consider can explicitly be spelled out as%
\[
\tilde{\xi}(a)=\left[
\xymatrix{
0 \ar@<1ex>@{^{(}->}[r]^-{0} \ar@<-1ex>@{^{(}.>}[r]_-{0} & D \ar@
<1ex>@{->>}[r]^{a} \ar@<-1ex>@{.>>}[r]_{1} & D
}%
\right]
\]
in the Nenashev presentation. Now, we simply follow the formula in Equation
\ref{lmixi12} step by step. Applying $q$, we still can use the same
representative in $K_{1}(\mathsf{LCA}_{\mathfrak{A}}/\mathsf{LCA}%
_{\mathfrak{A},cg})$. However, we also have the exact sequence $\Delta
\hookrightarrow D\twoheadrightarrow D/\Delta$, in $\mathsf{LCA}_{\mathfrak{A}%
}$, where $\Delta$ is the maximal order of $D$. Now, $\Delta$ is a free finite
rank $\mathbb{Z}_{p}$-module and a compact clopen subgroup of $D$. Hence, the
quotient $D/\Delta$ carries the discrete topology. Since $\Delta$ is compact,
it is in particular compactly generated, and thus a zero object in the
quotient exact category $\mathsf{LCA}_{\mathfrak{A}}/\mathsf{LCA}%
_{\mathfrak{A},cg}$. Thus, we get an isomorphism $D\cong D/\Delta$ in
$\mathsf{LCA}_{\mathfrak{A}}/\mathsf{LCA}_{\mathfrak{A},cg}$. Hence,%
\begin{equation}
\left[
\xymatrix{
0 \ar@<1ex>@{^{(}->}[r]^-{0} \ar@<-1ex>@{^{(}.>}[r]_-{0} & D/{\Delta}
\ar@<1ex>@{->>}[r]^{a} \ar@<-1ex>@{.>>}[r]_{1} & D/{\Delta}
}%
\right]  \label{lmixi14b}%
\end{equation}
represents that \textit{same} class in $K_{1}(\mathsf{LCA}_{\mathfrak{A}%
}/\mathsf{LCA}_{\mathfrak{A},cg})$. Now, simply read the above as a Nenashev
representative in $K_{1}(\mathsf{Mod}_{\mathfrak{A}}/\mathsf{Mod}%
_{\mathfrak{A},fg})$. The exact functor $\Phi$ as in Diagram \ref{lmixi14}
sends this to itself, equipped with the discrete topology, but since
$D/\Delta$ has carried the discrete topology anyway, we see that we have found
a preimage of the element in Equation \ref{lmixi14b} under $\Phi$. Hence, in
view of Equation \ref{lmixi12}, it suffices to compute $\partial^{\ast}$ of
the element, regarded in $\mathsf{Mod}_{\mathfrak{A}}/\mathsf{Mod}%
_{\mathfrak{A},fg}$. As was explained in the proof of \cite[Lemma
3.5]{etnclca2}, this reduces to a homotopical problem: The boundary map
$\partial^{\ast}$ in%
\[
K_{1}(\mathsf{Mod}_{\mathfrak{A}}/\mathsf{Mod}_{\mathfrak{A},fg}%
)\longrightarrow K_{0}(\mathsf{Mod}_{\mathfrak{A},fg})
\]
does the following: Starting with a closed loop around the zero object in the
$K$-theory space of $\mathsf{Mod}_{\mathfrak{A}}/\mathsf{Mod}_{\mathfrak{A}%
,fg}$, it lifts it to a non-closed path from zero to some object $(P,Q)$ under
the fibration%
\[
K(\mathsf{Mod}_{\mathfrak{A}})\longrightarrow K(\mathsf{Mod}_{\mathfrak{A}%
}/\mathsf{Mod}_{\mathfrak{A},fg})\text{,}%
\]
and then the output is the $\pi_{0}$-element corresponding to the connected
component of $(P,Q)$ in which this path ends. Hence, in order to compute
$\partial^{\ast}$, we need to produce an explicit lift of the closed loop in
question.\newline(Step 5) Following the concrete properties of the
Gillet--Grayson model as summarized in \S \ref{sect_AlgKThy} and Figure
\ref{lmixi3b}, the element in Equation \ref{lmixi14b} corresponds to a loop,
depicted below on the left:%
\begin{equation}%
{\includegraphics[
height=1.2168in,
width=4.9415in
]%
{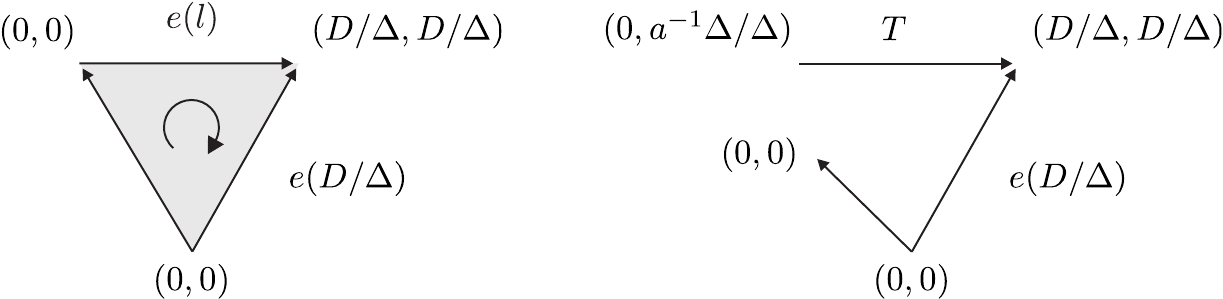}%
}
\label{lmixivio}%
\end{equation}
where the top horizontal arrow $e(l)$ comes from the $1$-simplex%
\[%
\xymatrix@!=0.5in{
0 \ar@{^{(}.>}[r] & D/{\Delta} \ar@{.>>}[r]^{1} & D/{\Delta} & \qquad&
0 \ar@{^{(}->}[r] & D/{\Delta} \ar@{->>}[r]^{a} & D/{\Delta}
}%
\]
in the Gillet--Grayson model of $K(\mathsf{Mod}_{\mathfrak{A}}/\mathsf{Mod}%
_{\mathfrak{A},fg})$. Consider the $1$-simplex $T$ given by%
\[%
\xymatrix@!=0.5in{
0 \ar@{^{(}.>}[r] & D/{\Delta} \ar@{.>>}[r]^{1} & D/{\Delta} & \qquad&
{a^{-1}\Delta/ \Delta} \ar@{^{(}->}[r] & D/{\Delta} \ar@{->>}[r]^{a}
& D/{\Delta}
}%
\]
in the Gillet--Grayson model of $K(\mathsf{Mod}_{\mathfrak{A}})$. Next, note
that $a^{-1}\Delta/\Delta$ is a finitely generated $\mathfrak{p}$-torsion
right $\mathfrak{A}_{\mathfrak{p}}$-module (note: this is true for both $a$ a
unit since then it is zero, or if $a=\pi$, for $\pi$ is a uniformizer of $D$
and since the valuation of $D$ extends the one of $F_{\mathfrak{p}}$, the
$\mathfrak{p}$-torsion property follows, see Elaboration
\ref{rmk_LocalStructureAtFinitePlace}). But being $\mathfrak{p}$-torsion, it
follows that $a^{-1}\Delta/\Delta$ is also a finitely generated right
($\mathfrak{p}$-torsion) $\mathfrak{A}$-module. Thus, $a^{-1}\Delta/\Delta
\in\mathsf{Mod}_{\mathfrak{A},fg}$ and we conclude that the quotient exact
functor $\mathsf{Mod}_{\mathfrak{A}}\longrightarrow\mathsf{Mod}_{\mathfrak{A}%
}/\mathsf{Mod}_{\mathfrak{A},fg}$ sends the $1$-simplex $T$ to the top
horizontal arrow in Figure \ref{lmixivio}. Thus, we have found a candidate for
the desired lift of the closed loop. It is depicted above in Figure
\ref{lmixivio} on the right. Analogous to the argument in the proof of
\cite[Lemma 3.5]{etnclca2}, we obtain that the endpoint of the path is the
vertex $(0,a^{-1}\Delta/\Delta)$ in the Gillet--Grayson model. In summary,%
\begin{align*}
\partial(a)  &  =\left(  \partial^{\ast}\circ\Phi^{-1}\circ q\right)  \left[
\xymatrix{
0 \ar@<1ex>@{^{(}->}[r]^-{0} \ar@<-1ex>@{^{(}.>}[r]_-{0} & D \ar@
<1ex>@{->>}[r]^{\cdot a} \ar@<-1ex>@{.>>}[r]_{1} & D
}%
\right] \\
&  =\text{connected component of }(0,a^{-1}\Delta/\Delta)
\end{align*}
and by \S \ref{subsect_ExplicitK0}, this is $[a^{-1}\Delta/\Delta]-[0]\in
K_{0}(\mathsf{Mod}_{\mathfrak{A},fg})$.\newline(Step 6) We follow the element
$a$ through the square $Z$ in the other way. The map $j$ sends it simply to
itself, merely under a further quotient operation. Finally, the bottom
horizontal map $w$ in Figure \ref{lfigA1} needs to be unravelled. To this end,
we refer to \cite[proof of (49.14)\ Corollary, p. 223]{MR892316}. In the
notation loc. cit. we consider the id\`{e}le formed with $a$ in the
$\mathfrak{p}$-component and $1$ as the component of all other places. Since
$a\in\Delta$ in both cases (a) and (b) of Equation \ref{lmixi15} we get%
\[
\lbrack\Delta/a\Delta]\in K_{0}(\mathfrak{A})\text{.}%
\]
Finally, we have the isomorphism%
\[
\frac{a^{-1}\Delta}{\Delta}\underset{\cdot a}{\overset{\sim}{\longrightarrow}%
}\frac{\Delta}{a\Delta}\text{,}%
\]
and hence $[\Delta/a\Delta]=[a^{-1}\Delta/\Delta]$ agree in $K_{0}%
(\mathfrak{A})$. This confirms that in both cases (a) and (b) we get the same
element, whichever way we follow the square $Z$. This finishes the proof.
\end{proof}

Now we can finally prove Proposition \ref{Prop_IdentifyK1LCA_Use_K1Ideles}.

\begin{proof}
[Proof of Proposition \ref{Prop_IdentifyK1LCA_Use_K1Ideles}]Consider the
commutative diagram in Figure \ref{lfigA1}. (Step 1) We truncate it to three
rows by quotienting out the images of the top vertical arrows, so that the top
horizontal arrow now reads%
\[
\operatorname*{pr}\nolimits_{2}:\left(  K_{1}(A)\oplus K_{1}(A_{\mathbb{R}%
})\right)  /\operatorname*{im}K_{1}(\mathfrak{A})\longrightarrow
K_{1}(A_{\mathbb{R}})/\operatorname*{im}K_{1}(\mathfrak{A})\text{.}%
\]
Since the map is projection to the second factor, the image of $K_{1}(A)$ in
the left-hand side quotient is the kernel of this map. On the other hand, the
map is obviously surjective. (Step 2) Now apply the Snake Lemma to the
remaining diagram. Since $w$ is an isomorphism by Lemma \ref{mz2}, the
resulting snake long exact sequence is%
\[
0\rightarrow\operatorname*{im}K_{1}(A)\overset{\alpha}{\rightarrow}\ker
(\tilde{\xi})\rightarrow0\rightarrow0\rightarrow\operatorname*{coker}%
(\tilde{\xi})\rightarrow0\text{,}%
\]
where by \textquotedblleft$\operatorname*{im}K_{1}(A)$\textquotedblright\ we
mean the kernel discussed in Step 1. We deduce that $\tilde{\xi}$ is
surjective, and that once we additionally quotient out by $\operatorname*{im}%
K_{1}(A)$ on the left-hand side (now with image taken under $\alpha$),
$\tilde{\xi}$ will be additionally injective on the quotient. However, this is
precisely the claim of Proposition \ref{Prop_IdentifyK1LCA_Use_K1Ideles}.
\end{proof}

\section{Noncommutative id\`{e}les II}

We work under the standing assumptions of \S \ref{sect_Setup}. In particular,
$F$ denotes a number field and $\mathfrak{A}\subset A$ an $\mathcal{O}_{F}%
$-order in a finite-dimensional semisimple $F$-algebra $A$. We shall write
\textquotedblleft$K_{0}(\mathfrak{A},\mathbb{R})_{\operatorname{Swan}}%
$\textquotedblright\ whenever we want to stress that we think of the relative
$K$-group $K_{0}(\mathfrak{A},\mathbb{R})$ in terms of the explicit Swan
presentation. Concretely, in the case at hand this means that generators have
the shape $[P,\varphi,Q]$, where $P,Q$ are finitely generated projective right
$\mathfrak{A}$-modules and%
\[
\varphi:P_{\mathbb{R}}\overset{\sim}{\longrightarrow}Q_{\mathbb{R}}%
\]
an isomorphism of right $A_{\mathbb{R}}$-modules. Then $K_{0}(\mathfrak{A}%
,\mathbb{R})_{\operatorname{Swan}}$ is the free abelian group generated by
these formal elements modulo Swan's Relation A and Relation B. We will not
recall these in full, see \cite[\S 1.1]{etnclca2}, or \cite[Chapter II,
Definition 2.10]{MR3076731}, where they are called \textquotedblleft relation
(a) and (b)\textquotedblright.

\begin{example}
\label{example_OneMiddleIsZero}In $K_{0}(\mathfrak{A},\mathbb{R}%
)_{\operatorname{Swan}}$ the identity $[\mathfrak{X},1,\mathfrak{X}]=0$ holds
for any finitely generated projective right $\mathfrak{A}$-module
$\mathfrak{X}$. To see this, use Swan's Relation B to obtain $[\mathfrak{X}%
,1,\mathfrak{X}]+[\mathfrak{X},1,\mathfrak{X}]=[\mathfrak{X},1\cdot
1,\mathfrak{X}]$.
\end{example}

\begin{definition}
\label{def_FrohlichTheory}Let an id\`{e}le $a=(a_{\mathfrak{p}})_{\mathfrak{p}%
}\in J(A)$ be given.

\begin{enumerate}
\item Then there is a unique $\mathcal{O}_{F}$-lattice \textquotedblleft%
$a\mathfrak{A}$\textquotedblright\ inside $A$ such that%
\begin{equation}
(a\mathfrak{A})_{\mathfrak{p}}=a_{\mathfrak{p}}\mathfrak{A}_{\mathfrak{p}}
\label{lciops1E}%
\end{equation}
holds for all finite places $\mathfrak{p}$ of $F$. See \cite[\S 2, Equation
2.2 and Theorem 1]{MR0376619} for background.

\item Secondly, we write $a_{\infty}$ for the map%
\[
a_{\infty}:A_{\mathbb{R}}\overset{\sim}{\longrightarrow}A_{\mathbb{R}}%
\]
coming from those components $a_{\mathfrak{p}}$ alone for which $\mathfrak{p}$
runs through the infinite places of $F$.
\end{enumerate}
\end{definition}

Given any $a\mathfrak{A}$ as in part (1) of the definition, tensoring with the
rationals yields an injection%
\[
a\mathfrak{A}\subset\mathbb{Q}\cdot(a\mathfrak{A})=\mathbb{Q}\cdot
\mathfrak{A}=A
\]
since each $a\mathfrak{A}$ is a torsion-free right $\mathfrak{A}$-module. We
we view $a\mathfrak{A}\subset A\subset A_{\mathbb{R}}$ as a full rank
$\mathbb{Z}$-lattice inside the real vector space $A_{\mathbb{R}}$. Thus, we
may alternatively regard $a_{\infty}$ as a map $a_{\infty}:\mathfrak{A}%
_{\mathbb{R}}\overset{\sim}{\rightarrow}(a\mathfrak{A})_{\mathbb{R}}$. We use
this for the following definition.

\begin{definition}
\label{def_MapTheta}Define%
\[
\theta:J(A)\longrightarrow K_{0}(\mathfrak{A},\mathbb{R})\text{,}\qquad
\qquad(a_{\mathfrak{p}})_{\mathfrak{p}}\longmapsto\lbrack\mathfrak{A}%
,a_{\infty},a\mathfrak{A}]
\]
with $a\mathfrak{A}$ and $a_{\infty}$ as in Definition
\ref{def_FrohlichTheory}.
\end{definition}

We first need to check that this definition makes sense at all.

\begin{lemma}
\label{lemma_ThetaWelldef}The map $\theta$ is well-defined.
\end{lemma}

\begin{proof}
We check that $\theta$ is a group homomorphism. Suppose $a:=(a_{\mathfrak{p}%
})_{\mathfrak{p}}$ and $b:=(b_{\mathfrak{p}})_{\mathfrak{p}}$ are id\`{e}les.
We compute%
\[
\theta(a)+\theta(b)=[\mathfrak{A},a_{\infty},a\mathfrak{A}]+[\mathfrak{A}%
,b_{\infty},b\mathfrak{A}]=[\mathfrak{A}\oplus\mathfrak{A},a_{\infty}\oplus
b_{\infty},a\mathfrak{A}\oplus b\mathfrak{A}]
\]
by using\ Swan's Relation A for the split exact sequence of the direct sum.
Now we use the classification of projective modules, in the following concrete
form: There exists an isomorphism of right $\mathfrak{A}$-modules,%
\[
\varphi:a\mathfrak{A}\oplus b\mathfrak{A}\overset{\sim}{\longrightarrow
}ab\mathfrak{A}\oplus\mathfrak{A}\text{.}%
\]
by\ Example \ref{example_CancellationRankTwo} (see \cite[Theorem 1,
(ii)]{MR0376619} for the proof this example is based on), and this isomorphism
sits in a commutative diagram of projective right $\mathfrak{A}$-modules
comprising the solid arrows in%
\[%
\xymatrix{
\mathfrak{A} \oplus\mathfrak{A} \ar@{..>}[d]_{a_{\infty} \oplus b_{\infty}}
\ar@{^{(}->}[r]^{1 \oplus1} & \mathfrak{A} \oplus\mathfrak{A} \ar@
{->>}[r] \ar@{..>}[d]_{a_{\infty} \cdot b_{\infty} \oplus1}  & 0 \ar@
{..>}[d]^{1} \\
a\mathfrak{A} \oplus b\mathfrak{A} \ar@{^{(}->}[r]_{\varphi} & ab\mathfrak{A}
\oplus\mathfrak{A} \ar@{->>}[r] & 0, \\
}%
\]
while the dotted downward arrows only exist (and commute alongside the solid
arrows) after tensoring everything over $\mathfrak{A}$ with $A_{\mathbb{R}}$.
This data can serve as the input for Swan's Relation A, and implies that%
\begin{align*}
\lbrack\mathfrak{A}\oplus\mathfrak{A},a_{\infty}\oplus b_{\infty
},a\mathfrak{A}\oplus b\mathfrak{A}]  &  =[\mathfrak{A}\oplus\mathfrak{A}%
,(a_{\infty}\cdot b_{\infty})\oplus1,ab\mathfrak{A}\oplus\mathfrak{A}%
]-[0,1,0]\\
&  =[\mathfrak{A},a_{\infty}b_{\infty},ab\mathfrak{A}]+[\mathfrak{A}%
,1,\mathfrak{A}]-[0,1,0]
\end{align*}
and by Example \ref{example_OneMiddleIsZero} we obtain equality to
$[\mathfrak{A},a_{\infty}b_{\infty},ab\mathfrak{A}]=\theta(a\cdot b)$, proving
our claim.
\end{proof}

\begin{lemma}
\label{lemma_mixi1}The map $\theta$ sends $\operatorname*{im}%
UK^{\operatorname{fin}}_{1}(\mathfrak{A})\subseteq JK_{1}(A)$ to zero.
\end{lemma}

\begin{proof}
Suppose the id\`{e}le $a:=(a_{\mathfrak{p}})_{\mathfrak{p}}$ comes from
$UK_{1}^{\operatorname{fin}}(\mathfrak{A})$. Then it sits entirely in the
components of the finite places, so $a_{\infty}=1$. Next, by Equation
\ref{lciops1E} (or Equation \ref{lciops1}) the lattice $a\mathfrak{A}$ is
uniquely determined by the equation $(a\mathfrak{A})_{\mathfrak{p}%
}=a_{\mathfrak{p}}\mathfrak{A}_{\mathfrak{p}}$ for all finite places
$\mathfrak{p}$. Since $\mathfrak{A}_{p}^{\times}\rightarrow K_{1}%
(\mathfrak{A}_{p})$ is surjective, we have $a_{\mathfrak{p}}\cdot
\mathfrak{A}_{\mathfrak{p}}=\mathfrak{A}_{\mathfrak{p}}$, and thus
$(a\mathfrak{A})_{\mathfrak{p}}=\mathfrak{A}_{\mathfrak{p}}$ holds for all
finite $\mathfrak{p}$, uniquely characterizing $a\mathfrak{A}$ as
$a\mathfrak{A}=\mathfrak{A}$. Hence, $\theta(a)=[\mathfrak{A},a_{\infty
},a\mathfrak{A}]=[\mathfrak{A},1,\mathfrak{A}]=0$ by Example
\ref{example_OneMiddleIsZero}.
\end{proof}

\begin{lemma}
\label{lemma_mixi2}The map $\theta$ sends $\operatorname*{im}K_{1}(A)\subseteq
JK_{1}(A)$ to zero.
\end{lemma}

\begin{proof}
The map $A^{\times}\rightarrow K_{1}(A)$ is an isomorphism. Thus, the image in
question consists of the diagonally constant id\`{e}les $(a)_{\mathfrak{p}}$,
with $a\in A^{\times}$. In particular, $a_{\infty}=a$. Since $A=\mathfrak{A}%
\otimes_{\mathbb{Z}}\mathbb{Q}$, we can write $a=\frac{1}{n}a_{0}$ for some
$a_{0}\in\mathfrak{A}\setminus\{0\}$ and $n\in\mathbb{Z}_{\geq1}%
\subset\mathfrak{A}\setminus\{0\}$. We have $\mathfrak{A}\setminus\{0\}\subset
A^{\times}$ and since $\theta$ is a group homomorphism, without loss of
generality it suffices to prove our claim for all elements $a\in
\mathfrak{A}\setminus\{0\}$. For such an element, we find a commutative
diagram%
\[%
\xymatrix{
\mathfrak{A} \ar@{..>}[d]_{1} \ar@{^{(}->}[r]^{1} & \mathfrak{A} \ar@
{->>}[r] \ar@{..>}[d]_{a}  & 0 \ar@{..>}[d]^{1} \\
\mathfrak{A} \ar@{^{(}->}[r]_{a} & a\mathfrak{A} \ar@{->>}[r] & 0, \\
}%
\]
where the solid arrows exist on the level of projective right $\mathfrak{A}%
$-modules and the dotted arrows only after tensoring with $A_{\mathbb{R}}$
(the same notation which we had used in the proof of Lemma
\ref{lemma_ThetaWelldef}). Having such a diagram, Swan's Relation A yields
$[\mathfrak{A},a,a\mathfrak{A}]=[\mathfrak{A},1,\mathfrak{A}]+[0,1,0]$ and by
Example \ref{example_OneMiddleIsZero} both terms on the right vanish.
\end{proof}

We will construct a commutative diagram%
\begin{equation}%
\xymatrix{
K_1(\mathfrak{A}) \ar[d] \ar[r]^{1} \ar@{}[dr]|{X}
& K_1(\mathfrak{A}) \ar[d] \\
K_1(A_{\mathbb{R}}) \ar[r]^{1} \ar@{->}[d]_{\beta} \ar@{}[dr]|{Y}
& K_1(A_{\mathbb{R}}) \ar[d] \\
\frac{JK_{1}(A)}{\operatorname*{im}K_1(A) + \operatorname*{im}%
UK^{\operatorname{fin}}_{1}(\mathfrak{A})} \ar[r]^{{\theta} }
\ar@{->>}[d]_{j} \ar@{}[dr]|{Z} & K_0({\mathfrak{A} },\mathbb{R}) \ar@
{->>}[d] \\
\frac{JK_{1}(A)}{\operatorname*{im}K_1(A) + \operatorname*{im}%
UK^{\operatorname{fin}}_{1}(\mathfrak{A}) + \operatorname*{im}K_1(A_{\mathbb
{R}})} \ar[r]_-{-{\sigma}_{\mathfrak{A}}}
& \operatorname{Cl}(\mathfrak{A}),
}
\label{lbixi1}%
\end{equation}
using the following maps: (a) The map $\beta$ sends $a\in A_{\mathbb{R}%
}^{\times}$ to the id\`{e}le $(a_{\mathfrak{p}})_{\mathfrak{p}}$ with
$a_{\mathfrak{p}}=1$ for all finite places $\mathfrak{p}$, while
$a_{\mathfrak{p}}$ agrees with the corresponding component of $a$ for all
infinite places. We may suggestively write%
\[
(1,1,\ldots,1,\underset{a_{\infty}}{\underbrace{a_{v},\ldots,a_{v^{\prime}}}%
})\text{,}%
\]
where $v,\ldots,v^{\prime}$ are the infinite places. (b)\ The map
$\sigma_{\mathfrak{A}}$ sends an id\`{e}le to its associated ideal class. This
construction comes from Fr\"{o}hlich's id\`{e}le classification of projective
modules, see \cite[\S 2, Theorem 1, especially\ Consequence \textquotedblleft
II\textquotedblright]{MR0376619}. This is the inverse map to the one in
Theorem \ref{thm_FrohlichTheory}.

\begin{lemma}
Diagram \ref{lbixi1} commutes.
\end{lemma}

\begin{proof}
The commutativity of the square $X$ is obvious. Square $Y$ is not much
harder:\ Suppose we are given $a\in A_{\mathbb{R}}$. Then $\beta$ maps it to
the id\`{e}le%
\[
\hat{a}:=(1,1,\ldots,1,\underset{a_{\infty}}{\underbrace{a_{v},\ldots
,a_{v^{\prime}}}})
\]
with $a_{\infty}=a$. Next, $\theta$ sends this to $[\mathfrak{A},a,\hat
{a}\mathfrak{A}]$. However, the id\`{e}le $\hat{a}$ differs from the neutral
element $(1,1,\ldots,1)$ only in the infinite places. Thus, $\hat
{a}\mathfrak{A}=\mathfrak{A}$ by Example
\ref{example_IdeleNontrivOnlyAtInfinity}, i.e. we get $[\mathfrak{A}%
,a,\mathfrak{A}]$. On the other hand, the map $K_{1}(A_{\mathbb{R}%
})\rightarrow K_{0}(\mathfrak{A},\mathbb{R})$ sends $a$ to the same class, by
definition (see \cite[Theorem 3.2]{etnclca2}, the map $\delta$ in the diagram
loc. cit., or \cite[p. 215]{MR0245634}). Hence, Square $Y$ commutes. The map
$\theta$ is well-defined and vanishes on $\operatorname*{im}UK_{1}%
^{\operatorname{fin}}(\mathfrak{A})+\operatorname*{im}K_{1}(A)$ by Lemma
\ref{lemma_mixi1} and \ref{lemma_mixi2}. It remains to check Square $Z$. Let
$a:=(a_{\mathfrak{p}})_{\mathfrak{p}}$ be some id\`{e}le in $JK_{1}(A)$. Then
$\theta(a)=[\mathfrak{A},a_{\infty},a\mathfrak{A}]$ and going down on the
right sends this to $[\mathfrak{A}]-[a\mathfrak{A}]$ (see \cite[Theorem
3.2]{etnclca2}, and concretely the proof of Lemma 3.5 loc. cit.). On the other
hand, the id\`{e}le $a$ under Fr\"{o}hlich's id\`{e}le classification of
projective right $\mathfrak{A}$-modules corresponds to $\sigma_{\mathfrak{A}%
}(a)=[a\mathfrak{A}]-[\mathfrak{A}]$, see \cite[\S 2, Consequence
\textquotedblleft II\textquotedblright\ of Theorem 1]{MR0376619}. The map
$\sigma_{\mathfrak{A}}$ loc. cit. specifically describes the class group as
the kernel from $K_{0}(\mathfrak{A})$ under the rank, as in Equation
\ref{lmixi7}. We use the same notation as in Fr\"{o}hlich's article. Thanks to
the negative sign, Square $Z$ commutes.
\end{proof}

\begin{lemma}
\label{Lemma_ExactColumns}The columns in Diagram \ref{lbixi1} are exact.
\end{lemma}

\begin{proof}
(Step 1)\ The right column comes from the standard long exact sequence of
relative $K$-groups,%
\[
\cdots\longrightarrow K_{1}(\mathfrak{A})\longrightarrow K_{1}(A_{\mathbb{R}%
})\longrightarrow K_{0}(\mathfrak{A},\mathbb{R})\longrightarrow K_{0}%
(\mathfrak{A})\longrightarrow K_{0}(A_{\mathbb{R}})\longrightarrow
\cdots\text{.}%
\]
From this sequence we obtain that%
\[
K_{1}(\mathfrak{A})\longrightarrow K_{1}(A_{\mathbb{R}})\longrightarrow
K_{0}(\mathfrak{A},\mathbb{R})\longrightarrow\ker\left(  K_{0}(\mathfrak{A}%
)\longrightarrow K_{0}(A_{\mathbb{R}})\right)  \longrightarrow0
\]
is exact, but one can show that the kernel on the right consists precisely of
those classes in $K_{0}(\mathfrak{A})$ with vanishing rank, i.e. this kernel
agrees with the one in Equation \ref{lmixi7}. This settles the right column.
(Step 2) For the left column it is clear that $j$ is just the quotient map
under the image coming from $\beta$. Thus, exactness is clear, except perhaps
at $K_{1}(A_{\mathbb{R}})$. We check this now: Firstly, we have $j\circ
\beta=0$ because the image of $K_{1}(\mathfrak{A})$ in $JK_{1}(A)$ is
contained in the image of $K_{1}(A)$, which we had quotiented out. Thus, we
only need to show that every $a\in K_{1}(A_{\mathbb{R}})$ such that
$\beta(a)\equiv0$ comes from $K_{1}(\mathfrak{A})$. Suppose such an $a\in
K_{1}(A_{\mathbb{R}})$ is given. As an id\`{e}le, we may write a
representative of its image in $JK_{1}(A)$ suggestively as%
\begin{equation}
\beta(a)=(1,1,\ldots,1,\underset{a_{\infty}}{\underbrace{a_{v},\ldots
,a_{v^{\prime}}}})\text{.} \label{lmixi20}%
\end{equation}
However, since we had assumed that $\beta(a)\equiv0$, we also know that
$\beta(a)=x\cdot y$ with $x\in UK_{1}^{\operatorname{fin}}(\mathfrak{A})$ and
$y\in K_{1}(A)$. We have tacitly dropped distinguishing between these elements
are their images in the id\`{e}les. Since the image of $UK_{1}%
^{\operatorname{fin}}(\mathfrak{A})$ is supported in the finite places alone
by Equation \ref{lmixi5a}, we learn that $y_{\mathfrak{p}}=a_{\mathfrak{p}}$
holds for all infinite places. Further, since $\operatorname*{im}K_{1}(A)$ is
diagonal, this means that we can assume that $a=y\otimes1_{\mathbb{R}}$.
However, Equation \ref{lmixi20} also implies that%
\[
1=x_{\mathfrak{p}}\cdot y_{\mathfrak{p}}\qquad\text{for all finite places
}\mathfrak{p}\text{.}%
\]
Thus, in Wall's exact sequence for $K$-theory, \cite[(42.19)\ Theorem]%
{MR892316},%
\[
K_{1}(\mathfrak{A})\overset{\operatorname*{diag}}{\longrightarrow}%
\underset{(y,x)}{K_{1}(A)\oplus K_{1}(\widehat{\mathfrak{A}})}\overset
{\operatorname*{diff}}{\longrightarrow}K_{1}(\widehat{A})\longrightarrow
K_{0}(\mathfrak{A})\longrightarrow\cdots
\]
we learn that the pair $(y,x)$ goes to zero, and therefore there exists some
$z\in K_{1}(\mathfrak{A})$ such that $(y,x)=(z\otimes1_{A},z\otimes
1_{\widehat{\mathfrak{A}}})$ under the diagonal map. Hence, $a=y\otimes
1_{\mathbb{R}}=(z\otimes1_{A})\otimes1_{\mathbb{R}}$, proving that $a$ lies in
the image from $K_{1}(\mathfrak{A})$ as desired.
\end{proof}

\subsection{Main results regarding Fr\"{o}hlich's id\`{e}le perspective}

We write $U^{\operatorname*{fin}}(\mathfrak{A})$ with the same meaning as
$U(\mathfrak{A})$ in this section in order to stress the analogy with
$UK_{1}^{\operatorname*{fin}}$. In case the reader has forgotten the
definition, see Equation \ref{lcixiDex1}.

\begin{theorem}
[Global--Local Formula, Swan presentation]\label{thm_GlobalLocal_Swan}Let $F$
be a number field, $\mathcal{O}_{F}$ its ring of integers. Suppose
$\mathfrak{A}\subset A$ is an $\mathcal{O}_{F}$-order in a finite-dimensional
semisimple $F$-algebra $A$. Then the following diagrams, whose rows are
isomorphisms, commute:

\begin{enumerate}
\item (Classical id\`{e}le formulation)%
\[%
\xymatrix{
\cfrac{J(A)}{J^{1}(A)+\operatorname*{im}(A^{\times})+\operatorname
*{im}U^{\operatorname{fin}}(\mathfrak{A})} \ar[r]^-{\theta}_-{\sim} \ar@
{->>}[d] &
K_{0}(\mathfrak{A},\mathbb{R}) \ar@{->>}[d] \\
\cfrac{J(A)}{J^{1}(A)+\operatorname*{im}(A^{\times})+\operatorname
*{im}U^{\operatorname{fin}}(\mathfrak{A})+\operatorname*{im}(A_{\mathbb{R}
}^{\times})}
\ar[r]_-{\sim} & \operatorname{Cl}(\mathfrak{A})
}%
\]

\item ($K_{1}$-id\`{e}le formulation)%
\[%
\xymatrix{
\cfrac{JK_{1}(A)}{\operatorname*{im}K_{1}(A)+\operatorname*{im}UK_{1}%
^{\operatorname{fin}}(\mathfrak{A})} \ar[r]^-{\theta}_-{\sim} \ar@{->>}[d] &
K_{0}(\mathfrak{A},\mathbb{R}) \ar@{->>}[d] \\
\cfrac{JK_{1}(A)}{\operatorname*{im}K_{1}(A)+\operatorname*{im}UK_{1}%
^{\operatorname{fin}}(\mathfrak{A}) + \operatorname*{im}K_1(A_{\mathbb{R} })}
\ar[r]_-{\sim} & \operatorname{Cl}(\mathfrak{A})
}%
\]

\item (Formulation in terms of the center)%
\[%
\xymatrix{
\cfrac{J(\zeta(A))}{\operatorname*{im}(\zeta(A)^{+,\times})+\prod
_{\mathfrak{p},\text{fin.}}\operatorname*{im}(\operatorname{nr}(\mathfrak
{A}^{\times}_{\mathfrak{p} }))} \ar[r]^-{\vartheta}_-{\sim} \ar@{->>}[d] &
K_{0}(\mathfrak{A},\mathbb{R}) \ar@{->>}[d] \\
\cfrac{J(\zeta(A))}{\operatorname*{im}(\zeta(A)^{+,\times})+\prod
_{\mathfrak{p},\text{fin.}}\operatorname*{im}(\operatorname{nr}(\mathfrak
{A}^{\times}_{\mathfrak{p} }))+\operatorname*{im}(A_{\mathbb{R} }^{\times})}
\ar[r]_-{\sim} & \operatorname{Cl}(\mathfrak{A})
}%
\]
Here $\zeta(-)$ denotes the center, and $()^{+}$ means: We restrict to
$a\in\zeta(A)$ such that $a_{\mathfrak{p}}>0$ for all real places of $F$ which
ramify in $A$. The products run only over the finite places of $F$.
\end{enumerate}

Recall the notation \textquotedblleft$a\mathfrak{A}$\textquotedblright\ of
Fr\"{o}hlich's id\`{e}le classification of projective $\mathfrak{A}%
$-modules\ (Equation \ref{lciops1}). In terms of the Swan presentation, the
maps are given by%
\[
\theta:(a_{\mathfrak{p}})_{\mathfrak{p}}\longmapsto\lbrack\mathfrak{A}%
,a_{\infty},a\mathfrak{A}]\qquad\vartheta:(a_{\mathfrak{p}})_{\mathfrak{p}%
}\longmapsto\lbrack\mathfrak{A},\operatorname*{nr}\nolimits^{-1}(a_{\infty
}),\operatorname*{nr}\nolimits^{-1}(a)\mathfrak{A}]\text{.}%
\]
Here the reduced norm maps $\operatorname*{nr}\nolimits^{-1}(-)$ are
understood component-wise for each place $\mathfrak{p}$. See the proof for
further clarification.
\end{theorem}

The notation $(-)^{+}$ is standard, and for example also used by Curtis and
Reiner \cite{MR1038525}, \cite{MR892316}. Note that its meaning depends on $A$
and not just on $\zeta(A)$. Let us point out that some readers might prefer to
think of the id\`{e}les as a multiplicative gadget and would prefer writing
\textquotedblleft$\cdot$\textquotedblright\ in the quotients rather than
\textquotedblleft$+$\textquotedblright. This is a matter of taste and we hope
it does not lead to confusion.

\begin{remark}
The bottom horizontal map in all formulations (1)-(3) of the theorem give
well-known variations on a theme due to Fr\"{o}hlich. The id\`{e}le
formulation (1) of the bottom horizontal map appears as Consequence
\textquotedblleft II\textquotedblright\ in Fr\"{o}hlich \cite{MR0376619}. For
the other formulations we refer the reader to \cite[(49.16)-(49.23)]{MR892316}
or Wilson's paper \cite{MR0447211}. Thus, in a sense, the above theorem
generalizes these results from the locally free class group
$\operatorname*{Cl}(\mathfrak{A})$ to all of $K_{0}(\mathfrak{A},\mathbb{R})$.
\end{remark}

\begin{remark}
The theorem is not really new. Agboola and Burns \cite{MR2192383} basically
give the $\operatorname*{Hom}$-description version of a more general result,
also for more general relative $K$-groups $K(\mathfrak{A},-)$, see
\cite[Theorem 3.5]{MR2192383}. See also \cite[Example 3.9, (2)]{MR2192383} for
an id\`{e}le description derived from it, especially Equation (10) loc. cit.,
which agrees literally with our formula (in this particular example
$\mathfrak{A}$ is commutative, so there is no $J^{1}$-quotient in their paper
since it is zero on the nose).
\end{remark}

\begin{proof}
(Step 1) The core of the proof is the verification of (2), the $K_{1}%
$-id\`{e}le formulation. We claim that the map $\theta$ of Definition
\ref{def_MapTheta} induces the isomorphism. Diagram \ref{lbixi1} is
commutative with exact columns by Lemma \ref{Lemma_ExactColumns}. The two top
horizontal maps are the identity and thus isomorphisms. The bottom horizontal
map $-\sigma_{\mathfrak{A}}$ is an isomorphism by Fr\"{o}hlich's description
of the locally free class group, in the format of \cite[(49.16) Proposition]%
{MR892316}. A crucial remark on notation: Loc. cit. the group Curtis
and\ Reiner denote by \textquotedblleft$JK_{1}(A)$\textquotedblright\ does not
contain the infinite places. Correspondingly, their \textquotedblleft%
$UK_{1}(A)$\textquotedblright\ is precisely the group $UK_{1}%
^{\operatorname{fin}}(\mathfrak{A})$ in this paper. Since we also quotient out
by $\operatorname*{im}K_{1}(A_{\mathbb{R}})$ in the bottom row of Diagram
\ref{lbixi1}, the quotient on the left is the same as the one discussed in
Curtis and\ Reiner \cite[(49.16) Proposition]{MR892316} on the left. Finally,
the\ Five Lemma implies that $\theta$ is also an isomorphism. Now the
commutativity of Square $Z$ in Diagram \ref{lbixi1}, as well as both
horizontal maps being isomorphisms, is the same as claim (2) in the above
theorem.\newline(Step 2) The rest of the proof follows the pattern of
\cite[(49.17)-(49.23)]{MR892316}. We first prove (3). As in the proof of
\cite[(49.17) Theorem]{MR892316}, the reduced norm induces isomorphisms%
\begin{equation}
\operatorname*{nr}:K_{1}(A_{\mathfrak{p}})\overset{\sim}{\longrightarrow}%
\zeta(A_{\mathfrak{p}})^{\times} \label{lfixi1}%
\end{equation}
for all places $\mathfrak{p}$. Hence, we obtain an isomorphism%
\begin{equation}
\operatorname*{nr}:JK_{1}(A)\overset{\sim}{\longrightarrow}J(\zeta(A))\text{.}
\label{lfixi2}%
\end{equation}
Next, by the Hasse--Schilling--Maass theorem \cite[(33.15)\ Theorem]%
{MR1972204}, the image of the reduced norm of $\operatorname*{im}K_{1}(A)$
inside $JK_{1}(A)$ under this map is $\zeta(A)^{+,\times}$. Similarly, the
image of the group $K_{1}(\mathfrak{A}_{\mathfrak{p}})$ gets sent to
$\operatorname*{im}(\operatorname*{nr}K_{1}(\mathfrak{A}_{\mathfrak{p}}))$.
However, since $\mathfrak{A}_{\mathfrak{p}}$ is local, the natural map
$\mathfrak{A}_{\mathfrak{p}}^{\times}\rightarrow K_{1}(\mathfrak{A}%
_{\mathfrak{p}})$ is surjective, so this image agrees with the image
$\operatorname*{im}(\operatorname*{nr}\mathfrak{A}_{\mathfrak{p}}^{\times})$
under the composition%
\[
\mathfrak{A}_{\mathfrak{p}}^{\times}\longrightarrow K_{1}(\mathfrak{A}%
_{\mathfrak{p}})\longrightarrow K_{1}(A_{\mathfrak{p}})\overset
{\operatorname*{nr}}{\longrightarrow}\zeta(A_{\mathfrak{p}})^{\times}\text{.}%
\]
All we have just done was transporting the subgroups appearing in the
denominator in (2) under the isomorphism of Equation \ref{lfixi2}. Thus, we
obtain an isomorphism%
\[
\frac{JK_{1}(A)}{\operatorname*{im}K_{1}(A)+\operatorname*{im}UK_{1}%
^{\operatorname{fin}}(\mathfrak{A})}\overset{\sim}{\longrightarrow}%
\frac{J(\zeta(A))}{\operatorname*{im}\zeta(A)^{+,\times}+%
{\textstyle\prod\nolimits_{\mathfrak{p}}}
\operatorname*{im}(\operatorname*{nr}\mathfrak{A}_{\mathfrak{p}}^{\times}%
)}\text{.}%
\]
Since the map is still described by the reduced norm in all components of the
id\`{e}les, formulation (2) implies formulation (3) in our claim. Finally, use
that the natural map $A_{\mathbb{R}}^{\times}\rightarrow K_{1}(A_{\mathbb{R}%
})$ is also surjective.\newline(Step 3) Finally, we prove formulation (1).
Define $c:J(A)\longrightarrow JK_{1}(A)$ by using the maps $A_{\mathfrak{p}%
}^{\times}\rightarrow K_{1}(A_{\mathfrak{p}})$ for all places $\mathfrak{p}$.
Since both $A_{\mathfrak{p}}^{\times}\rightarrow K_{1}(A_{\mathfrak{p}})$ as
well as $\mathfrak{A}_{\mathfrak{p}}^{\times}\rightarrow K_{1}(\mathfrak{A}%
_{\mathfrak{p}})$ are surjective, it is clear that the morphism $c$ is
surjective. Let $J^{1}(A)$ denote the kernel of this map. Consider the
commutative diagram%
\[%
\xymatrix{
J(A) \ar@{->>}[r]^-{c} \ar[dr]_{\operatorname{nr}} & JK_1(A) \ar
[d]^{\operatorname{nr}}_{\cong} \\
& J(\zeta(A)).
}%
\]
This yields the alternative characterization%
\[
J^{1}(A)=\left\{  (a_{\mathfrak{p}})_{\mathfrak{p}}\in J(A)\mid
\operatorname*{nr}\nolimits_{A_{\mathfrak{p}}}(a_{\mathfrak{p}})=1\right\}
\]
as the id\`{e}les of reduced norm one. We obtain the isomorphism of groups%
\[
\frac{J(A)}{J^{1}(A)}\underset{c}{\overset{\sim}{\longrightarrow}}%
JK_{1}(A)\text{.}%
\]
Moreover, under this isomorphism, the image $\operatorname*{im}A^{\times}$
inside $J(A)$ gets identified with $\operatorname*{im}K_{1}(A)$, and the image
$\operatorname*{im}U^{\operatorname{fin}}(\mathfrak{A})$ with
$\operatorname*{im}UK_{1}^{\operatorname{fin}}(\mathfrak{A})$. This finishes
the proof.
\end{proof}

\begin{theorem}
[Global--Local Formula, Nenashev presentation]\label{thm_GlobalLocal_Nenashev}%
Let $F$ be a number field, $\mathcal{O}_{F}$ its ring of integers. Suppose
$\mathfrak{A}\subset A$ is a regular $\mathcal{O}_{F}$-order in a
finite-dimensional semisimple $F$-algebra $A$. Then the following diagrams,
whose rows are isomorphisms, commute:

\begin{enumerate}
\item (Classical id\`{e}le formulation)%
\[%
\xymatrix{
\cfrac{J(A)}{J^{1}(A)+\operatorname*{im}(A^{\times})+\operatorname
*{im}U^{\operatorname{fin}}(\mathfrak{A})} \ar[r]^-{\theta}_-{\sim} \ar@
{->>}[d] &
{K_{1}(\mathsf{LCA}_{\mathfrak{A} })} \ar@{->>}[d] \\
\cfrac{J(A)}{J^{1}(A)+\operatorname*{im}(A^{\times})+\operatorname
*{im}U^{\operatorname{fin}}(\mathfrak{A})+\operatorname*{im}(A_{\mathbb{R}
}^{\times})}
\ar[r]_-{\sim} & \operatorname{Cl}(\mathfrak{A})
}%
\]

\item ($K_{1}$-id\`{e}le formulation)%
\[%
\xymatrix{
\cfrac{JK_{1}(A)}{\operatorname*{im}K_{1}(A)+\operatorname*{im}UK_{1}%
^{\operatorname{fin}}(\mathfrak{A})} \ar[r]^-{\theta}_-{\sim} \ar@{->>}[d] &
{K_{1}(\mathsf{LCA}_{\mathfrak{A} })} \ar@{->>}[d] \\
\cfrac{JK_{1}(A)}{\operatorname*{im}K_{1}(A)+\operatorname*{im}UK_{1}%
^{\operatorname{fin}}(\mathfrak{A}) + \operatorname*{im}K_1(A_{\mathbb{R} })}
\ar[r]_-{\sim} & \operatorname{Cl}(\mathfrak{A})
}%
\]

\item (Formulation in terms of the center)%
\[%
\xymatrix{
\cfrac{J(\zeta(A))}{\operatorname*{im}(\zeta(A)^{+,\times})+\prod
_{\mathfrak{p},\text{fin.}}\operatorname*{im}(\operatorname{nr}(\mathfrak
{A}^{\times}_{\mathfrak{p} }))} \ar[r]^-{\vartheta}_-{\sim} \ar@{->>}[d] &
{K_{1}(\mathsf{LCA}_{\mathfrak{A} })} \ar@{->>}[d] \\
\cfrac{J(\zeta(A))}{\operatorname*{im}(\zeta(A)^{+,\times})+\prod
_{\mathfrak{p},\text{fin.}}\operatorname*{im}(\operatorname{nr}(\mathfrak
{A}^{\times}_{\mathfrak{p} }))+\operatorname*{im}(A_{\mathbb{R} }^{\times})}
\ar[r]_-{\sim} & \operatorname{Cl}(\mathfrak{A})
}%
\]
Here $\zeta(-)$ denotes the center, and $()^{+}$ means: We restrict to
$a\in\zeta(A)$ such that $a_{\mathfrak{p}}>0$ for all real places of $F$ which
ramify in $A$. The products run only over the finite places of $F$.
\end{enumerate}

In terms of the Nenashev presentation, the maps are given by%
\[
\theta:(a_{\mathfrak{p}})_{\mathfrak{p}}\mapsto\left[
\xymatrix{
0 \ar@<1ex>@{^{(}->}[r]^-{0} \ar@<-1ex>@{^{(}.>}[r]_-{0} & {A_{\mathbb{A}}}
\ar@<1ex>@{->>}[r]^{\cdot(\ldots,a_{\mathfrak{p}},\ldots)} \ar@<-1ex>@{.>>}%
[r]_{1}
& {A_{\mathbb{A}}}
}%
\right]  \qquad\vartheta:(a_{\mathfrak{p}})_{\mathfrak{p}}\mapsto\left[
\xymatrix{
0 \ar@<1ex>@{^{(}->}[r]^-{0} \ar@<-1ex>@{^{(}.>}[r]_-{0} & {A_{\mathbb{A}}}
\ar@<1ex>@{->>}[r]^{\cdot(\ldots,{\operatorname{nr}\nolimits^{-1}}%
a_{\mathfrak{p}},\ldots)} \ar@<-1ex>@{.>>}[r]_{1}
& {A_{\mathbb{A}}}
}%
\right]  \text{.}%
\]

\end{theorem}

\begin{proof}
We use the same proof as for Theorem \ref{thm_GlobalLocal_Swan}. Simply
replace Step 1 loc. cit. by Proposition \ref{Prop_IdentifyK1LCA_Use_K1Ideles}.
Step 2 and Step 3 then follow analogously.
\end{proof}

\subsection{Extended boundary map}

We define the \emph{relative free class group} as%
\[
\operatorname*{Cl}(\mathfrak{A},\mathbb{R}):=\ker\left(  K_{1}(\mathsf{LCA}%
_{\mathfrak{A}})\overset{\partial}{\longrightarrow}K_{0}(\mathfrak{A}%
)\longrightarrow\prod_{p}K_{0}(\mathfrak{A}_{p})\right)  \text{,}%
\]
where $\partial$ is the boundary map in the long exact sequence of
\cite[Theorem 11.3]{etnclca}. This theorem also implies that this definition
is equivalent to the one in Burns--Flach \cite[\S 2.9]{MR1884523}. We follow
the notation of \textit{loc. cit.}: For an associative algebra $R$, we write
$\zeta(R)$ for its center, and $\operatorname*{nr}_{R}$ denotes the reduced
norm (see also \cite[\S 7D]{MR1038525}).

\begin{definition}
We define the \emph{extended boundary map} $\hat{\delta}_{\mathfrak{A}%
,\mathbb{R}}^{1}:\zeta(A_{\mathbb{R}})^{\times}\rightarrow\operatorname*{Cl}%
(\mathfrak{A},\mathbb{R})$ as follows: Given $y\in\zeta(A_{\mathbb{R}%
})^{\times}$, pick some $\lambda\in\zeta(A)^{\times}$ such that $\lambda
y\in\operatorname*{im}(\operatorname*{nr}\nolimits_{A_{\mathbb{R}}})$. Then
define%
\[
\psi_{y,\lambda}:=\left(  \prod_{p}\operatorname*{nr}\nolimits_{A_{p}}%
^{-1}(\lambda),\operatorname*{nr}\nolimits_{A_{\mathbb{R}}}^{-1}(\lambda
y)\right)  \in K_{1}(\widehat{A})\oplus K_{1}(A_{\mathbb{R}})\text{.}%
\]
Then $\hat{\delta}_{\mathfrak{A},\mathbb{R}}^{1}(y):=\operatorname*{sum}%
(\psi_{y,\lambda})$, where the sum map is the one from Theorem
\ref{thm_reciprocity_law}.%
\[%
\xymatrix{
\cdots\ar[r] & K_1(\mathfrak{A}) \ar[r] & K_1(A_{\mathbb{R}}) \ar[r] \ar@
{^{(}->}[d]_{\operatorname{nr}_{A_{\mathbb{R}}}} & \operatorname
*{Cl}(\mathfrak{A},\mathbb{R}) \\
& & \zeta(A_{\mathbb{R}})^{\times} \ar@{-->}[ur]_{\hat{\delta}_{\mathfrak
{A},\mathbb{R}}^{1}}
}%
\]

\end{definition}

This definition is very close to the one given in Burns--Flach, albeit with a
sum instead of a difference (and this is for the same reason as in Remark
\ref{rem_SignInReciprocityLaw}).

\begin{lemma}
The map $\hat{\delta}_{\mathfrak{A},\mathbb{R}}^{1}$ is well-defined.
\end{lemma}

\begin{proof}
We adapt \cite[Lemma 9]{MR1884523} to our locally compact setting. We shall
use the structure of the image of the reduced norm map in both the local as
well as the global situation, see \cite[(45.3)]{MR892316} for a summary
sufficient for our purposes. As in loc. cit., given $y$, by Weak Approximation
we find a (highly non-unique) $\lambda\in\zeta(A)^{\times}$ such that
$y\lambda\in\operatorname*{im}(\operatorname*{nr}\nolimits_{A_{\mathbb{R}}})$.
This is possible by the description of the image of the reduced norm of units
over the reals, \cite[(33.4)\ Theorem]{MR1972204}, i.e. we just need to make
$y\lambda$ positive at real places. Then $\operatorname*{nr}%
\nolimits_{A_{\mathbb{R}}}^{-1}(\lambda y)$ is a unique element, because the
reduced norm is injective when restricted to $K_{1}(A_{\mathbb{R}})$ by
\cite[(33.1)\ Theorem, (ii)]{MR1972204}. For all but finitely many primes $p$,
we have that the image of $\lambda$ in $\zeta(A_{p})^{\times}$ lies even in
$\operatorname*{nr}\nolimits_{A_{p}}^{-1}\zeta(\mathfrak{A}_{p})^{\times}$ and
that the latter lies in the image of $K_{1}(\mathfrak{A}_{p})$. If
$\lambda^{\prime}$ is an alternative choice, we find%
\[
\psi_{y,\lambda}\psi_{y,\lambda^{\prime}}^{-1}=\left(  \prod_{p}%
\operatorname*{nr}\nolimits_{A_{p}}^{-1}(\lambda\lambda^{\prime-1}%
),\operatorname*{nr}\nolimits_{A_{\mathbb{R}}}^{-1}(\lambda
y)\operatorname*{nr}\nolimits_{A_{\mathbb{R}}}^{-1}(\lambda^{\prime}%
y)^{-1}\right)  =\left(  \prod_{p}\operatorname*{nr}\nolimits_{A_{p}}%
^{-1}(\lambda\lambda^{\prime-1}),\operatorname*{nr}\nolimits_{A_{\mathbb{R}}%
}^{-1}(\lambda\lambda^{\prime-1})\right)  \text{.}%
\]
However, for elements $x\in A$ we have $\operatorname*{nr}\nolimits_{A}%
(x)=\operatorname*{nr}\nolimits_{A_{p}}(x)=\operatorname*{nr}%
\nolimits_{A_{\mathbb{R}}}(x)$ by \cite[(33.3) Theorem]{MR1972204}. Thus, we
get%
\[
=\left(  \prod_{p}\operatorname*{nr}\nolimits_{A}^{-1}(\lambda\lambda
^{\prime-1}),\operatorname*{nr}\nolimits_{A}^{-1}(\lambda\lambda^{\prime
-1})\right)
\]
and then $\lambda\lambda^{\prime-1}\in\operatorname*{im}(\operatorname*{nr}%
_{A})$ by the Hasse--Schilling--Maass norm theorem, see
\cite[(33.15)\ Theorem]{MR1972204}. Thus, $\psi_{y,\lambda}\psi_{y,\lambda
^{\prime}}^{-1}$ is the image of $\operatorname*{nr}_{A}^{-1}(\lambda
\lambda^{\prime-1})$ in $K_{1}(\widehat{A})\oplus K_{1}(A_{\mathbb{R}})$ in
Equation \ref{lmits1}. But then $\operatorname*{sum}(\psi_{y,\lambda}%
\psi_{y,\lambda^{\prime}}^{-1})=0$ by the reciprocity law, Theorem
\ref{thm_reciprocity_law}.
\end{proof}

Similarly to the discussion in \cite[\S 2.9]{MR1884523}, the exact sequence%
\[
\cdots\longrightarrow K_{1}(A_{\mathbb{R}})\longrightarrow K_{1}%
(\mathsf{LCA}_{\mathfrak{A}})\longrightarrow K_{0}(\mathfrak{A}%
)\longrightarrow\cdots
\]
can be truncated on the right and re-spliced to
\[
\cdots\longrightarrow K_{1}(A_{\mathbb{R}})\longrightarrow\operatorname*{Cl}%
(\mathfrak{A},\mathbb{R})\longrightarrow\operatorname*{Cl}(\mathfrak{A}%
)\longrightarrow0\text{.}%
\]

\section{Proof of the principal id\`{e}le fibration}

This section is fairly independent of the rest of the text. It is entirely
devoted to proving that%
\[
K(\widehat{\mathfrak{A}})\times K(A)\longrightarrow K(\widehat{A})\times
K(A_{\mathbb{R}})\longrightarrow K(\mathsf{LCA}_{\mathfrak{A}})
\]
in Equation \ref{lcixi1} is indeed a fibration. While loc. cit. it is stated
as a fibration of pointed simplicial sets having our conventions of
\S \ref{sect_GettingPrecise} in mind, we work on the level of spectra in this
section, relying on the results and language of the previous article
\cite{etnclca}. As the $K$-theory spaces in question are infinite loop spaces,
this amounts to the same and is just a change of language.

A certain sign switch will play an important r\^{o}le in the proof, so let us
begin with some careful considerations around signs:

\begin{elaboration}
\label{elab_SignsHtpyCartesianSquares}Choose some $\varepsilon\in\{-1,+1\}$.
Suppose $\mathsf{C}$ is a stable $\infty$-category and $h\mathsf{C}$ its
homotopy category. We write $\Sigma$ and $\Omega=\Sigma^{-1}$ for the
translation functors of $h\mathsf{C}$. Then a square%
\begin{equation}%
\xymatrix{
A \ar[r]^{a} \ar[d]_{f} & B \ar[d]^{g} \\
A^{\prime} \ar[r]_{a^{\prime}} & B^{\prime}
}
\label{lmisu2}%
\end{equation}
in $\mathsf{C}$ is called \emph{(homotopy)\ Cartesian} if there exists a
morphism $\partial_{\square}:B^{\prime}\rightarrow\Sigma A$ in $h\mathsf{C}$
such that%
\begin{equation}%
\xymatrix{
A \ar[r]^-{f+a} & A^{\prime} \oplus B \ar[r]^-{\varepsilon({a^{\prime}} - g)}
& B^{\prime} \ar[r]^-{\partial_{\square}} & \Sigma A
}
\label{lmisu1}%
\end{equation}
is a distinguished triangle in the category $h\mathsf{C}$. See Neeman
\cite[\S 1.4]{MR1812507} for a careful discussion purely on the level of
$h\mathsf{C}$. For both choices of $\varepsilon$ this definition makes sense
and one obtains the full theory. This choice of orientation is also discussed
by\ Lurie, from a slightly different angle \cite[Lemma 1.1.2.10]{LurieHA}. In
this paper we use the convention $\varepsilon:=1$ (which is compatible to
\cite{obloc}, \cite{etnclca}, \cite{etnclca2}), but the other option would
also work. Nothing would change, except a few signs here and there.
Nonetheless, the following is important: Suppose we are given the commutative
diagram%
\begin{equation}%
\xymatrix{
A \ar[r]^{a} \ar@{}[dr]|{\square} \ar[d]_{f} & B \ar[d]^{g} \ar[r]^{b}
& C \ar[r]^-{\partial_{F_1}} \ar[d]^{h}_{\cong} & \Sigma A \ar[d]^{\Sigma f}
\\
A^{\prime} \ar[r]_{a^{\prime}} & B^{\prime} \ar[r]_{b^{\prime}} & C^{\prime}
\ar[r]_-{\partial_{F_2}} & \Sigma A^{\prime}
}
\label{lmisu4}%
\end{equation}
in $h\mathsf{C}$ (with the left two squares lifted to $\mathsf{C}$) and with
$h$ an isomorphism in $h\mathsf{C}$. Then there is an attached distinguished
triangle as in Equation \ref{lmisu1} with $\partial_{\square}$ given as the
composition%
\[
B^{\prime}\overset{b^{\prime}}{\longrightarrow}C^{\prime}\underset{\sim
}{\overset{h}{\longleftarrow}}C\overset{\partial_{F_{1}}}{\longrightarrow
}\Sigma A\text{.}%
\]
in $h\mathsf{C}$. This is a variation of \cite[Lemma 1.4.3]{MR1812507}. By the
above definition, this means that the square on the left (marked by the
central `$\square$') is homotopy Cartesian in the stable $\infty$-category.
Now, what if $f$ instead of $h$ is an equivalence? To figure this out, we
rotate both distinguished triangles, giving the commutative diagram%
\[%
\xymatrix{
\Omega C \ar[r]^-{-\Omega\partial_{F_1}} \ar[d]_{\Omega h}^{\cong}
& A \ar[r]^{-a} \ar[d]_{f} & B \ar[d]^{g} \ar[r]^{-b} & C  \ar[d]^{h}_{\cong}
\\
\Omega C^{\prime} \ar[r]_-{-\Omega\partial_{F_2}} & A^{\prime} \ar
[r]_{-a^{\prime}} & B^{\prime} \ar[r]_{-b^{\prime}} & C^{\prime}
}%
\]
in $h\mathsf{C}$ so that upon renaming $A,B,C$ we are in the desired
situation. Next, check that%
\[%
\xymatrix@L=2.6mm{
A \ar[r]^-{f-a} & A^{\prime} \oplus B \ar[r]^-{\varepsilon({-a^{\prime}} - g)}
& B^{\prime} \ar[r]^-{-\partial_{\square}} & \Sigma A
}%
\]
is isomorphic to the triangle in Equation \ref{lmisu1} (to see this: Map $A$
and $A^{\prime}$ to themselves via the identity, on $B$ and $B^{\prime}$ use
the negative of the identity; all resulting squares commute. Note that this is
only true because we use $-\partial_{\square}$; it is not possible to make
this work without changing the sign there, or at some other point). It follows
that this triangle is also distinguished. Now rename $\tilde{a}:=-a$,
$\tilde{a}^{\prime}:=-a^{\prime}$ (same for $b,b^{\prime}$) and $C:=\Sigma D$.
Then%
\begin{equation}%
\xymatrix{
A \ar[r]^-{f+\tilde{a}} & A^{\prime} \oplus B \ar[r]^-{\varepsilon({\tilde
{a}^{\prime}} - g)} & B^{\prime} \ar[r]^-{-\partial_{\square}} & \Sigma A
}
\label{lmisu3}%
\end{equation}
is distinguished, and our input diagram reads%
\begin{equation}%
\xymatrix{
D \ar[r] \ar[d]_{\Omega h}^{\cong} & A \ar[r]^{\tilde{a}} \ar[d]_{f} \ar@
{}[dr]|{\square} & B \ar[d]^{g} \ar[r]^{\tilde{b}} & \Sigma D  \ar
[d]^{h}_{\cong} \\
D^{\prime} \ar[r] & A^{\prime} \ar[r]_{{\tilde{a}}^{\prime}} & B^{\prime}
\ar[r]_{{\tilde{b}}^{\prime}} & \Sigma D^{\prime}.
}
\label{lmisu3b}%
\end{equation}
Note that the distinguished triangle in Equation \ref{lmisu3} has exactly the
same shape as the one in Equation \ref{lmisu1} except for the different sign
of $\partial_{\square}$. We may summarize this as follows: Depending on
whether the first or third vertical arrow of a commutative diagram of the
shape of Equation \ref{lmisu4} is an isomorphism (that is: $f$ or $h$), the
other square will be homotopy Cartesian, and both variants only differ by the
sign of $\partial_{\square}$ (which by an extension of \cite[Lemma
1.1.2.10]{LurieHA} is equivalent to mirroring the diagram along the diagonal
from the upper left to the lower right). This is true independently of which
sign $\varepsilon$ we use in the first place. We repeat that we use the
convention $\varepsilon:=+1$ in this paper.
\end{elaboration}

With this preparation on signs in place, we can begin the proof. Firstly, we
elaborate on a theme due to Wall.

\begin{lemma}
\label{lemma_WallSeqForGTheory}Suppose $A$ is a finite-dimensional semisimple
$\mathbb{Q}$-algebra and $\mathfrak{A}\subset A$ an order. Then there is a
canonical fiber sequence%
\[
G(\mathfrak{A})\overset{\iota}{\longrightarrow}G(\widehat{\mathfrak{A}})\oplus
G(A)\overset{\operatorname*{diff}}{\longrightarrow}G(\widehat{A})
\]
in spectra. Here $\iota$ is the induced map on $K$-theory coming from the
exact functors of tensoring with $\widehat{\mathfrak{A}}$ resp. $A$ on the
right. Moreover, $\operatorname*{diff}=\rho-\tau$, where $\rho$ and $\tau$ are
induced from the exact functors of tensoring with $\widehat{A}$ in both cases.
\end{lemma}

For the algebraic $K$-theory of projective modules and restricted to low
degrees, this result was originally established by Wall. It was originally
proven using a different method based on excision squares. A more general
version is due to Swan, \cite[(42.22)\ Remark, (ii)]{MR892316}. We give a
quick self-contained account in contemporary language, if only to set up
notation and signs.

\begin{proof}
Let $R$ be a unital associative ring, finite as a $\mathbb{Z}$-module. We
write $\mathsf{Mod}_{R,fg}^{tor}$ for the abelian category of finitely
generated right $R$-modules which are torsion over $\mathbb{Z}$, that is: The
support of each modules over $\mathbb{Z}$ is supposed to be of codimension
$\geq1$ in $\operatorname*{Spec}\mathbb{Z}$. Then $\mathsf{Mod}_{R,fg}^{tor}$
is a Serre subcategory of $\mathsf{Mod}_{R,fg}$ and the quotient abelian
category is $\mathsf{Mod}_{R\otimes\mathbb{Q},fg}$. Now, applying Quillen's
Localization Theorem for Serre subcategories \cite[Ch. V, Theorem
5.1]{MR3076731} both for $R=\mathfrak{A}$ as well as $R=\widehat{\mathfrak{A}%
}$, we obtain that the two rows in the diagram
\begin{equation}%
\xymatrix@L=2.6mm@!C=0.9in{
K({\mathsf{Mod}_{\mathfrak{A},fg}^{tor}}) \ar[r] \ar[d]_{i}^{\cong}
& K({\mathsf{Mod}_{\mathfrak{A},fg}}) \ar@{}[dr]|{\Diamond} \ar[r] \ar[d]_{j}
& K({\mathsf{Mod}_{A,fg}}) \ar[d]^{k} \ar[r]^{\partial_{\mathfrak{A}%
}^{tor \hookrightarrow fg}} & \Sigma K({\mathsf{Mod}_{\mathfrak{A},fg}^{tor}%
}) \ar[d]^{\Sigma i}_{\cong} \\
K({\mathsf{Mod}_{\widehat{\mathfrak{A}},fg}^{tor}}) \ar[r] & K({\mathsf
{Mod}_{\widehat{\mathfrak{A}},fg}}) \ar[r] & K({\mathsf{Mod}_{\widehat{A},fg}%
}) \ar[r]_{\partial_{\widehat{\mathfrak{A}}}^{tor \hookrightarrow fg}}
& \Sigma K({\mathsf{Mod}_{\widehat{\mathfrak{A}},fg}^{tor}})
}
\label{ltixi1}%
\end{equation}
are distinguished. The downward arrows are%
\[
j:\mathsf{Mod}_{\mathfrak{A},fg}\longrightarrow\mathsf{Mod}_{\widehat
{\mathfrak{A}},fg}\text{,}\qquad M\mapsto M\otimes_{\mathfrak{A}}%
\widehat{\mathfrak{A}}%
\]%
\[
k:\mathsf{Mod}_{A,fg}\longrightarrow\mathsf{Mod}_{\widehat{A},fg}%
\text{,}\qquad M\mapsto M\otimes_{A}\widehat{A}%
\]
and $i$ is the restriction of $j$ to torsion modules. Note that since
completions are flat, the functors $j$ and $k$ are exact. The functor $i$ is
not just exact; it induces an equivalence of categories. By functoriality of
localization, Diagram \ref{ltixi1} commutes. Thus, we are in the situation of
Diagram \ref{lmisu3b} in Elaboration \ref{elab_SignsHtpyCartesianSquares}.
Hence, Equation \ref{lmisu3} gives a corresponding distinguished square in the
homotopy category of spectra $h\mathsf{Sp}$. Concretely, this means that%
\begin{equation}%
\xymatrix@L=2.6mm{
K({\mathsf{Mod}_{\mathfrak{A},fg}}) \ar[r]^-{(-)\otimes\widehat{\mathbb{Z}%
}+(-)\otimes\mathbb{Q}} & K({\mathsf{Mod}_{\widehat{\mathfrak{A}},fg}}) \oplus
K({\mathsf{Mod}_{A,fg}}) \ar[r]^-{{(-)\otimes\mathbb{Q}} - {(-)\otimes
\widehat{\mathbb{Q}}}} & K({\mathsf{Mod}_{\widehat{A},fg}}) \ar[r]^-{-\partial
_{\Diamond}} & \Sigma K({\mathsf{Mod}_{\mathfrak{A},fg}})
}
\label{l_WallDistTriangle}%
\end{equation}
is distinguished, where `$\Diamond$' refers to the respective square\ Diagram
\ref{ltixi1}. Note the negative sign in front of $\partial_{\Diamond}$.
\end{proof}

\begin{theorem}
\label{thm_PrincipalIdeleFibration}Let $A$ be a finite-dimensional semisimple
$\mathbb{Q}$-algebra. Suppose $\mathfrak{A}\subset A$ is a regular order. Then
there is a canonical fibration of pointed spaces%
\[
K(\widehat{\mathfrak{A}})\times K(A)\longrightarrow K(\widehat{A})\times
K(A_{\mathbb{R}})\longrightarrow K(\mathsf{LCA}_{\mathfrak{A}})\text{,}%
\]
which we call the \emph{principal id\`{e}le fibration}.

\begin{enumerate}
\item The first arrow is induced from the exact functors%
\begin{align*}
\operatorname*{PMod}(\widehat{\mathfrak{A}})  &  \longrightarrow
\operatorname*{PMod}(\widehat{A})\text{,}\qquad X\mapsto X\otimes
_{\widehat{\mathfrak{A}}}\widehat{A}\\
\operatorname*{PMod}(A)  &  \longrightarrow\operatorname*{PMod}(\widehat
{A})\times\operatorname*{PMod}(A_{\mathbb{R}})\text{,}\qquad X\mapsto
(X\otimes_{A}\widehat{A},X\otimes_{A}A_{\mathbb{R}})\text{.}%
\end{align*}

\item The second arrow is induced from the exact functor sending a right
$\widehat{A}$-module to itself, but equipped with the natural ad\`{e}le
topology. Similarly, a right $A_{\mathbb{R}}$-module gets sent to itself,
equipped with the natural real vector space topology.
\end{enumerate}
\end{theorem}

\begin{proof}
(Step 1) The commutative diagram%
\begin{equation}%
\xymatrix@L=2.6mm@!C=1.2in{
K({\mathsf{Mod}_{{\mathfrak{A}},fg}}) \ar[r] \ar[d]_{l} \ar@{}[dr]|{\ddag}
& K({\mathsf{Mod}_{\mathfrak{A}}}) \ar[r] \ar[d]_{m} & K({{\mathsf
{Mod}_{\mathfrak{A}}}}/{{\mathsf{Mod}_{{\mathfrak{A}},fg}}}) \ar[d]^{\Phi
}_{\cong} \ar[r]^-{\partial_{\mathfrak{A}}^{fg \hookrightarrow all}} & \Sigma
K({\mathsf{Mod}_{\mathfrak{A},fg}}) \ar[d]^{\Sigma l} \\
K(\mathsf{LCA}_{\mathfrak{A},cg}) \ar[r] & K(\mathsf{LCA}_{\mathfrak{A}}%
) \ar[r] & K({\mathsf{LCA}_{\mathfrak{A}}}/{\mathsf{LCA}_{\mathfrak{A},cg}%
}) \ar[r]_-{\partial_{\mathsf{LCA}}^{cg \hookrightarrow all}} & \Sigma
K(\mathsf{LCA}_{\mathfrak{A},cg})
}
\label{lmisu5}%
\end{equation}
was set up in \cite[Proposition 11.1]{etnclca}, using the same notation. Loc.
cit. we have only spelled out a commutative diagram of fiber sequences in
$\mathsf{Sp}$, whereas here we have expanded the entire datum including the
maps $\partial$ belonging to the underlying homotopy Cartesian squares. The
maps $l$ and $m$ come from reading the discrete $\mathfrak{A}$-modules and
locally compact $\mathfrak{A}$-modules, equipped with the discrete topology.
This clearly defines an exact functor. Since $\Phi$ (in the notation of the
reference) stems from an exact equivalence of exact categories, it induces an
isomorphism in $h\mathsf{Sp}$. Hence, we are in the situation of Diagram
\ref{lmisu4}. Thus, the square denoted by `$\ddag$' is homotopy Cartesian.
Unravelling the meaning of this along\ Elaboration
\ref{elab_SignsHtpyCartesianSquares}, we obtain the distinguished triangle%
\[%
\xymatrix@L=2.6mm{
K({\mathsf{Mod}_{\mathfrak{A},fg}}) \ar[r]^-{\text{incl}+\operatorname{incl}}
& K(\mathsf{LCA}_{\mathfrak{A},cg}) \oplus K({\mathsf{Mod}_{\mathbb{A}}}%
) \ar[r]^-{\operatorname{incl}_1 - \operatorname{incl}_2} & K(\mathsf
{LCA}_{\mathfrak{A}}) \ar[r]^-{\partial_{\ddag}} & \Sigma K({\mathsf
{Mod}_{\mathfrak{A},fg}})
}%
\]
Let us stress that this time the map $\partial$ carries a positive sign, as
carefully discussed in Elaboration \ref{elab_SignsHtpyCartesianSquares} on the
basis of the equivalence $\Phi$ in Diagram \ref{lmisu5} sitting on a different
position as in Diagram \ref{ltixi1}. Note that the signs we get here are
exactly the ones as in \cite[Proposition 11.1]{etnclca}, justifying our choice
of $\varepsilon=+1$ in Elaboration \ref{elab_SignsHtpyCartesianSquares}%
.\newline(Step 2) Next, the category of all right $\mathfrak{A}$-modules
$\mathsf{Mod}_{\mathfrak{A}}$ is closed under coproducts, so by the\ Eilenberg
swindle, Lemma \ref{lemma_EilenbergSwindle}, we have $K(\mathsf{Mod}%
_{\mathfrak{A}})=0$.\newline(Step 3) Now we shall set up the following
diagram:%
\[%
\xymatrix@L=2.6mm{
K({\mathsf{Mod}_{\mathfrak{A},fg}}) \ar@{}[dr]|{\sharp} \ar[d]_{1}
\ar[r]^-{(-)\otimes\widehat{\mathbb{Z}}+(-)\otimes\mathbb{Q}} & K({\mathsf
{Mod}_{\widehat{\mathfrak{A}},fg}}) \oplus K({\mathsf{Mod}_{A,fg}}) \ar@
{}[dr]|{\square} \ar[d]^{(0,(-)\otimes\mathbb{R})} \ar[r]^-{{(-)\otimes
\mathbb{Q}} - {(-)\otimes\widehat{\mathbb{Q}}}} & K({\mathsf{Mod}_{\widehat
{A},fg}}) \ar@{}[dr]|{\natural} \ar[d]_{\text{top. realiz.}}
\ar[r]^-{-\partial_{\Diamond}} & \Sigma K({\mathsf{Mod}_{\mathfrak{A},fg}%
}) \ar[d]^{1} \\
K({\mathsf{Mod}_{\mathfrak{A},fg}}) \ar[r]_-{\operatorname{incl}}
& K(\mathsf{LCA}_{\mathfrak{A},cg}) \ar[r]_-{\operatorname{incl}}
& K(\mathsf{LCA}_{\mathfrak{A}}) \ar[r]_-{\partial_{\ddag}} & \Sigma
K({\mathsf{Mod}_{\mathfrak{A},fg}})
}%
\]
Both rows are the distinguished triangles which we had produced in Lemma
\ref{lemma_WallSeqForGTheory} (and more specifically given in detail
in\ Equation \ref{l_WallDistTriangle}), and the one coming from Step 1 and
Step 2. Thus, it remains to describe the downward arrows and prove the
commutativity of the three squares.\newline(Square `$\sharp$') We compose the
underlying exact functors, first going down and then right resp. the other way
round. We obtain the exact functors%
\[
h_{i}:\mathsf{Mod}_{\mathfrak{A},fg}\longrightarrow\mathsf{LCA}_{\mathfrak{A}%
,cg}\qquad\text{(for }i=1,2\text{),}%
\]
where $h_{1}$ sends a right $\mathfrak{A}$-module $X$ to itself, equipped with
the discrete topology, while $h_{2}$ sends it to $X_{\mathbb{R}}%
:=X\otimes_{\mathfrak{A}}A_{\mathbb{R}}$, and regards this as a topological
right $\mathfrak{A}$-module, equipped with the real topology. Clearly,
$h_{1}\neq h_{2}$ as exact functors. However, we only need to show that the
induced square commutes in $h\mathsf{Sp}$ after taking $K$-theory. To this
end, consider the exact functor $\operatorname*{PMod}(\mathfrak{A}%
)\rightarrow\mathcal{E}\mathsf{LCA}_{\mathfrak{A},cg}$ sending a finitely
generated projective right $\mathfrak{A}$-module $X$ to the exact sequence%
\[
X\hookrightarrow X_{\mathbb{R}}\twoheadrightarrow X_{\mathbb{R}}/X
\]
in $\mathsf{LCA}_{\mathfrak{A},cg}$. Here $X$ carries the discrete topology,
$X_{\mathbb{R}}$ the real vector space topology and $X_{\mathbb{R}}/X$ the
torus topology (topologically it stems from quotienting a real vector space by
a full rank $\mathbb{Z}$-lattice). Denote the individual functors $f_{i}$ for
$i=1,2,3$ for the left\ (resp. middle, resp. right) individual exact functor.
Since $\mathfrak{A}$ is regular, $\operatorname*{PMod}(\mathfrak{A})$ and
$\mathsf{Mod}_{\mathfrak{A},fg}$ have the same $K$-theory by resolution. Thus,
it suffices to define this exact functor on $\operatorname*{PMod}%
(\mathfrak{A})$. By Additivity we get $f_{2\ast}=f_{1\ast}+f_{3\ast}$. Next,
note that $f_{3}$ can be factored as%
\begin{equation}
\operatorname*{PMod}(\mathfrak{A})\longrightarrow\mathsf{LCA}_{\mathfrak{A}%
,C}\longrightarrow\mathsf{LCA}_{\mathfrak{A},cg}\text{,} \label{lmisu7}%
\end{equation}
where $\mathsf{LCA}_{\mathfrak{A},C}$ denotes the exact category of compact
right $\mathfrak{A}$-modules. Since products of compact spaces are compact,
the latter category is closed under products, so $K(\mathsf{LCA}%
_{\mathfrak{A},C})$ by the Eilenberg swindle, Lemma
\ref{lemma_EilenbergSwindle}. Thus, we necessarily have $f_{3\ast}=0$ since it
can be factored over a zero object. Hence, $f_{2\ast}=f_{1\ast}$, but
$f_{1}=h_{1}$ and $f_{2}=h_{2}$, proving $h_{1\ast}=h_{2\ast}$, and thus
proving the commutativity of the square `$\sharp$'. Let us point out that a
factorization of $f_{2}$ as in%
\begin{equation}
\text{\textquotedblleft}\operatorname*{PMod}(\mathfrak{A})\longrightarrow
\mathsf{LCA}_{\mathfrak{A},D}\longrightarrow\mathsf{LCA}_{\mathfrak{A}%
,cg}\text{\textquotedblright} \label{lmisu7a}%
\end{equation}
with $\mathsf{LCA}_{\mathfrak{A},D}$ the discrete right $\mathfrak{A}$-modules
\textit{does not exist}. The point is that while all compact right
$\mathfrak{A}$-modules are compactly generated, leading to Equation
\ref{lmisu7}, a discrete right $\mathfrak{A}$-module is compactly generated if
and only if it is finitely generated, so we cannot define the second arrow in
Equation \ref{lmisu7a} on all of $\mathsf{LCA}_{\mathfrak{A},D}$. We could
only define it on the finitely generated ones at best, but this category then
is not closed under countable coproducts, so the Eilenberg swindle cannot be
applied.\newline(Square `$\square$') This square commutes if and only if the
following two squares commute:%
\[%
\xymatrix{
K({\mathsf{Mod}_{\widehat{\mathfrak{A}},fg}}) \ar[d]^{0} \ar[r]^-{{(-)\otimes
\mathbb{Q}}} & K({\mathsf{Mod}_{\widehat{A},fg}}) \ar[d]_{\text{top. realiz.}}
\\
K(\mathsf{LCA}_{\mathfrak{A},cg}) \ar[r]_-{\operatorname{incl}} & K(\mathsf
{LCA}_{\mathfrak{A}})
}%
\qquad%
\xymatrix{
K({\mathsf{Mod}_{A,fg}}) \ar[d]^{(-)\otimes\mathbb{R}} \ar[r]^-{-{(-)\otimes
\widehat{\mathbb{Q}}}} & K({\mathsf{Mod}_{\widehat{A},fg}}) \ar[d]_{\text
{top. realiz.}} \\
K(\mathsf{LCA}_{\mathfrak{A},cg}) \ar[r]_-{\operatorname{incl}} & K(\mathsf
{LCA}_{\mathfrak{A}})
}%
\]
In the left square, we compare the zero map with the composition%
\[
K(\mathsf{Mod}_{\widehat{\mathfrak{A}},fg})\longrightarrow K(\mathsf{Mod}%
_{\widehat{A},fg})\longrightarrow K(\mathsf{LCA}_{\mathfrak{A}})\text{,}%
\]
but the latter is also zero by Local Triviality, Theorem
\ref{thm_LocalTriviality} (either give a precise argument using the
isomorphism of rings $\widehat{\mathfrak{A}}\cong\prod\widehat{\mathfrak{A}%
}_{p}$ and an approximation argument, or much more elegantly: Copy the proof
of Theorem \ref{thm_LocalTriviality} and use that $\widehat{\mathfrak{A}}$ is
a compact clopen in $\widehat{A}$ with discrete quotient $\widehat{A}%
/\widehat{\mathfrak{A}}$ if we equip $\widehat{\mathfrak{A}}$ with its natural
profinite topology, and $\widehat{A}/\widehat{\mathfrak{A}}$ with the natural
product as a restricted product of $\mathbb{Q}_{p}$-vector spaces. This way,
one can avoid any approximation argument). The right square can be done very
similarly: The two functors are induced from%
\[
v_{i}:\mathsf{Mod}_{A,fg}\longrightarrow\mathsf{LCA}_{\mathfrak{A}}%
\qquad\text{(for }i=1,2\text{)}%
\]%
\[
v_{1}(X):=X_{\mathbb{R}}\text{ (}=X\otimes_{A}A_{\mathbb{R}}\text{)}%
\qquad\text{and}\qquad v_{2}(X):=\widehat{X}\text{ (}=X\otimes_{A}\widehat
{A}\text{),}%
\]
where instead of $v_{2}$ we take the negative of what is induced by this
functor. Because of this sign switch, the two induced maps on $K$-theory agree
if and only if \textit{their sum} is the zero map $K(\mathsf{Mod}%
_{A,fg})\rightarrow K(\mathsf{LCA}_{\mathfrak{A}})$. However, this is
precisely the statement of the fundamental Reciprocity Law, Theorem
\ref{thm_reciprocity_law}.\newline(Square `$\natural$') The commutativity of
this square is the most delicate part of the proof.%
\begin{equation}%
\xymatrix@L=2.6mm{
K({\mathsf{Mod}_{\widehat{A},fg}}) \ar@{}[dr]|{\natural} \ar[d]_{\text
{top. realiz.}}
\ar[r]^-{-\partial_{\Diamond}} & \Sigma K({\mathsf{Mod}_{\mathfrak{A},fg}%
}) \ar[d]^{1} \\
K(\mathsf{LCA}_{\mathfrak{A}}) \ar[r]_-{\partial_{\ddag}} & \Sigma
K({\mathsf{Mod}_{\mathfrak{A},fg}})
}
\label{ltixi8b}%
\end{equation}
Note that, from the point of view of regarding $\mathsf{Sp}$ as a stable
$\infty$-category, checking the commutativity of this square amounts to
checking that the fiber sequences attached to the two rows have compatible
nullhomotopies.\newline We first follow the top horizontal arrow and then go
down. We unravel the definition of $\partial_{\Diamond}$. It comes from the
homotopy Cartesian square in Diagram \ref{ltixi1}. We have recalled how to set
up the attached distinguished triangle in Elaboration
\ref{elab_SignsHtpyCartesianSquares}, namely%
\begin{equation}
\partial_{\Diamond}:%
\xymatrix{
K(\mathsf{Mod}_{\widehat{A},fg}) \ar[r]^-{\partial_{\widehat{\mathfrak{A}}%
}^{tor \hookrightarrow fg}} & \Sigma K(\mathsf{Mod}_{\widehat{\mathfrak{A}%
},fg}^{tor}) & \Sigma K(\mathsf{Mod}_{\mathfrak{A},fg}^{tor}) \ar
[r] \ar[l]^{\sim} & \Sigma K(\mathsf{Mod}_{\mathfrak{A},fg})
}%
\text{.} \label{ltixi8a}%
\end{equation}
On the other hand, going around the square `$\natural$' the other way, we
unravel\bigskip%
\begin{equation}
\partial_{\ddag}:%
\xymatrix@L=1.4mm{
K(\mathsf{Mod}_{\widehat{A},fg}) \ar[r]^-{\text{top. realiz.}} \ar@
/^2pc/[rr]^{f_2} & K(\mathsf{LCA}_{\mathfrak{A}}) \ar[r] & K(\mathsf
{LCA}_{\mathfrak{A}}/\mathsf{LCA}_{\mathfrak{A},cg}) \\
& K(\mathsf{Mod}_{\mathfrak{A}}/\mathsf{Mod}_{\mathfrak{A},fg}) \ar
[ur]^-{\sim}_-{\Phi} \ar[r]_-{\partial_{{\mathfrak{A}}}^{fg \hookrightarrow
all}} & \Sigma K(\mathsf{Mod}_{\mathfrak{A},fg})
}
\label{ltixi4}%
\end{equation}
Ignore the arrow with the label \textquotedblleft$f_{2}$\textquotedblright%
\ temporarily. Let us first focus on $\partial_{\ddag}$. We have an exact
equivalence of exact categories%
\[
\mathsf{Mod}_{\widehat{\mathfrak{A}},fg}/\mathsf{Mod}_{\widehat{\mathfrak{A}%
},fg}^{tor}\overset{\sim}{\longrightarrow}\mathsf{Mod}_{\widehat{A}}%
\text{,}\qquad M\mapsto M\otimes\mathbb{Q}\text{.}%
\]
This is the same equivalence which underlies the fiber sequences in Diagram
\ref{ltixi1}. Consider the exact functor%
\begin{equation}
\operatorname*{PMod}(\widehat{\mathfrak{A}})\longrightarrow\mathcal{E}%
\mathsf{LCA}_{\mathfrak{A}} \label{ltixi3}%
\end{equation}
sending $\mathfrak{A}$ to the exact sequence%
\begin{equation}
\widehat{\mathfrak{A}}\hookrightarrow\widehat{A}\twoheadrightarrow\widehat
{A}/\widehat{\mathfrak{A}} \label{ltixi3a}%
\end{equation}
in $\mathsf{LCA}_{\mathfrak{A}}$, where (a) $\widehat{\mathfrak{A}}$ is
equipped with its natural compact topology. Its underlying LCA group is a
product $\prod\mathbb{Z}_{p}$; (b) $\widehat{A}$ is equipped with its natural
locally compact topology. Its underlying LCA group is a restricted product
$\left.  \prod\nolimits^{\prime}\right.  (\mathbb{Q}_{p}:\mathbb{Z}_{p})$; (c)
and $\widehat{A}/\widehat{\mathfrak{A}}$ is equipped with the quotient
topology. This just amounts to the discrete topology since by the construction
of the restricted product topology, $\widehat{\mathfrak{A}}$ sits as a clopen
subgroup in it. We want to use the Additivity Theorem. Write $f_{i}$,
$i=1,2,3$ for the three exact functors $f_{i}:\mathsf{Mod}_{\widehat
{\mathfrak{A}},fg}\longrightarrow\mathsf{LCA}_{\mathfrak{A}}$ pinned down by
the functor in Equation \ref{ltixi3}. We get an induced exact functor%
\begin{equation}
\mathsf{Mod}_{\widehat{\mathfrak{A}},fg}/\mathsf{Mod}_{\widehat{\mathfrak{A}%
},fg}^{tor}\longrightarrow\mathcal{E}\left(  \mathsf{LCA}_{\mathfrak{A}%
}/\mathsf{LCA}_{\mathfrak{A},cg}\right)  \text{.} \label{ltixi5}%
\end{equation}
By the Additivity Theorem, $f_{2\ast}=f_{1\ast}+f_{3\ast}$. However, since
$\widehat{\mathfrak{A}}$ is compact, it is compactly generated, so $f_{1}$
sends all objects to zero objects in the quotient exact category
$\mathsf{LCA}_{\mathfrak{A}}/\mathsf{LCA}_{\mathfrak{A},cg}$. Thus, $f_{2\ast
}=f_{3\ast}$. However, note that $f_{2}$ agrees with the functor, suggestively
denoted by $f_{2}$, in Diagram \ref{ltixi4}: The straight arrows just equip
$\widehat{A}$ with its natural locally compact topology. This is the same as
using the identification of Equation \ref{ltixi5} first, and then equipping
the outcome with the topology as discussed above in (b). Thus, by Additivity,
we may work with the functor underlying $f_{3}$ instead, since it induces the
same map on the level of $K$-theory.\newline Now, we repeat the same trick in
a similar fashion. Consider the exact functor%
\[
\operatorname*{PMod}(\widehat{\mathfrak{A}})\longrightarrow\mathcal{E}\left(
\mathsf{Mod}_{\mathfrak{A}}/\mathsf{Mod}_{\mathfrak{A},fg}\right)  \text{,}%
\]
sending $\mathfrak{A}$ again to the exact sequence $\widehat{\mathfrak{A}%
}\hookrightarrow\widehat{A}\twoheadrightarrow\widehat{A}/\widehat
{\mathfrak{A}}$, but now regarded in the category $\mathsf{Mod}_{\mathfrak{A}%
}/\mathsf{Mod}_{\mathfrak{A},fg}$ (this is a precise statement already, but
note that philosophically it corresponds to considering the same functor, but
this time equipping all terms in the exact sequence with the discrete topology
instead. Of course this is still exact). Again, we get an induced functor from
$\mathsf{Mod}_{\widehat{\mathfrak{A}},fg}/\mathsf{Mod}_{\widehat{\mathfrak{A}%
},fg}^{tor}$ since the torsion modules go to zero objects in $\mathsf{Mod}%
_{\mathfrak{A}}/\mathsf{Mod}_{\mathfrak{A},fg}$. Write $g_{i}$, $i=1,2,3$ for
the three exact functors%
\[
g_{i}:\mathsf{Mod}_{\widehat{\mathfrak{A}},fg}/\mathsf{Mod}_{\widehat
{\mathfrak{A}},fg}^{tor}\longrightarrow\mathsf{Mod}_{\mathfrak{A}%
}/\mathsf{Mod}_{\mathfrak{A},fg}\text{.}%
\]
The key point is the following: Running the equivalence $\Phi$ in Diagram
\ref{ltixi4} backwards, we get%
\[
\Phi_{\ast}^{-1}\circ f_{3\ast}=g_{3\ast}\text{.}%
\]
The point behind this is that the quotient $\widehat{A}/\widehat{\mathfrak{A}%
}$ in Equation \ref{ltixi3a} carries the discrete topology. However, by
Additivity we have $g_{2\ast}=g_{1\ast}+g_{3\ast}$, so combining these two
equations, and remembering $f_{2\ast}=f_{3\ast}$, we get%
\begin{equation}
\Phi_{\ast}^{-1}\circ f_{2\ast}=g_{2\ast}-g_{1\ast}\text{.} \label{ltixi5a}%
\end{equation}
Finally, consider the diagram%
\[%
\xymatrix@L=2.6mm@!C=1.2in{
& & K({\mathsf{Mod}_{{\widehat{\mathfrak{A}}},fg}}/{{\mathsf{Mod}%
^{tor}_{{\widehat{\mathfrak{A}}},fg}}}) \ar[d]^{\otimes\mathbb{Q}}_{g_2}
\ar@{-->}[dl] \ar[dr]^{0} \\
\cdots\ar[r]
& K({\mathsf{Mod}_{\mathfrak{A}}}) \ar[r] & K({{\mathsf{Mod}_{\mathfrak{A}}}%
}/{{\mathsf{Mod}_{{\mathfrak{A}},fg}}}) \ar[r]_-{\partial_{\mathfrak{A}%
}^{fg \hookrightarrow all}} & \Sigma K({\mathsf{Mod}_{\mathfrak{A},fg}}).
}%
\]
The bottom row stems from the localization sequence. The exact functor $g_{2}$
admits a lift to $\mathsf{Mod}_{\mathfrak{A}}$. This would not work for
$g_{1}$ for example since $g_{1}$ would send the torsion modules
$\mathsf{Mod}_{\widehat{\mathfrak{A}},fg}^{tor}$ would go to non-zero objects
in $\mathsf{Mod}_{\mathfrak{A}}$. However, since $g_{2}$ sends torsion modules
to zero anyway, this lift exists. We deduce that%
\[
\partial_{\mathfrak{A}}^{fg\hookrightarrow all}\circ g_{2\ast}=0\text{.}%
\]
Thus, Equation \ref{ltixi5a} leads to%
\begin{equation}
\partial_{\mathfrak{A}}^{fg\hookrightarrow all}\circ\Phi_{\ast}^{-1}\circ
f_{2\ast}=\partial_{\mathfrak{A}}^{fg\hookrightarrow all}\circ g_{2\ast
}-\partial_{\mathfrak{A}}^{fg\hookrightarrow all}\circ g_{1\ast}%
=-\partial_{\mathfrak{A}}^{fg\hookrightarrow all}\circ g_{1\ast}\text{.}
\label{ltixi5b}%
\end{equation}
Returning to Diagram \ref{ltixi4}, we have shown that the morphism
$\partial_{\ddag}$ is the same morphism as%
\[%
\xymatrix@L=2.6mm{
K(\mathsf{Mod}_{\widehat{A},fg}) & K(\mathsf{Mod}_{\widehat{\mathfrak{A}}%
,fg}/\mathsf{Mod}_{\widehat{\mathfrak{A}},fg}^{tor})
\ar[l]_-{\sim} \ar[r]^-{-\partial_{\mathfrak{A}}^{fg\hookrightarrow all}}
& K(\mathsf{Mod}_{\mathfrak{A},fg})
}%
\]
since the functor $g_{1}$ just sends $\widehat{\mathfrak{A}}$ to itself,
treated as a right $\mathfrak{A}$-module. We swallow this rather na\"{\i}ve
operation into the notation. Consider the commutative diagram%
\[%
\xymatrix@L=2.6mm@!C=0.9in{
K({\mathsf{Mod}_{\widehat{\mathfrak{A}},fg}^{tor}}) \ar[r] \ar[d]_{s}
& K({\mathsf{Mod}_{\widehat{\mathfrak{A}},fg}}) \ar[r] \ar[d] & K({\mathsf
{Mod}_{\widehat{A},fg}}) \ar[d]_{g_1} \ar[r]^-{\partial_{\widehat{\mathfrak
{A}}}^{tor \hookrightarrow fg}}
& \Sigma K({\mathsf{Mod}_{\widehat{\mathfrak{A}},fg}^{tor}}) \ar[d]^{\Sigma s}
\\
K(\mathsf{Mod}_{\mathfrak{A},fg}) \ar[r] & K(\mathsf{Mod}_{\mathfrak{A}}%
) \ar[r] &
K({\mathsf{Mod}_{\mathfrak{A}}}/\mathsf{Mod}_{\mathfrak{A},fg}) \ar
[r]_-{\partial_{{\mathfrak{A}}}^{fg \hookrightarrow all}}
& \Sigma K(\mathsf{Mod}_{\mathfrak{A},fg}).
}%
\]
Both rows are distinguished triangles coming from the respective localization
sequences of $\mathsf{Mod}_{\widehat{\mathfrak{A}},fg}^{tor}$ as a Serre
subcategory of $\mathsf{Mod}_{\widehat{\mathfrak{A}},fg}$, resp.
$\mathsf{Mod}_{\mathfrak{A},fg}$ as a Serre subcategory of $\mathsf{Mod}%
_{\mathfrak{A}}$. The functor underlying $s$ sends a torsion right
$\widehat{\mathfrak{A}}$-module to itself, regarded as a right $\mathfrak{A}%
$-module. Note that just like $\mathsf{Mod}_{\widehat{\mathfrak{A}},fg}%
^{tor}\cong\mathsf{Mod}_{\mathfrak{A},fg}^{tor}$ are equivalent categories,
finitely generated torsion $\widehat{\mathfrak{A}}$-modules are indeed
finitely generated right $\widehat{\mathfrak{A}}$-modules. The commutativity
of the right square implies that we can continue the computation in Equation
\ref{ltixi5b} as%
\[
-\partial_{\mathfrak{A}}^{fg\hookrightarrow all}\circ g_{1\ast}=-(\Sigma
s)\circ\partial_{\widehat{\mathfrak{A}}}^{tor\hookrightarrow fg}\text{.}%
\]
Thus, in total, the map $\partial_{\ddag}$ of Diagram \ref{ltixi4} can be
expressed as follows.%
\[
\partial_{\ddag}=-(\Sigma s)\circ\partial_{\widehat{\mathfrak{A}}%
}^{tor\hookrightarrow fg}\text{.}%
\]
Finally, compare this to the map $\partial_{\Diamond}$ of Equation
\ref{ltixi8a}. The part $\partial_{\widehat{\mathfrak{A}}}^{tor\hookrightarrow
fg}$ agrees for both maps, and the following maps in Equation \ref{ltixi8a}
merely amount to regarding a finitely generated torsion right $\widehat
{\mathfrak{A}}$-module as a finitely generated right $\mathfrak{A}$-module.
This is the same functor as $s$. Thus, in total the only difference is the
sign, $\partial_{\ddag}=-\partial_{\Diamond}$. However, this is exactly what
we had to show, see Diagram \ref{ltixi8b}. Note that the appearance of this
sign is quite subtle. In our computations above it arose from a sign when
using the Additivity theorem, while in general it is needed for the right
compatibility because of r\^{o}le of signs as explained in Elaboration
\ref{elab_SignsHtpyCartesianSquares}.
\end{proof}

\section{Proof of compatibility\label{sect_CompatProofSection}}

In this section we will prove that our approach is equivalent to the original
construction of Burns and Flach in \cite{MR1884523}. To this end, let us go
through their construction, so roughly from \cite{MR1884523} \S 2.1 to \S 4.3
(although we can jump over certain parts).

Regarding our approach on the other hand, we use the construction of $T\Omega$
of \S \ref{sect_Overview}, using the rigorous justification of all
construction steps from \S \ref{sect_GettingPrecise}, especially Convention
\ref{convention_KThySpace}.

We recall the concept of a determinant functor from \cite[\S 4.3]{MR902592}.
Given any category $\mathsf{C}$, we write $\mathsf{C}^{\times}$ for its
internal groupoid, i.e. we delete all morphisms which are not isomorphisms.

\begin{definition}
\label{def_DeterminantFunctor}Suppose $\mathsf{C}$ is an exact category and
let $(\mathsf{P},\otimes)$ be a Picard groupoid. A \emph{determinant functor}
on $\mathsf{C}$ is a functor%
\[
\mathcal{D}:\mathsf{C}^{\times}\longrightarrow\mathsf{P}%
\]
along with the following extra structure and axioms:

\begin{enumerate}
\item For any exact sequence $\Sigma:G^{\prime}\hookrightarrow
G\twoheadrightarrow G^{\prime\prime}$ in $\mathsf{C}$, we are given an
isomorphism%
\[
\mathcal{D}(\Sigma):\mathcal{D}(G)\overset{\sim}{\longrightarrow}%
\mathcal{D}(G^{\prime})\underset{\mathsf{P}}{\otimes}\mathcal{D}%
(G^{\prime\prime})
\]
in $\mathsf{P}$. This isomorphism is required to be functorial in morphisms of
exact sequences.

\item For every zero object $Z$ of $\mathsf{C}$, we are given an isomorphism
$z:\mathcal{D}(Z)\overset{\sim}{\rightarrow}1_{\mathsf{P}}$ to the neutral
object of the Picard groupoid.

\item Suppose $f:G\rightarrow G^{\prime}$ is an isomorphism in $\mathsf{C}$.
We write%
\[
\Sigma_{l}:0\hookrightarrow G\twoheadrightarrow G^{\prime}\qquad
\text{and}\qquad\Sigma_{r}:G\hookrightarrow G^{\prime}\twoheadrightarrow0
\]
for the depicted exact sequences. We demand that the composition%
\begin{equation}
\mathcal{D}(G)\underset{\mathcal{D}(\Sigma_{l})}{\overset{\sim}%
{\longrightarrow}}\mathcal{D}(0)\underset{\mathsf{P}}{\otimes}\mathcal{D}%
(G^{\prime})\underset{z\otimes1}{\overset{\sim}{\longrightarrow}}%
1_{\mathsf{P}}\underset{\mathsf{P}}{\otimes}\mathcal{D}(G^{\prime}%
)\underset{\mathsf{P}}{\overset{\sim}{\longrightarrow}}\mathcal{D}(G^{\prime})
\label{l_CDetFunc1}%
\end{equation}
and the natural map $\mathcal{D}(f):\mathcal{D}(G)\overset{\sim}{\rightarrow
}\mathcal{D}(G^{\prime})$ agree. We further require that $\mathcal{D}(f^{-1})$
agrees with a variant of Equation \ref{l_CDetFunc1} using $\Sigma_{r}$ instead
of $\Sigma_{l}$.

\item If a two-step filtration $G_{1}\hookrightarrow G_{2}\hookrightarrow
G_{3}$ is given, we demand that the diagram%
\[%
\xymatrix{
\mathcal{D}(G_3) \ar[r]^-{\sim} \ar[d]_{\sim} & \mathcal{D}(G_1) \otimes
\mathcal{D}(G_3/G_1) \ar[d]^{\sim} \\
\mathcal{D}(G_2) \otimes\mathcal{D}(G_3/G_2) \ar[r]_-{\sim} & \mathcal
{D}(G_1) \otimes\mathcal{D}(G_2/G_1) \otimes\mathcal{D}(G_3/G_2)
}%
\]
commutes.

\item Given objects $G,G^{\prime}\in\mathsf{C}$ consider the exact sequences%
\[
\Sigma_{1}:G\hookrightarrow G\oplus G^{\prime}\twoheadrightarrow G^{\prime
}\qquad\text{and}\qquad\Sigma_{2}:G^{\prime}\hookrightarrow G\oplus G^{\prime
}\twoheadrightarrow G
\]
with the natural inclusion and projection morphisms. Then the diagram%
\[%
\xymatrix{
& \mathcal{D}(G \oplus G^{\prime}) \ar[dl]_{\mathcal{D}(\Sigma_1)}
\ar[dr]^{\mathcal{D}(\Sigma_2)} \\
\mathcal{D}(G) \otimes\mathcal{D}(G^{\prime}) \ar[rr]_{s_{G,G^{\prime}}}
& & \mathcal{D}(G^{\prime}) \otimes\mathcal{D}(G)
}%
\]
commutes, where $s_{G,G^{\prime}}$ denotes the symmetry constraint of
$\mathsf{P}$.
\end{enumerate}
\end{definition}

As usual, suppose $A$ is a finite-dimensional semisimple $\mathbb{Q}$-algebra,
$\mathfrak{A}\subset A$ an order. Also, let $F$ be a number field, $S_{\infty
}$ the set of infinite places of $F$, and $M\in\operatorname*{CHM}%
(F,\mathbb{Q})$ a Chow motive over $F$ (where we take the category of Chow
motives to have $\mathbb{Q}$-coefficients). Let%
\[
A\longrightarrow\operatorname*{End}\nolimits_{\operatorname*{CHM}%
(F,\mathbb{Q})}(X)\text{.}%
\]
be the action of $A$ on the motive. Pick a projective $\mathfrak{A}$-structure
$\{T_{v},v\in S_{\infty}\}$ and assume the\ \textit{Coherence Hypothesis} (as
defined and discussed in detail in \cite[\S 3.3]{MR1884523}).

In the construction of Burns and Flach, they work with the framework of Picard
groupoids\footnote{So this is inspired from the fact that before the
introduction of non-commutative coefficients, this would have been phrased in
terms of determinant lines (e.g. Fontaine, Perrin--Riou,\ldots), and these
form a Picard groupoid.}. The connection to our approach is as follows: in our
picture the Tamagawa number $T\Omega$ lives in $\pi_{1}K(\mathsf{LCA}%
_{\mathfrak{A}})$, so instead of working with the full $K$-theory space, it is
sufficient to work with a $1$-skeleton of that space (i.e. it does not matter
if we kill all homotopy groups $\pi_{i}$ for $i\geq2$). Viewed from the angle
of homotopy theory, this $1$-skeleton is a stable $(0,1)$-type. A priori it
would only be an unstable $(0,1)$-type, but since $K$-theory really comes from
a spectrum (or: when being viewed as a simplicial set in our setting of
\S \ref{subsect_Spaces}, it comes equipped with a $\Gamma$-space structure),
it is a stable homotopy type. However, the category of stable $(0,1)$-types
can alternatively be modelled in a somewhat more concrete fashion
through\ Picard groupoids.

The precise relation is as follows:

\begin{theorem}
\label{thm_PicardGrpdsAnd01Types}There is an equivalence of homotopy
categories,%
\begin{equation}
\Psi:\operatorname*{Ho}(\mathsf{Picard})\overset{\sim}{\longrightarrow
}\operatorname*{Ho}(\mathsf{Sp}^{0,1})\text{,} \label{lcimez20}%
\end{equation}
where $\mathsf{Picard}$ denotes the $\infty$-category of Picard groupoids, and
$\mathsf{Sp}^{0,1}$ denotes the $\infty$-category of spectra such that
$\pi_{i}X=0$ for $i\neq0,1$, also known as \emph{stable }$(0,1)$\emph{-types}.
The functor $\Psi^{-1}$ can be described as follows: If $E\in\mathsf{Sp}%
^{0,1}$ is the input spectrum, let $\Omega^{\infty}E$ denote its infinite loop
space. Define%
\[
\Psi^{-1}(E):=GP(\left\vert \Omega^{\infty}E\right\vert )\text{,}%
\]
i.e. where $GP$ denotes the fundamental groupoid (in the setting of simplicial
homotopy theory, see \cite[Ch. I, p. 42]{MR2840650} for the Gabriel--Zisman
fundamental groupoid). The infinite loop space structure equips this groupoid
with a symmetric monoidal structure, which gives rise to the Picard groupoid
structure in question.
\end{theorem}

We refer to \cite[\S 5.1, Theorem 5.3]{MR2981817} or alternatively
\cite[1.5\ Theorem]{MR2981952} for detailed proofs.

\begin{proposition}
\label{Prop_CompatPatel}This description of $\Psi$ is equivalent to the one
given by Patel \cite{MR2981817}.
\end{proposition}

\begin{proof}
Just follow Patel's description of his construction, \cite[\S 5.1]{MR2981817}.
Starting from a very special $\Gamma$-space $X$, we attaches to it the
topological fundamental groupoid (which he calls Poincar\'{e} groupoid.
Objects are points in $X(\mathbf{1})$ and morphisms are homotopy classes of
paths in $\left\vert X(\mathbf{1})\right\vert $). This is equivalent to what
we do; we just take the Gabriel--Zisman fundamental groupoid instead. The
equivalence of these two ways to form the fundamental groupoid is proven in
\cite[Chapter III, \S 1, Theorem 1.1]{MR2840650}.
\end{proof}

\begin{remark}
It is also equivalent to the one given by Johnson and Osorno \cite{MR2981952}.
Instead of using $\Gamma$-spaces to model the connective spectrum, they use
operads. However, the basic link is also a (topological) fundamental groupoid,
just as in Patel.
\end{remark}

We write $(V(\mathsf{C}),\boxtimes)$ for Deligne's Picard groupoid of virtual
objects of an exact category $\mathsf{C}$, \cite{MR902592}. Deligne proved in
this paper that there is a determinant functor%
\[
\mathcal{D}:\mathsf{C}^{\times}\longrightarrow(V(\mathsf{C}),\boxtimes
)\text{,}%
\]
which is actually (2-)universal, and in particular for any other determinant
functor $\mathcal{D}^{\prime}:\mathsf{C}^{\times}\longrightarrow
(\mathsf{P},\otimes)$ to some Picard groupoid $(\mathsf{P},\otimes)$, there
exists a factorization%
\[
\mathsf{C}^{\times}\overset{\mathcal{D}}{\longrightarrow}(V(\mathsf{C}%
),\boxtimes)\longrightarrow(\mathsf{P},\otimes)
\]
such that the composition is the given $\mathcal{D}^{\prime}$. The precise
notion of (2-)universality is actually rather subtle, see for example
\cite[\S 4.1]{MR2842932}, because it needs to take the entire symmetric
monoidal structure into consideration. Even better, there is also a map of
spaces%
\[
\mathsf{C}^{\times}\longrightarrow K(\mathsf{C})
\]
(where $\mathsf{C}^{\times}$ is regarded as its nerve) and under the
truncation to the $1$-skeleton,%
\[
\mathsf{C}^{\times}\longrightarrow K(\mathsf{C})\longrightarrow\tau_{\leq
1}K(\mathsf{C})\text{,}%
\]
if we apply $\Psi^{-1}$, this map transforms into the universal determinant
functor $\mathcal{D}$ above. In particular, it follows that there is a
canonical equivalence of stable $(0,1)$-types $\tau_{\leq1}K(\mathsf{C}%
)\overset{\sim}{\rightarrow}\Psi(V(\mathsf{C}),\boxtimes)$. Thus,
pre-composing this with the $1$-truncation, we obtain a map of spectra%
\begin{equation}
J:K(\mathsf{C})\longrightarrow\tau_{\leq1}K(\mathsf{C})\overset{\sim
}{\rightarrow}\Psi(V(\mathsf{C}),\boxtimes)\text{.} \label{llh}%
\end{equation}
Next, let us show that our concept of fundamental line is compatible with Burns--Flach.

\begin{theorem}
\label{thm_LinesAgree}The fundamental line point $\Xi(M)$ in $K(A)$ of
Equation \ref{ltiops1} under $J$ gets sent to the fundamental line virtual
object $\Xi(M)^{\operatorname*{BF}}$ of Burns--Flach \cite[\S 3.4]{MR1884523}.
\end{theorem}

We will split the proof into several parts. In general, even when coproducts
exist in a category, they are only well-defined up to unique isomorphism, so
technically the expression $P\oplus P^{\prime}$ does not define a point in the
nerve. We circumvented this problem by picking a concrete bifunctor in
Equation \ref{lexit1}, but we shall now see that this was merely an ad hoc
choice compatible with a fully homotopy coherent solution of the issue, which
we shall recall now:

\begin{definition}
[Segal \cite{MR0353298}, 2$^{\text{nd}}$ page]\label{def_SegalNerve}Suppose
$\mathsf{C}$ is a pointed category (we write $0$ for the base point object)
which admits finite coproducts. Write \textquotedblleft$\oplus$%
\textquotedblright\ for the coproduct\footnote{So, in the context of this
definition, we do \textit{not} (yet) demand that we have picked a bifunctor
$\mathsf{C}\times\mathsf{C}\rightarrow\mathsf{C}$ which exhibits these
coproducts as a monoidal structure. Thus, (at this point) for these coproducts
it suffices to be well-defined up to unique isomorphism.}. Then we write
$N_{\bullet}^{\oplus}$ to denote its categorical \emph{Segal nerve}. That is:
$N_{\bullet}^{\oplus}$ is a simplicial category and the objects in
$N_{n}^{\oplus}$ are $n$-tuples%
\[
(X_{1},\ldots,X_{n})
\]
of objects $X_{i}\in\mathsf{C}$ along with a choice of a coproduct $X_{i_{1}%
}\oplus\cdots\oplus X_{i_{r}}$, where $\{i_{1},\ldots,i_{r}\}$ runs through
all finite subsets of $\{1,\ldots,n\}$. We demand additionally that

\begin{enumerate}
\item for the empty subset the choice of the (empty)\ coproduct is the base
point object $0$;

\item for the singleton subsets $\{i\}$, we pick $X_{i}$ itself as the
(one-element) coproduct.
\end{enumerate}

The simplicial structure comes from deleting (resp. duplicating) the $i$-th
entry. A detailed definition and discussion is given in \cite[\S 1.8]%
{MR802796}.
\end{definition}

\begin{remark}
Each category $N_{n}^{\oplus}$ is equivalent to the $n$-fold product category
$\mathsf{C}\times\cdots\times\mathsf{C}$. The geometric realization
$\left\vert N_{\bullet}^{\oplus}\mathsf{C}\right\vert $ of the Segal nerve
carries a canonical structure as a $\Gamma$-space, see \cite[\S 2]{MR0353298}.
\end{remark}

\begin{lemma}
[{\cite[Observation 3.2]{MR1167575}}]\label{lemma_InfLoopSpaceStructures}%
Suppose $\mathsf{C}$ is a pseudo-additive\footnote{in the sense of
\cite[Definition 2.3]{MR1167575}.} Waldhausen category. The $K$-theory space
$K(\mathsf{C})$ carries a canonical infinite loop space structure. This
infinite loop space structure equivalently comes from

\begin{enumerate}
\item iterates of the Waldhausen $S$-construction:%
\[
w\mathsf{C}\longrightarrow\Omega\left\vert wS_{\bullet}\mathsf{C}\right\vert
\overset{\sim}{\longrightarrow}\Omega^{2}\left\vert wS_{\bullet}S_{\bullet
}\mathsf{C}\right\vert \overset{\sim}{\longrightarrow}\Omega^{3}\left\vert
wS_{\bullet}S_{\bullet}S_{\bullet}\mathsf{C}\right\vert \overset{\sim
}{\longrightarrow}\cdots\text{,}%
\]

\item the Waldhausen $S$-construction, and iterates of the Segal nerve with
respect to the composition law \textquotedblleft$\vee$\textquotedblright\ of
Remark \ref{rmk_Coproducts}:%
\[
w\mathsf{C}\longrightarrow\Omega\left\vert wS_{\bullet}\mathsf{C}\right\vert
\overset{\sim}{\longrightarrow}\Omega^{2}\left\vert wS_{\bullet}N_{\bullet
}^{\vee}\mathsf{C}\right\vert \overset{\sim}{\longrightarrow}\Omega
^{3}\left\vert wS_{\bullet}N_{\bullet}^{\vee}N_{\bullet}^{\vee}\mathsf{C}%
\right\vert \overset{\sim}{\longrightarrow}\cdots\text{,}%
\]

\item the $G$-construction, and iterates of the Segal nerve with respect to
the composition law of \textquotedblleft$\vee$\textquotedblright:%
\[
w\mathsf{C}\longrightarrow\left\vert wG_{\bullet}\mathsf{C}\right\vert
\overset{\sim}{\longrightarrow}\Omega\left\vert wG_{\bullet}N_{\bullet}^{\vee
}\mathsf{C}\right\vert \overset{\sim}{\longrightarrow}\Omega^{2}\left\vert
wG_{\bullet}N_{\bullet}^{\vee}N_{\bullet}^{\vee}\mathsf{C}\right\vert
\overset{\sim}{\longrightarrow}\cdots\text{,}%
\]

\end{enumerate}

and the stabilized terms (anywhere starting from the second term) are
equivalent to $\left\vert D_{\bullet}\right\vert $.
\end{lemma}

\begin{proof}
Regarding (1) and (2), this is already mentioned in Waldhausen's classic
\cite[\S 1.3, the paragraphs after the definition]{MR802796}, and in more
precise form in \cite[Lemma 1.8.6]{MR802796}, applied to the identity functor
$\mathsf{C}\longrightarrow\mathsf{C}$. Or, as mentioned, see \cite[Observation
3.2]{MR1167575}. The cited lemma yields the equivalence%
\begin{equation}
\left\vert wS_{\bullet}N_{\bullet}^{\vee}\mathsf{C}\right\vert \overset{\sim
}{\longrightarrow}\left\vert wS_{\bullet}S_{\bullet}\mathsf{C}\right\vert
\text{.} \label{lapix2}%
\end{equation}
The claim (3) is only a mild variation: We have%
\[
\left\vert wG_{\bullet}\mathsf{C}\right\vert \overset{\sim}{\longrightarrow
}\Omega\left\vert wS_{\bullet}\mathsf{C}\right\vert
\]
by \cite[Theorem 2.6]{MR1167575} and we can functorially apply this to the
Segal nerve $N_{\bullet}^{\vee}\mathsf{C}$, getting%
\[
\left\vert wG_{\bullet}N_{\bullet}^{\vee}\mathsf{C}\right\vert \overset{\sim
}{\longrightarrow}\Omega\left\vert wS_{\bullet}N_{\bullet}^{\vee}%
\mathsf{C}\right\vert \text{.}%
\]
Now use Equation \ref{lapix2} and composing these equivalences, we obtain%
\[
\Omega\left\vert wG_{\bullet}N_{\bullet}^{\vee}\mathsf{C}\right\vert
\overset{\sim}{\longrightarrow}\Omega^{2}\left\vert wS_{\bullet}N_{\bullet
}^{\vee}\mathsf{C}\right\vert \overset{\sim}{\longrightarrow}\Omega
^{2}\left\vert wS_{\bullet}S_{\bullet}\mathsf{C}\right\vert \overset{\sim
}{\longrightarrow}\Omega\left\vert wS_{\bullet}\mathsf{C}\right\vert \text{.}%
\]
Replacing $\mathsf{C}$ inductively by $N_{\bullet}^{\vee}\mathsf{C}$ in this
entire equivalence, and using Equation \ref{lapix2} repeatedly on the right,
we obtain (3).
\end{proof}

\begin{corollary}
\label{cor_SegalNerveSum}Using the Segal nerve $N_{\bullet}^{\vee}$, the above
observation equips $K(\mathsf{C})$ with a concrete structure as a $\Gamma
$-space with \textquotedblleft$\vee$\textquotedblright\ as the underlying
composition law. The resulting infinite loop space structure is the same one
as coming from the $S$-construction in (1) of the Lemma.
\end{corollary}

This corollary is the homotopy correct replacement for Equation
\ref{l_maxiu_1}.

\begin{proof}
[Proof of Theorem \ref{thm_LinesAgree}]We consider the map $J:K(\mathsf{C}%
)\rightarrow\tau_{\leq1}K(\mathsf{C})\overset{\sim}{\rightarrow}%
\Psi(V(\mathsf{C}),\boxtimes)$ of Equation \ref{llh}. Since the target is only
a stable $(0,1)$-type, we can check this statement by truncating to a stable
$(0,1)$-type all along. By Theorem \ref{thm_PicardGrpdsAnd01Types} we may
equivalently perform this verification in the framework of Picard groupoids.
Now observe that%
\begin{align}
\Xi(M)  &  =H_{f}^{0}(F,M)-H_{f}^{1}(F,M)+H_{f}^{1}(F,M^{\ast}(1))^{\ast
}-H_{f}^{0}(F,M^{\ast}(1))^{\ast}\label{lef3}\\
&  -\sum_{v\in S_{\infty}}H_{v}(M)^{G_{v}}+\sum_{v\in S_{\infty}}\left(
H_{dR}(M)/F^{0}\right) \nonumber
\end{align}
of Equation \ref{ltiops1} is formed using the sum and negation map of
\S \ref{subsect_JustifySumAndNegation}. We had picked a concrete choice for
the coproduct $\left.  \oplus\right.  :\mathsf{C}\times\mathsf{C}%
\longrightarrow\mathsf{C}$ as in Equation \ref{lexit1}. Now by Corollary
\ref{cor_SegalNerveSum} the $\Gamma$-space structure of $K$-theory is
compatible with any such choice, and further with the $\Gamma$-space structure
coming from the infinite loop space structure of the $S$-construction. As
Equation \ref{llh} comes from a map of spectra, it induces (as spaces) a map
of $\Gamma$-spaces. However, by Theorem \ref{thm_PicardGrpdsAnd01Types} and
Proposition \ref{Prop_CompatPatel} the symmetric monoidal structure on the
Picard groupoids%
\[
\Psi^{-1}\tau_{\leq1}K(\mathsf{C})\overset{\sim}{\rightarrow}(V(\mathsf{C}%
),\boxtimes)
\]
stems from this $\Gamma$-space structure. Finally, the universal determinant
functor $[-]$ used by Burns--Flach in \cite[\S 2.3-2.4]{MR1884523} is taken
exactly with respect to this symmetric monoidal structure. Thus, Equation
\ref{lef3} gets mapped to \cite[Equation (29) in \S 3.4]{MR1884523}, i.e.
$\Xi(M)^{\operatorname*{BF}}$. This proves the claim.
\end{proof}

By Theorem \ref{thm_PrincipalIdeleFibration} and shifting (this corresponds to
rotating the attached distinguished triangle on the level of the homotopy
category), we have the fiber sequence of spectra%
\[
\Omega K(\mathsf{LCA}_{\mathfrak{A}})\longrightarrow K(\widehat{\mathfrak{A}%
})\times K(A)\longrightarrow K(\widehat{A})\times K(A_{\mathbb{R}})\text{.}%
\]
But this just means that%
\begin{equation}
\Omega K(\mathsf{LCA}_{\mathfrak{A}})\overset{\sim}{\longrightarrow
}\operatorname*{fib}\left(  K(\widehat{\mathfrak{A}})\times
K(A)\longrightarrow K(\widehat{A})\times K(A_{\mathbb{R}})\right)  \text{.}
\label{lxmc3}%
\end{equation}
We will shortly use this below. We can now compare the construction of our
$R\Omega$ versus the one in \cite{MR1884523}. Burns and Flach consider the
diagram of exact functors (tensoring) between exact categories%
\[%
\xymatrix{
\operatorname{PMod}(\mathfrak{A}) \ar[r] \ar[d] & \operatorname{PMod}%
(A) \ar[d]  \\
\operatorname{PMod}(\widehat{\mathfrak{A}}) \ar[r] & \operatorname
{PMod}(\widehat{\mathfrak{A}})
}%
\]
and using these exact functors, one $2$-functorially gets induced morphisms
between the attached Picard groupoids of virtual objects $V(-)$. From the
resulting diagram, they define
\begin{equation}
\mathbb{V}(\mathfrak{A}):=V(\widehat{\mathfrak{A}})\times_{V(\widehat{A})}V(A)
\label{llh2}%
\end{equation}
as a fiber product in Picard groupoids. They show \cite[Proposition
2.3]{MR1884523},%
\begin{equation}
\pi_{0}\mathbb{V}(\mathfrak{A})\cong\pi_{0}V(\mathfrak{A})=K_{0}%
(\mathfrak{A})\text{.} \label{llh3}%
\end{equation}
Further, they define $\mathbb{V}(\mathfrak{A},\mathbb{R}):=\mathbb{V}%
(\mathfrak{A})\times_{V(A_{\mathbb{R}})}0$, where we write \textquotedblleft%
$0$\textquotedblright\ for the trivial Picard groupoid (this is $\mathcal{P}%
_{0}$ loc. cit.), so there is another Cartesian diagram%
\begin{equation}%
\xymatrix{
\mathbb{V}(\mathfrak{A},\mathbb{R}) \ar[r] \ar[d] & 0 \ar[d] \\
\mathbb{V}(\mathfrak{A}) \ar[r] & V(A_{\mathbb{R}})
}
\label{lww_Z1}%
\end{equation}
of Picard groupoids. Hence, by Equation \ref{llh3} it follows that%
\begin{equation}
\pi_{0}\mathbb{V}(\mathfrak{A},\mathbb{R})\cong K_{0}(\mathfrak{A},\mathbb{R})
\label{lww_Z2}%
\end{equation}
since both are merely the groups $\pi_{0}$ of the fiber along maps induced
from the same functor, namely tensoring to $A_{\mathbb{R}}$. Now
\cite[\S 3.4]{MR1884523} define%
\begin{equation}
\Xi(M,T_{p},S)^{\operatorname*{BF}}:=([R\Gamma_{c}\left(  \mathcal{O}%
_{F,S_{p}},T_{p}\right)  ],\Xi(M)^{\operatorname*{BF}},\vartheta_{p})\in
V(\mathfrak{A}_{p})\times_{V(A_{p})}V(A) \label{lix1}%
\end{equation}
(see loc. cit. for the meaning of $S$, $S_{p}$; $T_{p}$ stems from the
projective $\mathfrak{A}$-structure picked above), where they use the notation
of their concrete model of fiber products of Picard groupoids and a $p$-local
slight variant of Equation \ref{llh2}. The map $\vartheta_{p}$ is the same as
we use in Equation \ref{lww_Z3}. They go on to prove that there is no actual
dependency on $T_{p}$ or $S$, \cite[Lemma 5]{MR1884523}. Next, they glue from
this $p$-local data a virtual object%
\[
\Xi(M,T,S)_{\mathbb{Z}}^{\operatorname*{BF}}\in\mathbb{V}(\mathfrak{A})
\]
encompassing all finite primes $p$. For this, see \cite[Lemma 6]{MR1884523}.
Finally, they use $\vartheta_{\infty}$ (same as in our Equation \ref{lww_x5})
to get a further trivialization, moving this virtual object into the fiber in
Diagram \ref{lww_Z1}.

What has happened here: We have twice constructed an object by using the
defining property of the fiber product Picard groupoid: (1) first we used
(modulo some details around \cite[\S 3.4]{MR1884523} and \cite[Lemma
6]{MR1884523}) the fiber $\mathbb{V}(\mathfrak{A})$, i.e.%
\begin{equation}
V(\widehat{\mathfrak{A}})\times V(A)\longrightarrow V(\widehat{A})
\label{lepsi1}%
\end{equation}
and then (2) secondly the fiber of Diagram \ref{lww_Z1}, i.e.%
\begin{equation}
\mathbb{V}(\mathfrak{A})\longrightarrow V(A_{\mathbb{R}})\text{.}
\label{lepsi2}%
\end{equation}
Taking the fiber twice consecutively can equivalently be described as taking
the fiber of%
\begin{equation}
\operatorname*{fib}\left(  V(\widehat{\mathfrak{A}})\times V(A)\longrightarrow
V(\widehat{A})\times V(A_{\mathbb{R}})\right)  \text{.} \label{lgg1}%
\end{equation}
We have not spelled out the maps here, but they just stem from tensoring. Now
we may truncate Equation \ref{lxmc3} to the attached stable $(0,1)$-type,
giving%
\[
\tau_{\leq1}\Omega K(\mathsf{LCA}_{\mathfrak{A}})\overset{\sim}%
{\longrightarrow}\operatorname*{fib}\left(  \tau_{\leq1}K(\widehat
{\mathfrak{A}})\times\tau_{\leq1}K(A)\longrightarrow\tau_{\leq1}K(\widehat
{A})\times\tau_{\leq1}K(A_{\mathbb{R}})\right)  \text{,}%
\]
where we now mean the fiber in $\mathsf{Sp}^{0,1}$. However, Equation
\ref{llh} is ($2$-)functorial in exact functors between exact categories, so
firstly the truncations of the $K$-theory spaces can all be identified with
the stable $(0,1)$-types of their virtual objects, and the middle arrow is
functorially induced. Finally, since $\Psi$ is an equivalence of homotopy
categories, the notions of fiber are compatible. Thus,%
\begin{align}
\Psi^{-1}\tau_{\leq1}\Omega K(\mathsf{LCA}_{\mathfrak{A}})  &  \cong%
\operatorname*{fib}\left(  \Psi^{-1}\tau_{\leq1}K(\widehat{\mathfrak{A}%
})\times\Psi^{-1}\tau_{\leq1}K(A)\longrightarrow\Psi^{-1}\tau_{\leq
1}K(\widehat{A})\times\Psi^{-1}\tau_{\leq1}K(A_{\mathbb{R}})\right)
\label{lgg2}\\
&  \cong\operatorname*{fib}\left(  V(\widehat{\mathfrak{A}})\times
V(A)\longrightarrow V(\widehat{A})\times V(A_{\mathbb{R}})\right)
\text{,}\nonumber
\end{align}
which agrees with the fiber which Burns and Flach take, see Equation
\ref{lgg1}. Here we have tacitly used that the maps in the fibration sequence
of Theorem \ref{thm_PrincipalIdeleFibration} are induced from the same
functors (tensoring).

Thus, in Equation \ref{lgg2} we have produced an isomorphism between the
object%
\[
\Psi^{-1}\tau_{\leq1}\Omega K(\mathsf{LCA}_{\mathfrak{A}})
\]
in which our construction of the Tamagawa number is formulated (modulo
truncating to the $1$-skeleton and $\Psi^{-1}$, but as discussed above
truncating does not affect $\pi_{1}$, where our $T\Omega$ lies, and $\Psi
^{-1}$ preserves $\pi_{1}$, transforming it into the notion of $\pi_{1}$ for
Picard groupoids); and the object%
\[
\operatorname*{fib}\left(  V(\widehat{\mathfrak{A}})\times V(A)\longrightarrow
V(\widehat{A})\times V(A_{\mathbb{R}})\right)
\]
in which Burns and Flach run their construction of their Tamagawa number,
which we shall call $T\Omega^{\operatorname*{BF}}$.

This isomorphism being set up, we need to compare the actual
constructions:\ The object $\Xi(M,T_{p},S)^{\operatorname*{BF}}$ of Equation
\ref{lix1} stems from the input%
\[
\lbrack R\Gamma_{c}\left(  \mathcal{O}_{F,S_{p}},T_{p}\right)  ]\text{, }%
\Xi(M)^{\operatorname*{BF}}\text{, }\vartheta_{p}%
\]
and we had used the same object $R\Gamma_{c}\left(  \mathcal{O}_{F,S_{p}%
},T_{p}\right)  $ for our construction, the same map $\vartheta_{p}$, and
$\Xi(M)^{\operatorname*{BF}}$ was already shown to be the image of our
$\Xi(M)$ in Theorem \ref{thm_LinesAgree}. Similarly for $\vartheta_{\infty}$
in the fiber of $\mathbb{V}(\mathfrak{A})\rightarrow V(A_{\mathbb{R}}%
)$.\ Finally, take $\pi_{0}$ of Equation \ref{lgg2}. We get%
\[
\pi_{0}\Psi^{-1}\tau_{\leq1}\Omega K(\mathsf{LCA}_{\mathfrak{A}})=\pi
_{0}\Omega K(\mathsf{LCA}_{\mathfrak{A}})=\pi_{1}K(\mathsf{LCA}_{\mathfrak{A}%
})=K_{1}(\mathsf{LCA}_{\mathfrak{A}})\text{,}%
\]
while%
\[
\pi_{0}\operatorname*{fib}\left(  V(\widehat{\mathfrak{A}})\times
V(A)\longrightarrow V(\widehat{A})\times V(A_{\mathbb{R}})\right)  =\pi
_{0}\mathbb{V}(\mathfrak{A},\mathbb{R})=K_{0}(\mathfrak{A},\mathbb{R})
\]
by\ Equation \ref{lww_Z2}. This gives an identification of the groups in
question, coming from the identification of the separate fibers of Equations
\ref{lepsi1} and \ref{lepsi2} with the composite fiber in Equation \ref{lgg1}.
The two elements $\Xi(M,T_{p},S)^{\operatorname*{BF}}$ in (essentially)
$V(\widehat{\mathfrak{A}})\times_{V(\widehat{A})}V(A)$ of \cite[\S 3.4, page
526]{MR1884523} (plus the independence lemma proven loc. cit.) and
$(\Xi(M)_{\mathbb{Z}},\vartheta_{\infty})^{\operatorname*{BF}}$ in
$\mathbb{V}(\mathfrak{A},\mathbb{R})$, given in terms of the explicit
structure of relative Picard groupoids, then topologically can be unravelled
to give paths after the respective base change of the relative\ Picard
groupoid. They correspond to the path we define in Equations \ref{lww_Z3} (and
see Elaboration \ref{elab_ThetaPConstructedAsInBF} why it is clear that they
match) and the path of Equation \ref{lww_x5} respectively.

This construction gives a more concrete formulation of Theorem
\ref{thm2_Intro} and proves the equivalence.%

\appendix

\section{Complements}

\begin{example}
[Arakelov interpretation]\label{example_Arakelov}If the semisimple algebra $A$
is merely a number field, i.e. $A:=F$ and $\mathfrak{A}:=\mathcal{O}_{F}$ its
ring of integers, then one can interpret the id\`{e}le group of Equation
\ref{l_Rado_3} as an extension of the Arakelov--Picard group, i.e. a group
classifying metrized line bundles. Write $s$ for the number of real places of
$F$, and $r$ for the number of complex places. Consider the following
commutative diagram with exact rows and columns:%
\begin{equation}%
\xymatrix{
0 \ar[r] & \mu_{F} \ar@{^{(}->}[d] \ar[r] & {{\textstyle\prod_{\mathfrak
{p}\text{ fin.}}}\mathcal{O}_{\mathfrak{p}}^{\times}\oplus(S^{1})^{r}%
\oplus\{\pm1\}^{s}} \ar[r] \ar@{^{(}->}[d] & T \ar[r] \ar@{^{(}->}[d] & 0 \\
0 \ar[r] & F^{\times} \ar@{->>}[d] \ar[r] & {{\textstyle\prod_{\mathfrak
{p}\text{ fin.}}^{\prime}}F_{\mathfrak{p}}^{\times}\oplus{\textstyle
\bigoplus_{\sigma}}\mathbb{R}_{\sigma}^{\times}} \ar[r] \ar@{->>}[d] & C_{F}
\ar@{->>}[d] \ar[r] & 0 \\
0 \ar[r] & F^{\times}/\mu_{F} \ar[r] & {{\textstyle\bigoplus_{\mathfrak
{p}\text{ fin.}}}\mathbb{Z}\oplus{\textstyle\bigoplus_{\sigma}}\mathbb{R}}
\ar[r] & \widehat{\operatorname*{Pic}}_{F} \ar[r] & 0
}
\label{l_Fig_E1}%
\end{equation}
We write $\mu_{F}$ for the roots of unity in $F$, $\mathbb{R}_{\sigma}$ to
denote the closure of $F$ in the image of the infinite place $\sigma$, i.e.
this can be either $\mathbb{R}$ or $\mathbb{C}$. We write $C_{F}$ for the
id\`{e}le class group and $\widehat{\operatorname*{Pic}}_{F}$ for the
Arakelov--Picard group, in the sense of \cite{MR1847381}. We explain how to
construct Figure \ref{l_Fig_E1}: Take the two bottom rows as the input for the
snake lemma to get the top row. As all the downward arrows of the bottom rows
are surjective, the exactness of the top row follows. The bottom two rows stem
from the map (a) $F^{\times}$ being sent along the embeddings along all the
places in the middle row, and (b) $F^{\times}$ being sent to its valuation at
the finite places and $x\mapsto\log\left\vert \sigma(x)\right\vert $ for each
infinite place $\sigma:F\hookrightarrow\mathbb{R}_{\sigma}$. We wrote $T$
merely to denote the cokernel in the top row. The middle downward surjection
sends each element in $F_{\mathfrak{p}}^{\times}$ to its $\mathfrak{p}$-adic
valuation for finite places, and $x\mapsto\log\left\vert x\right\vert $ for
infinite places.\newline Now quotient out the image of $U(\mathfrak{A})=_{def}%
{\textstyle\prod_{\mathfrak{p}\text{ fin.}}}
\mathcal{O}_{\mathfrak{p}}^{\times}$ in $T$, transforming the right downward
column into%
\begin{equation}
\frac{(S^{1})^{r}\times\{\pm1\}^{s}}{\mu_{F}}\hookrightarrow\frac{JA}%
{J^{1}(A)\cdot A^{\times}\cdot U(\mathfrak{A})}\twoheadrightarrow
\widehat{\operatorname*{Pic}}_{F}\text{.} \label{l_h_Arak}%
\end{equation}
We summarize this as follows.
\end{example}

\begin{proposition}
\label{prop_XArak1}If $F$ is a number field, pick $A:=F$ and $\mathfrak{A}%
:=\mathcal{O}_{F}$. Then there is a canonical extension of abelian groups,%
\[
\frac{(S^{1})^{r}\times\{\pm1\}^{s}}{\mu_{F}}\hookrightarrow K_{0}%
(\mathfrak{A},\mathbb{R})\twoheadrightarrow\widehat{\operatorname*{Pic}}%
_{F}\text{;}%
\]
and of course we could also write $K_{1}(\mathsf{LCA}_{\mathfrak{A}})$ for the
middle group.
\end{proposition}

Given an Arakelov divisor $\sum x_{\mathfrak{p}}+\sum x_{\sigma}$, i.e. an
element of $%
{\textstyle\bigoplus_{\mathfrak{p}\text{ fin.}}}
\mathbb{Z}\oplus%
{\textstyle\bigoplus_{\sigma}}
\mathbb{R}$, representing a class in $\widehat{\operatorname*{Pic}}_{F}$, one
can attach to this an Arakelov line bundle, by equipping a genuine line bundle
$L$ within the isomorphism class of the image under $\widehat
{\operatorname*{Pic}}_{F}\twoheadrightarrow\operatorname*{Pic}(\mathcal{O}%
_{F})$ with the metric such that on $L\otimes_{\mathbb{Z}}\mathbb{R}$ we have%
\[%
\begin{tabular}
[c]{ll}%
$\left\Vert 1\right\Vert _{\sigma}^{2}=e^{-2x_{\sigma}}$ & $\text{for }%
\sigma\text{ real}$\\
$\left\Vert 1\right\Vert _{\sigma}^{2}=2e^{-2x_{\sigma}}$ & $\text{for }%
\sigma\text{ complex}$%
\end{tabular}
\ \
\]
in terms of the norm of the image of $1\in F^{\times}$ under the embedding
$\sigma$. Relating this to our constructions, this means that $(z_{\sigma
})_{\sigma\in S_{\infty}}\in\prod\mathbb{R}_{\sigma}^{\times}$ goes to%
\begin{equation}
\left\Vert 1\right\Vert _{\sigma}^{2}=c_{\sigma}e^{-2\log\left\vert \sigma
z_{\sigma}\right\vert }=c_{\sigma}\left\vert \sigma z_{\sigma}\right\vert
^{-2}\qquad\text{with}\qquad c_{\sigma}\in\{1,2\}\text{.} \label{l_h_Arak2}%
\end{equation}
The group on the left in Equation \ref{l_h_Arak} thus corresponds precisely to
the kernel of the absolute values occuring in Equation \ref{l_h_Arak2}. Thus,
if one insisted on giving the middle group of Equation \ref{l_h_Arak} a
geometric interpretation, it would be (angularly) decorated metrized line
bundles. This could be extended to the non-commutative setting, where now
real, complex and quaternion embeddings as in the Artin--Wedderburn
decomposition of $A\otimes_{\mathbb{Q}}\mathbb{R}$ would play a r\^{o}le.

If $A$ is a finite-dimensional semisimple $\mathbb{Q}$-algebra and
$\mathfrak{A}\subset A$ an arbitrary order, Proposition \ref{prop_XArak1}
should have analogues in a suitably formulated theory of $\mathfrak{A}%
$-equivariant Arakelov modules. The papers \cite{MR1914000}, \cite[\S 4]%
{MR2192383} give possible answers to this.

\bibliographystyle{amsalpha}
\bibliography{ollinewbib}

\end{document}